\documentclass[11pt]{amsart}
\usepackage{geometry}                % See geometry.eps to learn the layout options. There are lots.
\usepackage{graphicx}
\usepackage{amssymb}
\usepackage{epstopdf}
\usepackage{color}

\usepackage{amssymb, amsmath, amscd, amsthm,
color, epsfig
}
\usepackage[all]{xy}          % for xy-pic pictures
\xyoption{dvips}              % for xy-pic pictures

\title[Rational homotopy of spaces of long embeddings]{Graph-complexes computing the rational homotopy of high dimensional analogues of spaces of long knots}
\author{Gregory Arone}
\address{Department of Mathematics \\ University of Virginia \\ Charlottesville, VA 22904 \\ USA}
\email{zga2m@virginia.edu}
\thanks{Both authors gratefully acknowledge NSF support via collaborative grant DMS 0967649 and via Midwest Topology Network grant DMS 0844249.}
\author{Victor Turchin}
\address{Department of Mathematics\\ Kansas State University \\ Manhattan, KS 66506 \\ USA}
\email{turchin@ksu.edu}
\date{}                                           % Activate to display a given date or no date

\newcommand{\calC}{\mathcal{C}}

\newcommand{\calE}{\mathcal{E}}
\newcommand{\calO}{\mathcal{O}}

\newcommand{\calK}{\mathcal{K}}
\newcommand{\calP}{\mathcal{P}}

\newcommand{\balls}{\mathtt{B}}

\newcommand{\R}{{\mathbb R}}

\newcommand{\Q}{{\mathbb{Q}}}

\newcommand{\ext}{\operatorname{Ext}}
\newcommand{\CR}{\operatorname{cr}}
\newcommand{\holim}{\operatorname{holim}\,}

\newcommand{\Emb}{\operatorname{Emb}}
\newcommand{\Ebar}{\overline{\Emb}}

\newcommand{\Top}{\operatorname{Top}}
\newcommand{\hTop}{\operatorname{hTop}}

\newcommand{\Epi}{\Omega}

\newcommand{\hRmod}{\operatorname{hRmod}}

\newcommand{\chains}{\mathtt{C}_*}

\newcommand{\HH}{{\operatorname{H}}}

\newcommand{\HHHH}{\mathrm{HH}}

\newcommand{\hatH}{\hat{\HH}_*(\mathrm{C}(\bullet,\R^n),\Q)}
\newcommand{\hatP}{\Q\otimes\hat\pi_*(\mathrm{C}(\bullet,\R^n))}
\newcommand{\unit}{{1\!\!1}}
\newcommand{\id}{{\mathrm{id}}}
\newcommand{\sign}{{\mathrm{sign}}}

\usepackage{graphicx, psfrag,rotating}

\newcommand{\calEp}{{\mathcal E}^{m,n}_\pi}
\newcommand{\calEH}{{\mathcal E}^{m,n}_{\HH}}
\newcommand{\Ebarmn}{{\overline{\mathrm{Emb}}}_c(\R^m,\R^n)}
\newcommand{\Embmn}{\mathrm{Emb}_c(\R^m,\R^n)}
\newcommand{\Imn}{\mathrm{Imm}_c(\R^m,\R^n)}
\newcommand{\rank}{\operatorname{rank}}
\newcommand{\Rmod}{\operatorname{Rmod}}

\newcommand{\BN}{{\mathbb N}}

\newcommand{\BR}{{\mathbb R}}
\newcommand{\BZ}{{\mathbb Z}}
\newcommand{\BQ}{{\mathbb Q}}
\newcommand{\BS}{{\mathbb S}}
\newcommand{\Assoc}{{\mathcal A}ssoc}
\newcommand{\Poiss}{{\mathcal P}oiss}
\newcommand{\Lie}{{\mathcal L}ie}
\newcommand{\coLie}{co{\mathcal L}ie}
\newcommand{\Comm}{{\mathcal C}om}
\def\tr{{\rm tr}\,}

\newcommand\rth{\refstepcounter{equation}}
\newcommand\numb{\rth{\rm \theequation}}
\numberwithin{equation}{section}

\theoremstyle{plain}
\newtheorem{theorem}{Theorem}[section]
\newtheorem{proposition}[theorem]{Proposition}
\newtheorem{lemma}[theorem]{Lemma}
\newtheorem{corollary}[theorem]{Corollary}
\newtheorem{conjecture}[theorem]{Conjecture}

\theoremstyle{definition}

\theoremstyle{remark}
\newtheorem{remark}[theorem]{Remark}

\begin{document}

\maketitle

\sloppy

\begin{abstract}
We continue our investigation of spaces of long embeddings (long embeddings are high-dimensional analogues of long knots). In previous work we showed that when the dimensions are in the stable range, the rational {\it homology} groups of these spaces can be calculated as the homology of a direct sum of certain finite graph-complexes, which we described explicitly. In this paper, we establish a similar result for the rational {\it homotopy} groups of these spaces. We also put emphasis on different ways how the calculations can be done.  In particular we describe three different graph-complexes computing these  rational homotopy groups. We also compute the generating functions of the Euler characteristics of the summands in the homological  splitting.

\end{abstract}

\setcounter{tocdepth}{1}

\tableofcontents

\section{Introduction}\label{s:introduction}
\subsection{Overview}\label{ss:overview}
In this paper we continue the study of spaces of high-dimensional long knots that we began in~\cite{AT}. Let $\Embmn$ be the space of smooth embeddings $\R^m\hookrightarrow\R^n$ that coincide with a fixed linear embedding $i\colon\R^m\hookrightarrow\R^n$ outside a compact subset of $\R^m$. For $m=1$ this space is usually called the {\it space of long knots}, and for a general $m$ we will be calling it the {\it space of long embeddings}. Let $\Imn$ be the analogous space of immersions $\R^m\looparrowright\R^n$ with the same behavior at infinity.  The space $\Embmn$ is an open  subset of $\Imn$; in case $n\geq 2m+1$, this subset is dense. Denote by $\Ebarmn$ the homotopy fiber of the inclusion $\Embmn\hookrightarrow\Imn$ over the fixed linear embedding~$i$. The subject of this paper is the rational homotopy type of $\Embmn$ and $\Ebarmn$, when $n\geq 2m+2$.

%The restriction $n\geq 2m+2$ guarantees that all the spaces involved are connected, moreover they are $H$-spaces, which means that their rational homotopy type is completely determined by the ranks of the rational homotopy groups. In particular the rational cohomology algebra is a free polynomial algebra generated by the rational homotopy (or rather its dual vector space).
The case $m=1$ (which is the case of usual long knots) has been studied extensively,  see~\cite{ALTV,LT,LTV,Turchin}. Our goal is to achieve the same level of understanding of the rational homology and homotopy of these spaces in cases $m>1$. Our methods do apply to the case $m=1$, but in this case all our results are known.

Our point of departure is Theorem 0.2 of~\cite{AT} (first part of Theorem~\ref{t:hom_as_rmod} below). This theorem describes the rational homology groups $\HH_*(\Ebarmn,\Q)$ as the homology of a \lq\lq space of derived maps" of certain right $\Omega$-modules (right $\Omega$-modules are the same thing as right modules over the commutative operad without unit). In Subsection~\ref{ss:about_theorem} we briefly explain how  \cite[Theorem~0.2]{AT} was obtained. Using this theorem in Sections~\ref{s:graphs} and~\ref{s:koszul} we define two explicit complexes  computing the homology  $\HH_*(\Ebarmn,\Q)$. The complex from Section~\ref{s:graphs} is obtained by taking a fibrant replacement (in the injective model structure) for the target  $\Omega$-module and the  one from Section~\ref{s:koszul} is obtained by taking a cofibrant replacement (in the projective model structure) of the source $\Omega$-module. Actually the latter complex appeared already in~\cite{AT}. It is given here for completeness of exposition and also because it is used in computations of the Euler characteristics of the double splitting, see Section~\ref{s:Euler}. We also compare this latter complex with a certain deformation complex of a morphism of operads. In Section~\ref{s:graphs} we obtain that the rational homotopy $\Q\otimes\pi_*(\Ebarmn)$ can also be described as the homology of a space of derived maps between  certain $\Omega$-modules. This result is equivalent to saying that the spectral sequence computing $\Q\otimes\pi_*(\Ebarmn)$ and associated with the Goodwillie-Weiss tower collapses at the second term. Similarly we obtain two different complexes computing  $\Q\otimes\pi_*(\Ebarmn)$, see Section~\ref{s:graphs} and~\ref{s:koszul}. It is quite interesting that all the obtained complexes computing the rational homology and homotopy of $\Ebarmn$ look very similar to the graph-complexes arising in the Bott-Taubes integration for the space of long knots and their higher dimensional analogues~\cite{Catt,CattRossi,Sakai,SakW,Wat}. In the paper we also determine how the rational homotopy type of $\Embmn$ is related to that of $\Ebarmn$, see Section~\ref{s:emb}.

As it follows from \cite[Theorem~0.2]{AT} the rational homology $\HH_*(\Ebarmn,\Q)$, $n\geq 2m+2$, has a natural double splitting and the terms of the splitting depend only on the parities of $m$ and $n$ (but the homological degree in which these terms appear do depend on $m$ and $n$). In other words up to a certain regrading this homology is the same for spaces $\Ebarmn$ with $m$ and $n$ of the same parities. In particular if we know this splitting in the homology of $\overline{\mathrm{Emb}}_c(\R^1,\R^n)$, we can determine the homology of any space $\Ebarmn$ with $m$ odd. This {\it biperiodicity} is quite surprising. Notice it does not hold for the initial spaces of long embeddings  $\Embmn$ since it is false for $\Imn$. In Section~\ref{s:Euler} we produce generating functions of the Euler characteristics of this splitting. We mention that in the case $m=1$ this splitting was earlier considered by the second author in~\cite{Turchin}. Our computations are completely analogous. In Appendix we present results of computer calculations of the Euler characteristics of the splitting in small degrees both for the homology and homotopy.

\subsection{Rational homology and homotopy of $\Ebarmn$ in terms of maps between $\Omega$-modules}\label{ss:hom_as_omega_maps}
As we mentioned in the previous subsection, our starting point for this paper is \cite[Theorem~0.2]{AT}. In order to formulate this result we need to evoke some terminology. Let $\Omega$ be the category of finite (possibly empty) sets with morphisms - surjective mappings. Let $\Gamma$ be the category of finite pointed sets with morphisms - maps preserving the based point. For a small category $\calC$ we will call a right $\calC$-module a contravariant functor whose source category is $\calC$. Similarly a left $\calC$-module is a covariant functor with source  $\calC$.  The target category usually is going to be the category of vector spaces over~$\Q$ or the category of non-negatively graded chain complexes over~$\Q$. The category of right $\calC$-modules in $\Q$-vector spaces will be denoted $\mathrm{mod}{-}\calC$. In~\cite{PirashviliDold} Pirashvili constructs a functor
$$
\mathrm{cr}\colon \mathrm{mod}{-}\Gamma\,\longrightarrow\,\mathrm{mod}{-}\Omega
\eqno(\numb)\label{eq:cross_effect}
$$
which defines an equivalence of categories. The functor $\mathrm{cr}$ is defined in the following way. Let $F$ be a right $\Gamma$-module. Let $k$ denote the set $\{1,2,\ldots,k\}$ and $k_+$ denote the set $\{0,1,2,\ldots,k\}$, pointed at~0. The component $\mathrm{cr}F(k)$ is the quotient of $F(k_+)$ by the images of the maps
$$
\alpha_i^*\colon F\left((k-1)_+\right)\to F(k_+)
$$
induced by the maps $\alpha_i\colon k_+\to (k-1)_+$, $i=1\ldots k$, defined as
$$
\alpha_i(j)=
\begin{cases}
j,& j<i;\\
0,& j=i;\\
j-1,& j>i.
\end{cases}
$$
The $\Omega$-module $\mathrm{cr}F$ is called {\it cross-effect} of $F$.

For any $m\geq 1$ we can consider a contravariant functor
$$
S^{m\bullet}\colon\Omega\to\Top
$$
that sends a set $k$ to the sphere $S^{mk}$. On morphisms this functor is defined by means of the diagonal maps (here $S^{mk}$ is viewed as a one-point compactification of $\R^{mk}=\underbrace{\R^m\times\ldots\times\R^m}_{k}$). Composing with the reduced homology functor produces the graded right $\Omega$-module  $\widetilde{\HH}_*(S^{m\bullet},\Q)$. It is easy to see that this $\Omega$-module is the cross-effect of the $\Gamma$-module $\HH_*((S^m)^\bullet,\Q)$ that assigns to a pointed set $k_+$ the homology of the space $(S^m)^k$ of pointed maps $k_+\to S^m$.

It is not hard to see that for $n\geq 2$ the assignment
$$
k_+\rightsquigarrow \mathrm{Emb}_*(k_+,S^n),
$$
where $\mathrm{Emb}_*(k_+,S^n)$ is the space of pointed embeddings $k_+\hookrightarrow S^n$, defines a contravariant functor
$$
\mathrm{Emb}_*(\bullet,S^n)\colon\Gamma\to\hTop
$$
where $\hTop$ is the {\em homotopy} category of topological spaces (whose morphisms are homotopy classes of maps). Notice that $\mathrm{Emb}_*(k_+,S^n)$ is homeomorphic to the configuration space $\mathrm{C}(k,\R^n)$ of $k$ labeled points in $\R^n$. Composing with the homology functor one gets a right $\Gamma$-module (again with trivial differential) $\HH_*(\mathrm{Emb}_*(\bullet,S^n),\Q)$. In case $n\geq 3$ these configuration spaces are simply connected which allows us to define a $\Gamma$-module $\Q\otimes \pi_*(\mathrm{Emb}_*(\bullet,S^n))$. The cross-effect of these two functors will be denoted by $\hatH$ and $\hatP$ respectively.
The corresponding groups are sometimes called the normalized (rational) homology and homotopy of configuration spaces.

The differential non-negatively graded right $\Omega$-modules form a  category $\mathrm{Ch}_{\geq 0}\left(\mathrm{mod}{-}\Omega\right)$ with a natural {\it projective} model structure enriched over chain complexes. In this model structure weak equivalences are quasi-isomorphisms, fibrations are surjective maps in all strictly positive degrees, cofibrations are inclusions with degree-wise projective cokernels. For a pair of $\Omega$-modules $F$ and $G$, we will denote by $\underset{\Omega}{\hRmod}(F,G)$ the corresponding derived $\mathrm{hom}$ object, which is a chain complex.\footnote{One could choose instead to work with the category of unbounded complexes of $\Omega$-modules. The theorem below would still be true. But it just takes a little bit more work to define a model structure.}

\begin{theorem}\label{t:hom_as_rmod}
For $n\geq 2m+2$, one has natural isomorphisms
$$
(i)\quad \HH_*(\Ebarmn,\Q)\simeq \HH\left(\underset{\Omega}{\hRmod}\left(\widetilde{\HH}_*(S^{m\bullet},\Q),\hatH\right)\right);
\eqno(\numb)\label{eq:homol_as_rmod}
$$
$$
(ii)\quad \Q\otimes \pi_*(\Ebarmn)\simeq \HH\left(\underset{\Omega}{\hRmod}\left(\widetilde{\HH}_*(S^{m\bullet},\Q),\hatP\right)\right).
\eqno(\numb)\label{eq:homot_as_rmod}
$$
\end{theorem}

The first part of this theorem is the main result \cite[Theorem~0.2]{AT} of our previous paper. The second part of this theorem is proved in Section~\ref{s:graphs}. It is  a consequence of Theorems~\ref{t:HE=RH}-\ref{t:HE_Ext}.

Notice that the source and the target objects in~\eqref{eq:homol_as_rmod} and~\eqref{eq:homot_as_rmod} have trivial differential. This means that the right-hand sides are just products of certain $\ext$ groups in the abelian category $\mathrm{mod}{-}\Omega$ of  $\Omega$-modules in $\Q$-vector spaces.

Another way to understand the theorem above is that it describes the rational homology and homotopy of $\Ebarmn$ in terms of the higher order Hochschild homology defined by Pirashvili~\cite{PirashviliDold}.

\subsection{Theorem~\ref{t:hom_as_rmod} (i)}\label{ss:about_theorem}
To make the presentation self-contained we present below the main ideas behind the proof of Theorem~\ref{t:hom_as_rmod}~(i) (a.k.a. \cite[Theorem~0.2]{AT}). In short this result is a combination of the Goodwillie-Weiss manifold calculus of functors and a result about the relative formality of the little discs operads~\cite[Theorem~1.4]{LV}. By {\it standard discs} in $\R^m$ we will understand discs obtained from the unit disc by translations and rescaling. Let $\widetilde{\calO}_\infty^{st}(\R^m)$ denote the subcategory of open subsets of $\R^m$ which are finite unions of open {\it standard discs} union a complement of a closed standard disc (we call it the {\it anti-disc}). All the discs and the anti-disc are supposed to be disjoint.  We will view $\R^m$ as a subset of $\R^n$ via inclusion $i$. By a {\it standard embedding}  of a subset $X$ of $\R^m$ in $\R^n$ we will understand an embedding which is a composition of the inclusion, translation, and rescaling on each connected component of $X$. Consider the cofunctor
$$
\mathrm{Emb}^{st}_c(\bullet,\R^n)\colon\widetilde{\calO}_\infty^{st}(\R^m)\to\Top,
$$
that assigns to an open set $U$ the space of standard embeddings $U\hookrightarrow \R^n$. It follows from the Goodwillie-Weiss manifold calculus of functors~\cite{Weiss:HomologyEmb}, that in the range $n\geq 2m+2$ one has a quasi-isomorphism
$$
\chains(\Ebarmn)\simeq\underset{\widetilde{\calO}_\infty^{st}(\R^m)}{\holim}\chains(\mathrm{Emb}^{st}_c(\bullet,\R^n)),
$$
where $\chains(-)$ is the functor of singular chains and the homotopy limit is taken in the model category %$\mathrm{Ch}_{\geq 0}$
of non-negatively graded chain complexes of abelian groups, %Similarly we can consider the functor
%$$
%\mathrm{Emb}_c^{st}(\bullet,\R^n)\colon\widetilde{\calO}_\infty^{st}(\R^m)\to\Top
%$$
%that assigns to an open set $U\in\widetilde{\calO}_\infty^{st}(\R^m)$ the space of {\it standard} embeddings $V\hookrightarrow\R^n$, that is embeddings which are compositions of a rescaling and a translation on each disc and just the inclusion on the anti-disc of $U$.
%One has
%$$
%\chains(\Ebarmn)\simeq\underset{\widetilde{\calO}_\infty^{st}(\R^m)}{\holim}\chains(\mathrm{Emb}_c^{st}(\bullet,\R^n)),
%$$
see \cite[Lemmas~4.4 and~5.1]{AT}.

Let $\balls_m$ and $\balls_n$ be the operads of little $m$-discs and $n$-discs respectively. One can define a natural inclusion $i\colon\balls_m\hookrightarrow\balls_n$ induced by the fixed linear inclusion $i\colon\R^m\hookrightarrow\R^n$ (we are abusing notation by denoting both inclusions by $i$). It was shown in~\cite{LV} that the morphism of operads
$$
i_*\colon \chains^\R(\balls_m)\to \chains^\R(\balls_n)
$$
is a formal map of operads when $n\geq 2m$. This means that the morphism of operads $i_*$ is connected by a zig-zag of quasi-isomorphisms to the morphism $\HH_*(\balls_m,\R)\to\HH_*(\balls_n,\R)$. Using this result one can show that the functor $\chains^\R(\mathrm{Emb}_c^{st}(\bullet,\R^n))$ is formal, i.e. quasi-isomorphic to the functor
$\HH_*(\mathrm{Emb}_c^{st}(\bullet,\R^n),\R)$. One gets
$$
\chains^\R(\Ebarmn)\simeq\underset{\widetilde{\calO}_\infty^{st}(\R^m)}{\holim}\HH_*(\mathrm{Emb}_c^{st}(\bullet,\R^n),\R).
$$
Over a field any complex is quasi-isomorphic to its homology (viewed as a complex with zero differential). On the other hand, a tensor product of a $\Q$-complex with $\R$ preserves the dimensions of the homology groups, thus the above quasi-isomorphism also (non-naturally) holds over $\Q$:
$$
\chains^\Q(\Ebarmn)\simeq\underset{\widetilde{\calO}_\infty^{st}(\R^m)}{\holim}\HH_*(\mathrm{Emb}_c^{st}(\bullet,\R^n),\Q).
$$
We prefer to use the singular chains rather than homology for the left-hand side to stress the fact that the right-hand side in the above quasi-isomorphism is viewed as a complex.

It is easy to see that the functor
$$
\HH_*(\mathrm{Emb}_c^{st}(\bullet,\R^n),\Q)\colon \widetilde{\calO}_\infty^{st}(\R^m)\longrightarrow \mathrm{Ch}^\Q_{\geq 0}
$$
can be factored through the category $\Gamma$. Indeed, the above functor is the composition of the functor $\HH_*(\mathrm{Emb}_*(\bullet,S^n),\Q)$ considered in Subsection~\ref{ss:hom_as_omega_maps} and the functor
$$
\pi_0\colon \widetilde{\calO}_\infty^{st}(\R^m)\longrightarrow \Gamma,
$$
that assigns the set of connected components based in the anti-disc. In case $m=1$ the anti-disc has two connected components, thus we modify a little bit $\pi_0$ so that it assigns only one point to an anti-disc.

Intuitively the Goodwillie-Weiss manifold calculus is a machinery that scans a manifold (in our case $\R^m$ or rather its one-point compactification $S^m$)  with finitely many discs, evaluate the functor on these collections of discs, and then extrapolate the functor from these data on the entire manifold. The fact that the functor
$\HH_*(\mathrm{Emb}_c^{st}(\bullet,\R^n),\Q)$ factors through $\Gamma$ means that instead of scanning our manifold $S^m$ with finitely many  discs we can scan it with finitely many points and extrapolate from there. More precisely in~\cite{AT} we proved the following:
$$
\chains^\Q(\Ebarmn)\simeq\underset{\Gamma}{\hRmod}\left(\chains^\Q((S^m)^\bullet),\HH_*(\mathrm{Emb}_*(\bullet,S^n),\Q)\right),
\eqno(\numb)\label{eq:sing_ch_as_omega_maps}
$$
see \cite[Proposition~6.3]{AT}.

It turns out that the $\Gamma$-module  $\chains^\Q((S^m)^\bullet)$ is formal, see \cite[Lemma~6.5]{AT}. %We actually prove in that paper that its cross-effect is formal, but obviously it is an equivalent statement.
Thus we can replace $\chains^\Q((S^m)^\bullet)$ in~\eqref{eq:sing_ch_as_omega_maps} by $\HH_*((S^m)^\bullet,\Q)$.\footnote{The formality of $\chains^\Q((S^m)^\bullet)$ is in fact the main reason for the Hodge splitting in the higher order Hochschild homology~\cite{PirashviliDold} for which our case is a particular example. This splitting (by Hodge degree $s$)  is discussed in the next subsection.}    Applying Pirashvili's cross-effect functor~\eqref{eq:cross_effect} we obtain exactly the statement of \cite[Theorem~0.2]{AT}. The reason we use $\Omega$-modules instead of $\Gamma$-modules is that working with $\Omega$-modules considerably reduces computations.

\subsection{Double splitting}\label{ss:dbl_split}
It turns out that in the rational homology and homotopy of $\Ebarmn$ one can introduce two additional gradings that we consider below. One obviously has
$$
\widetilde{\HH}_*(S^{m\bullet},\Q)=\bigoplus_{s=0}^{+\infty}\widetilde{\HH}_{ms}(S^{m\bullet});
$$
$$
\hatH=\prod_{t=0}^{+\infty}\hat\HH_{t(n-1)}(\mathrm{C}(\bullet,\R^n),\Q);
$$
$$
\hatP=\prod_{t=0}^{+\infty}\Q\otimes\hat\pi_{t(n-2)+1}(\mathrm{C}(\bullet,\R^n)).
$$
In the above each summand/factor is viewed as a chain complex concentrated in a single homological degree. Obviously in the above the infinite direct sum can be replaced by a direct product and vice versa. We need the sum for the first decomposition and the product for the second and third ones to obtain the following   factorizations:
$$
\underset{\Omega}{\hRmod}\left(\widetilde{\HH}_*(S^{m\bullet},\Q),\hatH\right)=
\prod_{s,t}\underset{\Omega}{\hRmod}\left(\widetilde{\HH}_{ms}(S^{m\bullet},\Q),\hat\HH_{t(n-1)}(\mathrm{C}(\bullet,\R^n),\Q)\right);
\eqno(\numb)\label{eq:homol_split}
$$
$$
\underset{\Omega}{\hRmod}\left(\widetilde{\HH}_*(S^{m\bullet},\Q),\hatP\right)=
\prod_{s,t}\underset{\Omega}{\hRmod}\left(\widetilde{\HH}_{ms}(S^{m\bullet},\Q),\Q\otimes\hat\pi_{t(n-2)+1}(\mathrm{C}(\bullet,\R^n))\right).
\eqno(\numb)\label{eq:homot_split}
$$
It is not hard to show that when the dimensions are in the stable range, that is when $n\ge 2m+2$, the product in~\eqref{eq:homol_split} can be replaced by a direct sum, see~\cite{AT}. Again this is equivalent to saying that only finitely many factors are non-trivial in any given homological degree. We show in Section~\ref{s:koszul} that the same is true for the second equation~\eqref{eq:homot_split}. Thus Theorem~\ref{t:hom_as_rmod} naturally defines a double splitting in the rational homology and homotopy of $\Ebarmn$:
$$
\HH_*(\Ebarmn,\Q)\cong
\bigoplus_{s,t}\underset{\Omega}{\hRmod}\left(\widetilde{\HH}_{ms}(S^{m\bullet},\Q),\hat\HH_{t(n-1)}(\mathrm{C}(\bullet,\R^n),\Q)\right);
\eqno(\numb)\label{eq:homol_split2}
$$
$$
\Q\otimes\pi_*(\Ebarmn)\cong\bigoplus_{s,t}\underset{\Omega}{\hRmod}\left(\widetilde{\HH}_{ms}(S^{m\bullet},\Q),\Q\otimes\hat\pi_{t(n-2)+1}(\mathrm{C}(\bullet,\R^n))\right).
\eqno(\numb)\label{eq:homot_split2}
$$
The additional grading $t$ will be called {\it complexity}. It is related with the Vassiliev filtration in the homology by complexity of strata in the discriminant~\cite{Vassiliev}. The grading $s$ will be called {\it Hodge degree}. It comes from the Hodge type decomposition in the Hochschild cohomology of a commutative algebra and more generally in the (higher order) homology of $\Gamma$-modules~\cite{GS,Loday,PirashviliDold}.

%{\color{blue}{QUESTION TO GREG: Is this isomorphism~\eqref{eq:homol_split2} canonical or it is defined up to some gauge transformations? In other words is it really a splitting or a sort of bifiltration? It would be nice to mention that in the paper.}}

We  expect that the bialgebra structure of $\HH_*(\Ebarmn,\Q)$ does not respect this splitting as a bigrading (we know this is false for $m=1$), but it must respect the splitting as a bifiltration. And the ranks of the terms of the splitting still behave in a way as if the splitting was respected as a bigrading. This follows from the fact that the graph-complex $\calE_\HH^{m,n}$ computing $\HH_*(\Ebarmn,\Q)$ is naturally a polynomial bialgebra generated by the graph-complex $\calE_\pi^{m,n}$ computing $\Q\otimes\pi_*(\Ebarmn)$, see Section~\ref{s:graphs}. This fact is used in Lemma~\ref{l:euler_homot}.

\subsection{Operads}\label{ss:operads}
In recent years it became clear that manifold calculus is deeply related to the theory of operads~\cite{Sinha-OKS, ALV,AT}. This paper continues this tradition. We will now review briefly some relevant notions about operads. A good introduction to the theory of operads can be found in~\cite{LodayVal}.

\subsubsection{Operads as monoids}\label{sss:operads_monoids}
Let $(\calC,\otimes,\unit)$ be a cocomplete symmetric monoidal category, where $\otimes$ distributes over colimits, and let $\BS$ be the category of finite sets with morphisms bijective maps. A {\it right $\BS$-module} with values in $\calC$ is a contravariant functor $M\colon\BS\to\calC$. This functor can be viewed as a sequence of objects $\{M(n),\, n\geq 0\}$ with a  right action of $\Sigma_n$ on each $M(n)$. We view elements of $M(n)$ as something that have $n$ inputs and one output. The symmetric group action permutes the inputs. With a right $\BS$-module one can assign a {\it power series  functor} %{\bf QUESTION FROM GREG: DO WE REALLY WANT TO CALL IT POLYNOMIAL IF M(n) IS NON ZERO FOR INFINITELY MANY n? IT IS MORE LIKE "POWER SERIES FUNCTOR"}
$F_M\colon \calC\to\calC$, that is a functor of the form
$$
F_M(V)=\bigoplus_nM(n)\otimes_{\Sigma_n}V^{\otimes n}.
$$
It turns out that a composition of two power series functors is again a power series functor. And we define composition $M\circ N$ of two $\BS$-modules $M$ and $N$ in such a way that $F_{M\circ N}=F_M\circ F_N$. For an explicit definition of $M\circ N$, see~\cite[Chapter~5]{LodayVal}. Intuitively elements of $M\circ N$ are objects of the form:

\begin{center}
\psfrag{m}[0][0][1][0]{$m$}
\psfrag{n1}[0][0][1][0]{$n_1$}
\psfrag{n2}[0][0][1][0]{$n_2$}
\psfrag{nk}[0][0][1][0]{$n_k$}
\includegraphics[width=4.5cm]{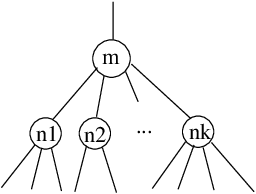}
\end{center}

In each input of some element of $M$ we insert elements of $N$. The operation $\circ$ turns the category of right $\BS$-modules into a monoidal category with unit $\id$:
$$
\id(n)=
\begin{cases}
0,& n\neq 1;\\
\unit,& n=1,
\end{cases}
$$
where $0$ is the initial object of $\calC$. The symmetric group action is trivial for all components of $\id$. An {\it operad} is  a monoid in the category of $\BS$-modules: it is an $\BS$-module $\calP$ endowed with morphisms $\calP\circ\calP\to\calP$, and $\id\to\calP$ which satisfy natural associativity and unit axioms. {\it Left} and {\it right modules} over an operad are defined in an expected way.

Let now $\calC$ be the category of chain complexes over $\Q$. Let $\Sigma$ be a suspension functor that shifts a complex in degree by $+1$. An {\it operadic suspension} is an endofunctor of the category of right $\BS$-modules, that assigns to an $\BS$-module $M$ an $\BS$-module $M[1]$ defined as
$$
M[1](n)=\Sigma^{n-1}M\otimes \sign_n,
$$
where $\sign_n$ is the sign representation of $\Sigma_n$. One can view this operation as a suspension of each input and desuspension of the output for every element of $M$. One can also define $M[1]$ as a symmetric sequence such that $F_{M[1]}=\Sigma^{-1}\circ F_{M}\circ\Sigma$. It is easy to see that the operadic suspension respects the monoidal structure:
$$
\id[1]=\id;
$$
$$
(M\circ N)[1]=(M[1])\circ (N[1]).
$$
Thus an operadic suspension of an operad is again an operad.

\subsubsection{Operads $\Comm$, $\Lie$, $\Assoc$, $\Poiss$, $\balls_n$, $\HH_*(\balls_n)$}\label{sss:operads_examples}
The main classical operads we make use of are the operads $\Comm$ of commutative unital algebras, $\Comm_+$ of commutative non-unital algebras, $\Lie$ of Lie algebras, $\Assoc$ of associative unital algebras, $\Poiss$ of Poisson algebras. We will also need the topological operad $\balls_n$ of little $n$-discs. The latter operad is central for the manifold calculus of functors. It was shown by F.~Cohen~\cite{Coh} that the homology operad $\HH_*(\balls_n)$ is the associative operad $\Assoc$ if $n=1$ and is the operad $\Poiss_{n-1}$ of graded Poisson algebras with a commutative product of degree zero and a Lie bracket of degree $(n-1)$ in case $n\geq 2$, see also~\cite{GetzJon,Sinha-LDO}. For $n\geq 2$, one has $\HH_0(\balls_n)$ is the operad $\Comm$, and ${\HH}_{(\bullet-1)(n-1)}(\balls_n(\bullet))$ is the operad $\Lie[n-1]$ -- the $(n-1)$-fold operadic suspension of the operad $\Lie$.

Each component $\balls_n(k)$ is homotopy equivalent to the configuration space $\mathrm{C}(k,\R^n)$ whose cohomology algebra is generated by the elements $\alpha_{ij}$, $1\leq i\neq j\leq k$, of degree $(n-1)$. The relations are
$$
\left\{
\begin{array}{l}
\alpha_{ij}=(-1)^n\alpha_{ji},\quad \alpha_{ij}^2=0;\\
\alpha_{ij}\alpha_{jk}+\alpha_{jk}\alpha_{ki}+\alpha_{ki}\alpha_{ij}=0,
\end{array}
\right.
$$
see~\cite{Arn,Coh}. The last relation is often called {\it Arnol'd relation}. To every monomial of this algebra one can assign a directed graph putting an edge from vertex $i$ to $j$ for every factor $\alpha_{ij}$. Using the above relations one can show that a  monomial is non-zero if and only if the corresponding graph is a forest. Thus  $\HH^*(\balls_n(k))$ can be described as a certain space of forests on $k$ vertices modulo orientation and Arnol'd relations. Notice that the top cohomology group ${\HH}^{(k-1)(n-1)}(\balls_n(k))$ is the space of trees (forests with exactly one component). They form an $\BS$-module ${\HH}^{(\bullet-1)(n-1)}(\balls_n(\bullet))$ which is naturally isomorphic to the cooperad $\coLie[n-1]$ dual to the operad $\Lie[n-1]$.  We refer to~\cite{Sinha:GD} where the reader can find how exactly this duality works and how exactly the cooperad structure on $\HH{}^*(\balls_n(\bullet))$ looks.

\subsubsection{Infinitesimal bimodules}\label{sss:inf_bimod}
It is almost straightforward to see that the structure of a right $\Omega$-module is equivalent to the structure of a right module over the operad $\Comm_+$ of commutative  algebras without unit. The structure of a right $\Gamma$-module is also intimately related to the commutative operad, more precisely we have seen in~\cite{Turchin,AT} that a right $\Gamma$-module is the same thing as a {\it weak bimodule} over $\Comm$. In this paper we adopt the terminology of Merkulov-Vallette-Loday~\cite{MerkVal,LodayVal} and call them {\it infinitesimal bimodules}. This term is more appropriate here since we are using this notion only for the abelian category of chain complexes. Let us recall that an {\it infinitesimal} (or {\it weak}) bimodule over an operad $\calO$ is a right $\BS$-module $M$ endowed with a family of maps:
$$
\overline{\circ}_i\colon \calO(n)\otimes M(k)\to M(n+k-1),\,\, i=1\ldots n, \text{ (left action)};
\eqno(\numb)\label{eq_left_action}
$$
$$
\underline{\circ}_i\colon M(k)\otimes \calO(n)\to M(k+n-1),\,\, i=1\ldots k, \text{ (right action)},
\eqno(\numb)\label{eq_right_action}
$$
satisfying natural unity, associativity, and compatibility with the $\Sigma_n$-group action conditions, see~\cite{MerkVal,Turchin,AT}.
 Each element of $\calO$ and of $M$ is viewed as an object with some number of inputs and one output. The composition is obtained by inserting an output in one of the inputs, see Figure~\ref{fig1} below.
The result of composition $\overline{\circ}_i(o,m)$, and $\underline{\circ}_i(m,o)$, for $o\in\calO(n)$, and $m\in M(k)$, will be denoted by $o\circ_i m$, and $m\circ_i o$.

\vspace{.3cm}

\begin{figure}[h]
\psfrag{OM}[0][0][1][0]{$o\circ_3 m$}
\psfrag{MO}[0][0][1][0]{$m\circ_2 o$}
\psfrag{o}[0][0][1][0]{$o$}
\psfrag{m}[0][0][1][0]{$m$}
\includegraphics[width=14cm]{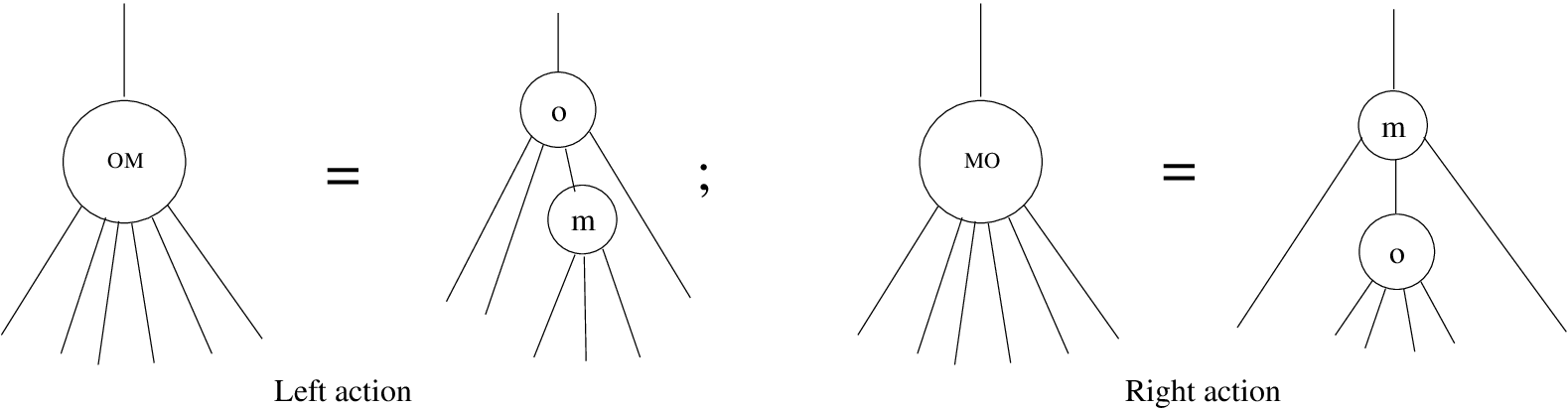}
\caption{}\label{fig1}
\end{figure}

For example, if $\calP$ is an operad endowed with a morphism $\calO\to\calP$, then $\calP$ is naturally an infinitesimal bimodule over $\calO$. Applying this construction to the map
$$
\HH_0(\balls_n(\bullet),\Q)\to\HH_*(\balls_n(\bullet),\Q),\quad n\geq 2,
$$
we obtain that $\HH_*(\balls_n(\bullet),\Q)$ is automatically an infinitesimal  bimodule over $\Comm$, thus $\HH_*(\balls_n(\bullet),\Q)$ is naturally a right $\Gamma$-module. One can easily see that this $\Gamma$-module is exactly $\HH_*(\mathrm{Emb}_*(\bullet,S^n),\Q)$ considered in Subsection~\ref{ss:hom_as_omega_maps}. We will use both $\hat\HH_*(\mathrm{C}(\bullet,\R^n),\Q)$ and $\hat\HH_*(\balls_n(\bullet),\Q)$ to denote its cross-effect $\Omega$-module. Notice that its dual left $\Omega$-module $\hat\HH{}^*(\balls_n(\bullet),\Q)$ in degree $k$ is the subspace of $\HH^*(\balls_n(k),\Q)$ spanned by forests with the property that all connected components have at least two vertices.

\subsection{A section by section outline}\label{ss:outline}
%In Section~\ref{s:dbl_split} we define a double splitting in the rational homology and homotopy of $\Ebarmn$, $n\geq 2m+4$.
In Section~\ref{s:graphs} we construct a graph-complex $\calE_\HH^{m,n}$ computing the rational homology $\HH_*(\Ebarmn,\Q)$.
 This construction is obtained by replacing the target $\Omega$-module $\hatH$ in~\eqref{eq:homol_as_rmod} by a quasi-isomorphic complex of injective $\Omega$-modules.
 Similarly we construct a graph-complex $\calE_\pi^{m,n}$ computing the right-hand side of~\eqref{eq:homot_as_rmod}. It turns out that $\calE_\HH^{m,n}$ is a
 polynomial bialgebra whose subcomplex of primitives is its subcomplex of connected graphs which is exactly $\calE_\pi^{m,n}$. This immediately implies
 Theorem~\ref{t:hom_as_rmod}~(ii). In Section~\ref{ss:small_complex} we  compute the rational homotopy of $\Ebarmn$ in small dimensions. In Section~\ref{s:emb} we compare the rational homotopy of $\Ebarmn$ with that of $\Embmn$ assuming as usual $n\geq 2m+2$. By the Smale-Hirsch theorem~\cite{Hirsch} the space $\Imn$, $n\geq 2m$, is weakly equivalent to the $m$-fold loop space $\Omega^m\mathrm{Inj}(\R^m,\R^n)$, where $\mathrm{Inj}(\R^m,\R^n)$ is the Stiefel manifold of isometric linear injections $\R^m\hookrightarrow\R^n$. One has a fibration
$$
\Omega^{m+1}\mathrm{Inj}(\R^m,\R^n)\stackrel\iota\longrightarrow\Ebarmn\longrightarrow\Embmn.
$$
Since the rational homotopy of the Stiefel manifold $\mathrm{Inj}(\R^m,\R^n)$ is finite-dimensional, it follows that up to a finite-dimensional correction the rational homotopy of $\Ebarmn$ is the same as that of $\Embmn$. In Section~\ref{s:emb} we determine this correction. It is again quite surprising that the image of the induced map $\iota_*$ in rational homotopy depends on the parities of $m$ and $n$ only.

In Section~\ref{s:koszul} we construct an explicit cofibrant replacement (in the projective model structure) of the right $\Omega$-module $\widetilde{\HH}_*(S^{m\bullet},\Q)$. This allows us to construct another type of complexes computing $\HH_*(\Ebarmn,\Q)$ and $\Q\otimes\pi_*(\Ebarmn)$. The corresponding complexes are denoted $\calK_\HH^{m,n}$ and $\calK_\pi^{m,n}$. We are calling them Koszul complexes, since  the main ingredient in the construction of the cofibrant replacement is the Koszul duality between the commutative and Lie operads.\footnote{This construction can actually be used to compute the   Hochschild-Pirashvili homology~\cite{PirashviliDold} in any general situation.} In Subsection~\ref{ss:bicol_graphs} we describe the complex $\HHHH^{m,n}$ dual to $\calK^{m,n}_\HH$ as a certain complex of graphs whose edges can have two colors. Its subcomplex $\HHHH^{m,n}_\pi$ of connected graphs is our third complex computing $\Q\otimes\pi_*(\Ebarmn)$. In Subsection~\ref{ss:operadic_interp} we interpret $\calK_\HH^{m,n}$ as the deformation complex of the morphism of operads
$$
\HH_*(\balls_m,\Q)\stackrel{i_*}\longrightarrow\HH_*(\balls_n,\Q),
$$
where $i\colon \balls_m\hookrightarrow \balls_n$ is the inclusion of the operad of little $m$-discs into the operad of  little $n$-discs (induced by our fixed linear embedding $\R^m\hookrightarrow\R^n$). This gives a connection between the homology of a certain deformation complex of a morphism of operads with the homology of the space of long embeddings --- a connection earlier conjectured by Kontsevich.

In Section~\ref{s:Euler} we compute the generating function of the Euler characteristics of the double splitting in $\HH_*(\Ebarmn,\Q)$. In Appendix we present results of computer calculations of these Euler characteristics in small dimensions for the splitting both in homology and in homotopy.

\section{Graph-complexes}\label{s:graphs}

\renewcommand{\thetable}{\Alph{table}}

\subsection{Complexes of uni-$\geq 3$-valent graphs}\label{ss:u3_graphs}
In this section we  introduce  complexes $\calEH$, $\calEp$ calculating the rational homology and rational homotopy of $\Ebarmn$. %The complex $\calEH$ is quasi-isomorphic to the direct sum of dual complexes of $\HH\HH^{m,n}_{s,t}$ defined previously, but the two complexes look different and are derived in a different way.
The starting point for the derivation of the complex $\calEH$ is Theorem~\ref{t:hom_as_rmod}~(i), which presents
$\chains^\Q(\Ebarmn)$ as the derived ``space'' of maps between right $\Epi$-modules $\widetilde\HH_*(S^{m\bullet},\Q)$ and $ \hat\HH_*(\balls_n(\bullet);\Q)$.  The complex $\calEH$ will be obtained by taking an injective resolution of  $ \hat\HH_*(\balls_n(\bullet);\Q)$.

The main attraction of the complex $\calEH$ is that it will enable us to construct another complex, denoted $\calEp$, which calculates the rational {\it homotopy} groups of the space $\Ebarmn$. In particular, this will enable us to prove the homotopical part of Theorem~\ref{t:hom_as_rmod}.

The complex $\calEp$ is defined as a {\it  complex of  connected uni-$\geq 3$-valent graphs}. We mention that such graph-complexes appeared earlier in the study of the Hodge decomposition of the homology groups of the space of long knots $\overline{\mathrm{Emb}}_c(\R,\R^n)$, see~\cite[Section~11]{Turchin}. Our Theorem~\ref{t:HE=RH} below is exactly~\cite[Conjecture~11.1]{Turchin} from the above reference. The construction of $\calEp$ was inspired by a work of Bar-Natan~\cite{BarNatan} where he studies the bialgebra of chord diagrams -- an object that combinatorially encodes finite type invariants of classical knots in $\R^3$. He shows that the space of primitive elements of this bialgebra is naturally isomorphic to a certain {\it space of  uni-trivalent graphs} quotiented out by some orientation and $IHX$ relations. One can easily see that this space is precisely the degree zero homology of our complex $\calE^{1,3}_\pi$ that we define below.

Let us define the complex $\calEp$. It is spanned by abstract connected graphs having a non-empty set of non-labeled {\it external vertices} of valence~1, and a possibly empty set of non-labeled {\it internal vertices} of valence~$\geq 3$. The graphs are allowed to have loops (edges joining a vertex to itself) and multiple edges. For such graph define its {\it orientation set} as the union of the set of its external vertices (considered as elements of degree $-m$), the set of its internal vertices (considered as elements of degree $-n$), and the set of its edges (considered as elements of degree $(n-1)$). By an {\it orientation} of a graph we will understand ordering of its orientation set together with an orientation of all its edges. Two such graphs are {\it equivalent} if there is a bijection between their sets of vertices and edges respecting the adjacency structure of the graphs, orientation of the edges, and the order of the orientation sets. The space of $\calEp$ is the quotient space of the vector space freely spanned by such graphs modulo the orientation relations:

\vspace{.2cm}

(1) $\Upsilon_1=(-1)^n\Upsilon_2$ if $\Upsilon_1$ differs from $\Upsilon_2$ by an orientation of an edge.

(2) $\Upsilon_1=\pm\Upsilon_2$, where $\Upsilon_2$ is obtained from $\Upsilon_1$ by a permutation of the orientation set. The sign here is the Koszul sign of permutation taking into account the degrees of the elements.

\vspace{.2cm}

The differential $\partial\Upsilon$ of a graph $\Upsilon\in\calEp$ is defined as the sum of expansions of its internal vertices. An expanded vertex is  replaced  by an edge. The set of edges adjacent to the  expanded vertex splits into two sets -- one containing the edges that go to one vertex of the new edge and the other set containing the edges that go to the other vertex.  An expansion of a vertex of valence $\ell$ is  a sum of $\frac{2^\ell-2\ell-2}{2}=2^{\ell-1}-\ell-1$ graphs obtained in such way. One subtracts  $2\ell+2$ to exclude graphs with internal vertices of valence $<3$, and one divides by 2 because of the symmetry. The orientation set of a new graph is obtained by adding the new vertex and the new edge as the first and second elements to the orientation set, and by orienting the new edge from the old vertex to the new one. There is a freedom which of 2 vertices of the new edge is considered as a new one and which as an old one, but regardless of this choice, the orientation of the boundary graph is the same. All the graphs in the differential appear with positive sign (the sign is hidden in the way we order the orientation set and orient the new edge).

Finally define the graph-complex $\calEH$ as the free polynomial bialgebra generated by $\calEp$. In other words $\calEH$ can be viewed as a graph-complex spanned by possibly empty or disconnected graphs with each connected component from $\calEp$.

In addition to the total grading (which is the sum of the degrees of the elements in the orientation set of a graph), we define two other gradings:
{\it complexity} --- the first Betti number of the graph obtained from initial graph by gluing together all univalent vertices, and {\it Hodge degree} --- the number of external vertices.
Notice that the differential preserves both the complexity and the Hodge degree. Section~\ref{ss:small_complex} describes $\calEp$ in complexities~$\leq 3$. We will denote the part of $\calEH$, $\calEp$ concentrated in complexity $t$ and Hodge degree $s$ by  $\calEH(s,t)$, $\calEp(s,t)$ respectively.

\begin{theorem}\label{t:HE=RH}
For $n\geq 2m+2$, the homology of the graph-complex $\calEH$ (respectively $\calEp$) is isomorphic to the rational homology (respectively homotopy) of $\Ebarmn$:
$$
\HH(\calEH)\simeq \HH_*(\Ebarmn,\Q);
\eqno(\numb)\label{eq:HE=RH}
$$
$$
\HH(\calEp)\simeq \pi_*(\Ebarmn)\otimes \Q.
\eqno(\numb)\label{eq:HE=Rp}
$$
\end{theorem}

It is well-known that the space $\Ebarmn$, $n\geq 2m+2$, is a connected $(m+1)$-loop space~\cite{Budney,Turchin}. This implies that its rational homology is a graded polynomial bialgebra generated by its rational homotopy groups. As a consequence the statements~\eqref{eq:HE=RH} and~\eqref{eq:HE=Rp} of the previous theorem are equivalent. The statement~\eqref{eq:HE=RH} follows from  Theorem~\ref{t:hom_as_rmod}~(i)  and  the statement~\eqref{eq:HEH_Ext} of Theorem~\ref{t:HE_Ext} below.

\begin{theorem}\label{t:HE_Ext}
For $n\geq 2m+2$, one has weak equivalences of chain complexes
$$
\calEH\simeq \underset{\Omega}{\hRmod}\left(\widetilde{\HH}_*(S^{m\bullet},\Q),\hat\HH_*(\balls_n(\bullet),\Q)\right).
\eqno(\numb)\label{eq:HEH_Ext}
$$
$$
\calEp\simeq \underset{\Omega}{\hRmod}\left(\widetilde{\HH}_*(S^{m\bullet},\Q),\hat\pi_*(\balls_n(\bullet))\otimes\Q\right),
\eqno(\numb)\label{eq:HEp_Ext}
$$
Moreover the above isomorphisms  preserve both the complexity and the Hodge degree, which means
$$
\calEH(s,t)\simeq \underset{\Omega}{\hRmod}\left(\widetilde{\HH}_{ms}(S^{m\bullet}),\hat\HH_{(n-1)t}(\balls_n(\bullet),\Q)\right).
\eqno(\numb)\label{eq:HEH_Ext2}
$$
$$
\calEp(s,t)\simeq \underset{\Omega}{\hRmod}\left(\widetilde{\HH}_{ms}(S^{m\bullet}),\hat\pi_{1+(n-2)t}(\balls_n(\bullet))\otimes\Q\right),
\eqno(\numb)\label{eq:HEp_Ext2}
$$
\end{theorem}

This theorem will be proved in Subsection~\ref{ss:reduced_graphs}. The idea of the proof is to replace the  right $\Omega$-modules
$\hat\pi_*(\balls_n(\bullet))\otimes\Q$ and $\hat \HH_*(\balls_n(\bullet),\Q)$ by quasi-isomorphic differential graded right $\Omega$-modules ${\hat P_n}{}^\bullet$, ${\hat D_n}{}^\bullet$, which happen to be injective in each homological degree. All their components  ${\hat P_n}{}^k$, ${\hat D_n}{}^k$, $k\geq 0$, are certain graph-complexes, see Subsections~\ref{ss:DnPn}-\ref{ss:reduced_graphs}.

Notice that the second statement of Theorem~\ref{t:HE=RH} together with the statement~\eqref{eq:HEp_Ext} of Theorem~\ref{t:HE_Ext} imply Theorem~\ref{t:hom_as_rmod}~(ii).

\subsubsection{Right $\Gamma$-modules ${D_n}^\bullet$, ${P_n}^\bullet$}\label{ss:DnPn}
The right $\Gamma$-modules that we define in this section were introduced  in~\cite{Turchin}.   The $k$-th component ${D_n}^k$  is a vector space spanned by the graphs with $k$ external vertices labeled by $1,2,\ldots,k$, and a bunch of non-labeled internal vertices. The external vertex can have any valence (including zero), the internal vertices are of valence $\geq 3$. The graphs are allowed to be disconnected, but each connected component of  a graph should have at least one external vertex. The graphs can have loops and multiple edges. {\it Orientation set} of such graph consists of the set of its internal vertices (having degree $-n$), and edges (having degree $(n-1)$). By an {\it orientation} of a graph we understand an ordering of its orientation set and a choice of orientation made for each one of its edges.  Two graphs are equivalent if there is a bijection between their sets of internal vertices and edges respecting the adjacency structure of the graphs, orientation of the edges, and the order of their orientation sets. The orientation relations are the same as in Subsection~\ref{ss:u3_graphs}. The differential is the sum of expansions of vertices. An expansion of an external vertex produces one external vertex with the same label and one internal one. An expansion of external vertices of valence $\ell$ is a sum of $2^\ell-\ell-1$ graphs. We excluded $\ell+1$ cases to make sure that the new internal vertex has valence $\geq 3$. The sequence of differential graded vector spaces ${D_n}^k,\, k\geq 0$, forms an operad. The composition $\Upsilon_1\circ_i\Upsilon_2$ of two graphs is defined as insertion of $\Upsilon_2$ into the $i$-th vertex of~$\Upsilon_1$, see Figure~\ref{fig:graph_compos}. The orientation set of each graph in the sum is obtained by concatenation of the orientation set of~$\Upsilon_1$ and that of~$\Upsilon_2$. With this definition all signs are positive in this figure.

%\vspace{0.3cm}

\begin{figure}[h]
\psfrag{c2}[0][0][1][0]{$\circ_2$}
\psfrag{c3}[0][0][1][0]{$\circ_3$}
\psfrag{p}[0][0][1][0]{$\pm$}
\includegraphics[width=14cm]{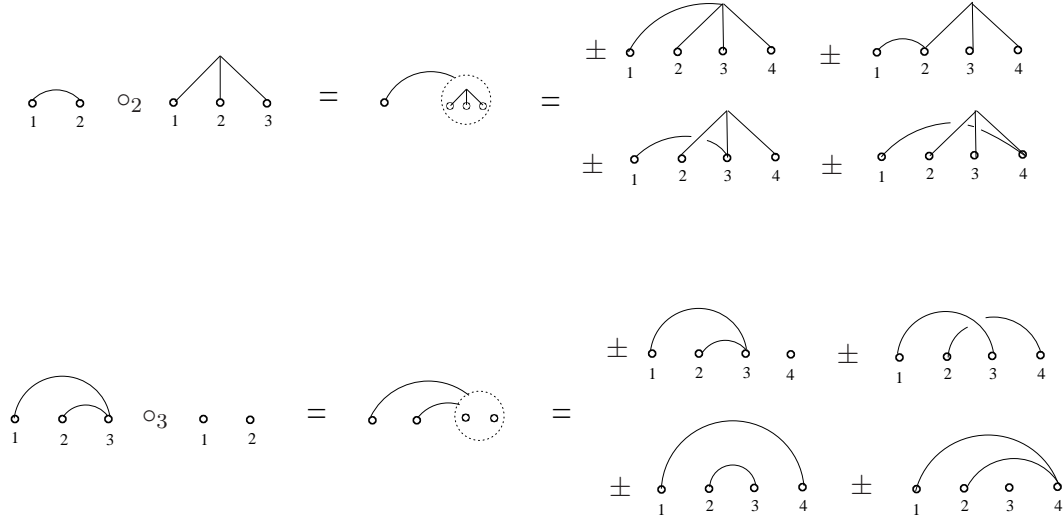}
\caption{Examples of composition}\label{fig:graph_compos}
\end{figure}

Recall that $\{\HH_*(\balls_n(\bullet),\Q)\}$ is the homology operad of $n$-dimensional little cubes. This homology operad is the operad of $(n-1)$-Poisson algebras, i.e. operad of graded Poisson algebras with a commutative product of degree zero and a Lie bracket of degree $(n-1)$.

%\begin{proposition}[ \cite[Theorem~9.1]{LV}, \cite[Proposition~9.2]{Turchin} ]\label{p:operad_inclusion}
\begin{proposition}[ \cite{LV,LT,Turchin} ]\label{p:operad_inclusion}
For $n\geq 2$, the assignment
$$
x_1x_2\mapsto \,
\begin{picture}(20,15)
\put(2,5){\circle*{3}}
\put(18,5){\circle*{3}}
\put(0,-1){$_1$}
\put(16,-1){$_2$}
\end{picture}\, , \qquad [x_1,x_2]\mapsto \,
\begin{picture}(20,15)
\put(2,5){\circle*{3}}
\put(18,5){\circle*{3}}
\qbezier(2,5)(10,15)(18,5)
\put(0,-1){$_1$}
\put(16,-1){$_2$}
\end{picture}\, ,
$$
where $x_1x_2,$ $[x_1,x_2]\in \HH_*(\balls_n(2),\Q)$ are the product and the bracket of the operad  of $(n-1)$-Poisson algebras, defines an inclusion of operads
$$
\xymatrix{
\HH_*(\balls_n(\bullet);\Q)\,\ar@{^{(}->}[r]^-\simeq&{D_n}^\bullet
}
\eqno(\numb)\label{eq:operad_inclusion}
$$
that turns out to be a quasi-isomorphism ($\HH_*(\balls_n(\bullet),\Q)$ is considered to have a zero differential).
\end{proposition}

The operad ${D_n}^\bullet$ is an operad in the category of differential graded cocommutative coalgebras. The coproduct in each component is given by cosuperimposing, see Figure~\ref{fig:coprod}.
%{\bf GREG: I DID NOT REALLY UNDERSTAND THE COPRODUCT}:

%\vspace{0.3cm}
%
%\begin{center}
\begin{figure}[h]
\psfrag{D}[0][0][1][0]{$\Delta$}
\psfrag{p}[0][0][1][0]{$\pm$}
\psfrag{o}[0][0][1][0]{$\otimes$}
\includegraphics[width=13cm]{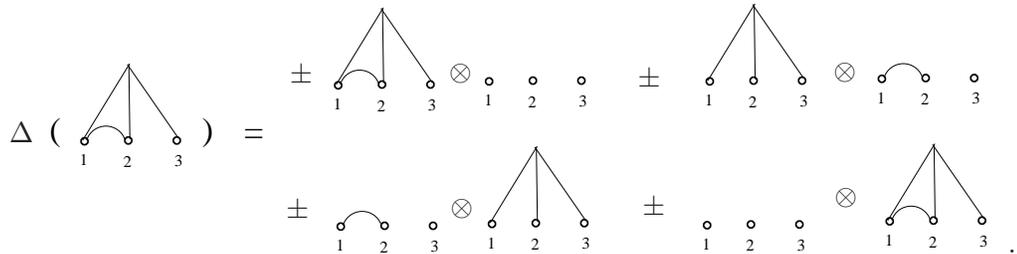}\,\, .
\caption{Example of a coproduct in ${D_n}^3$}\label{fig:coprod}
\end{figure}

%\end{center}

\vspace{.3cm}

\noindent In general for a graph $\Gamma\in {D_n}^k$ its coproduct $\Delta(\Gamma)\in {D_n}^k\otimes {D_n}^k$ is a sum of $2^c$ summands, where $c$ is the number of the connected components of the graph obtained from $\Gamma$ by removing its labeled vertices together with their small vicinities. For the graph from the above figure $c=2$. Its first connected component corresponds to the edge $12$, the second connected component corresponds to the subgraph consisting of the only internal vertex and its 3 adjacent edges. The counit  is defined as 1 on the trivial diagram without edges and internal vertices and as 0 on all the others. The morphism~\eqref{eq:operad_inclusion} is a morphism of  operads in coalgebras. Due to this morphism, ${D_n}^\bullet$ is an infinitesimal bimodule over $\HH_*(\balls_n(\bullet),\Q)$ and therefore over $\HH_0(\balls_n(\bullet),\Q)=\Comm$ as well, see Subsection~\ref{sss:inf_bimod}. Thus  ${D_n}^\bullet$ is a right $\Gamma$-module. Explicitly $\Comm$ in ${D_n}^\bullet$ is spanned by the diagrams without edges (and without internal vertices). The  infinitesimal left $\Comm$ action adds isolated label vertices. The infinitesimal right action is given by insertion of the product as in the lower part of Figure~\ref{fig:graph_compos}. It is easy to see that the $\Gamma$ structure maps respect the coalgebra structure, therefore  ${D_n}^\bullet$ is a right $\Gamma$-module in the category of coalgebras. %{\bf GREG: IS IT WORTHWHILE TO DESCRIBE THE GAMMA MODULE STRUCTURE EXPLICITLY?}

Let ${P_n}^k$ denote the primitive part of ${D_n}^k$. The space ${P_n}^k$ is spanned by the graphs with $k$ labeled external vertices that become connected if one removes all the external vertices together with their small vicinities.
The family of spaces ${P_n}^\bullet=\{{P_n}^k,\, k\geq 0\}$ is preserved by the $\Gamma$ structure maps, simply because these maps respect the coalgebra structure of ${D_n}^k$, $k\geq 0$.

%\begin{proposition}[\cite[Proposition~9.5]{Turchin}]\label{p:primitive_zigzag}
\begin{proposition}[\cite{Turchin,SevWil}]\label{p:primitive_zigzag}
For $n\geq 3$, the   right $\Gamma$-modules  $\pi_*(\balls_n(\bullet))\otimes\Q$ and ${P_n}^\bullet$  are quasi-isomorphic (by a zigzag of quasi-isomorphisms), where $\pi_*(\balls_n(\bullet))\otimes\Q$ is considered to have a zero differential.
\end{proposition}

For each $\bullet=k$, the morphism~\eqref{eq:operad_inclusion} is a quasi-isomorphism of differential graded cocommutative coalgebras. The configuration spaces $\balls_n(k)$ are known to be formal, thus the dual of ${D_n}^k$ is a rational model for $\balls_n(k)$. On the other hand the coalgebras ${D_n}^k$ are quasi-cofree with the space of cogenerators ${P_n}^k$. This explains why the homology of ${P_n}^k$ is $\Q\otimes\pi_*\balls_n(k)$.
The precise zigzag of quasi-isomorphisms is given in~the proof of~\cite[Theorem~9.3]{LT}, see also~\cite[Proposition~9.5]{Turchin}. The same construction was independently discovered and described in~\cite[Section~3]{SevWil}.

\begin{remark}\label{r:primitive_zigzag}
One can easily see that this zigzag respects the {\it complexity} $t$, which is the first Betti number of the graphs obtained by gluing all external vertices for the graphs from ${P_n}^\bullet$. For $\pi_*(\balls_n(\bullet))\otimes\Q$, the part of complexity $t$ is simply $\Q\otimes\pi_{t(n-2)+1}(\balls_n(\bullet))$. Moreover, for any given complexity one can also check that the zigzag of quasi-isomorphisms given in~\cite{LT} goes always through the bounded above non-negatively graded complexes.
\end{remark}

\subsubsection{Right $\Omega$-modules ${\hat D_n}{}^\bullet$, ${\hat P_n}{}^\bullet$}\label{ss:reduced_graphs}
Recall that in~\cite{PirashviliDold} Pirashvili defines a functor
$$
\CR\colon \mathrm{mod}{-}\Gamma\,\longrightarrow\,\mathrm{mod}{-}\Omega,
$$
that turns out to be an equivalence of abelian categories, see Subsection~\ref{ss:hom_as_omega_maps}.
%{\bf NEED A BETTER REFERENCE IN THE PAPER. ONE MIGHT NEED TO MAKE A Definition OF $\CR$.}
%If $X^\bullet$ is a right $\Gamma$-module, then the $k$-th component of $cr(X^\bullet)$ is the cokernel of the map
%$$
%\oplus_{i=1}^kX^{k-1}\to X^k,
%$$
%where the map on the $i$-th summand above is ${r_i}^*\colon X^{j-1}\to X^j$ corresponding through the $\Gamma^{op}$ structure to $r_i\colon \underline{j}\twoheadrightarrow \underline{j-1}$ defined as
%$$
%r_i(j)=
%\begin{cases}
%k,& \text{if $1\leq j<i$},\\
%*,& \text{$k=*$ or $i$},\\
%k-1,& \text{if $i<k\leq j$}.
%\end{cases}
%$$
Denote by ${\hat D_n}{}^\bullet$ and ${\hat P_n}{}^\bullet$ the right $\Omega$-modules $\CR({D_n}^\bullet)$ and $\CR({P_n}^\bullet)$ respectively. In each degree ${\hat D_n}{}^k$ and ${\hat P_n}{}^k$ are spanned by the same graphs as ${D_n}^k$, ${P_n}^k$ with the only restriction that all the external vertices in these graphs are of valence $\geq 1$. According to the definition of the cross-effect, ${\hat D_n}{}^k$ and ${\hat P_n}{}^k$ should be viewed as quotient spaces of ${D_n}^k$ and ${P_n}^k$, respectively (by the subspace spanned by the graphs having external vertices of valence~0). The category $\Omega$ can be viewed as a subcategory $\Gamma$ by adding to any finite set a base-point $*$. The right action of $\Omega$ is the restriction action of $\Gamma$ on these quotient spaces. For example, in the second composition in Figure~\ref{fig:graph_compos} one has to throw away the 2~graphs with external vertices of valence~0 to get the corresponding picture of the $\Omega$-action.

\begin{proposition}\label{p:injective}
The  right $\Omega$-modules ${\hat D_n}{}^\bullet$ and ${\hat P_n}{}^\bullet$ are finite-dimensional and injective in each homological degree.
\end{proposition}

\begin{proof}

We show first that they are finite-dimensional in each homological degree, which in particular means that all except a finite number of components of the right $\Omega$-modules are trivial for any given homological degree. Let $\Upsilon\in{\hat D_n}{}^k$ be a graph with $E$ edges, $I$ internal vertices, and $k$ external vertices. To recall the complexity $t$ of $\Upsilon$ is the first Betti number of the graph obtained from $\Upsilon$ by gluing together all external vertices. So, one has
$$
t=E-I.
\eqno(\numb)\label{eq:complexity}
$$
 The total degree of $\Upsilon$ is $(n-1)E-n\cdot I=(n-1)t-I$.
 %Thus one has that the total degree is greater than or equal to $n-1$ times the complexity. Thus the complexity is less than or equal to the total degree divided by $n-1$. As a consequence it is enough to show that ${\hat D_n}{}^\bullet$ (and therefore ${\hat P_n}{}^\bullet$) is finite-dimensional in each complexity $t$.
 Since the valence of any internal vertex is $\geq 3$, and the valence of any external one is $\geq 1$, one gets
 $$
 3I+k\leq 2E.
 $$
 Which implies $I\leq \frac 23E$, $k\leq 2E$. From~\eqref{eq:complexity} one has $E=t+I\leq t+\frac 23E$, so $E\leq 3t$, $I\leq 2t$, $k\leq 6t$. This very rough estimation shows that the set of graphs $\Upsilon$ in any given complexity $t$ is finite. On the other hand for a given complexity $t$ the total homological degree of any graph is $(n-1)t-I\geq (n-3)t$. Therefore there are finitely many complexities $t$ that can produce non-trivial graphs in a given homological degree.

 Before proving the injectivity, recall~\cite{PirashviliDold} that the category $\mathrm{mod}-\Omega$ of right $\Omega$-modules in $\Q$-vector spaces has injective cogenerators $\Omega_k^*$, $k\geq 0$. First one defines the left $\Omega$-modules $\Omega_k$, $k\geq 0$, as
 $$
 {\Omega_k}(\bullet)=\Q[\mathrm{Mor}_\Omega(k,\bullet)],
 $$
 which are projective generators of the category  $\Omega{-}\mathrm{mod}$ of left $\Omega$-modules in $\Q$-vector spaces. Their duals $\Omega_k^*$ are therefore injective right $\Omega$-modules. For any right $\Omega$-module $F$ one has
 $$
 \underset{\Omega}{\Rmod}(F,\Omega_k^*)\simeq (F(k))^*,
 $$
 where $(-)^*$ denote the dual vector space. This isomorphism  is due to the Yoneda lemma. Notice that $\Omega_k^*$ has a natural action of the symmetric group $\Sigma_k$ that comes from the automorphisms of $k\in \mathrm{Obj}(\Omega)$. Given any representation $V$ of $\Sigma_k$ one can define a right $\Omega$-module  $\Omega_k^*\otimes_{\Sigma_k}V$ whose $\ell$-th component is $\Q[\mathrm{Mor}_\Omega(k,\ell)]\otimes_{\Sigma_k}V$.

 \begin{lemma}\label{l:inj_mod} For any finite-dimensional representation $V$ of $\Sigma_k$, the right $\Omega$-module $\Omega_k^*\otimes_{\Sigma_k}V$ is injective. Moreover for any right $\Omega$-module $F$, one has
 $\underset{\Omega}{\Rmod}(F,\Omega_k^*\otimes_{\Sigma_k}V)\simeq \mathrm{hom}_{\Sigma_k}(F(k),V).$
 \end{lemma}

 \begin{proof}[Proof of Lemma~\ref{l:inj_mod}] Since the ground field is $\Q$, any finite-dimensional $\Sigma_k$-module $V$ is a direct summand of a finitely generated free $\Sigma_k$-module.   Thus $V\otimes_{\Sigma_k}\Omega_k^*$ is a direct summand of a finite sum of copies of $\Omega_k^*$ and therefore is also injective. For the second statement, since $V$ is finite-dimensional, one has:
 $$
  \underset{\Omega}{\Rmod}(F,\Omega_k^*\otimes_{\Sigma_k}V)\simeq  \underset{\Omega}{\Rmod}(F,\Omega_k^*)\otimes_{\Sigma_k}V\simeq (F(k))^*\otimes_{\Sigma_k}V\simeq \mathrm{hom}_{\Sigma_k}(F(k),V).
 $$
 \end{proof}

 Now let us show that ${\hat D_n}{}^\bullet$, ${\hat P_n}{}^\bullet$ are injective in each degree. Denote by $M({D_n}^k)$, $M({P_n}^k)$ the subspaces (which are actually subcomplexes) of ${\hat D_n}{}^\bullet$, ${\hat P_n}{}^\bullet$ respectively spanned by the graphs whose external vertices are all univalent. This notation comes from the fact that these spaces are spaces of {\it multiderivations} in ${D_n}^k$, ${P_n}^k$, see~\cite[Section~10]{Turchin}. Each $M({D_n}^k)$, $M({P_n}^k)$ has a natural $\Sigma_k$ action given by relabeling the external vertices.
The following lemma finishes the proof of Proposition~\ref{p:injective}.
\end{proof}

 \begin{lemma}\label{l:gr_omega_isom}
 The  graded right $\Omega$-modules ${\hat D_n}{}^\bullet$, ${\hat P_n}{}^\bullet$  are isomorphic to $\bigoplus_{k=0}^{+\infty}\Omega_k^*\otimes_{\Sigma_k}M({D_n}^k),$ and $\bigoplus_{k=1}^{+\infty}\Omega_k^*\otimes_{\Sigma_k}M({P_n}^k),$ respectively.
 \end{lemma}

\begin{proof}
Below we construct isomorphisms
\begin{gather}
\bigoplus_{k=0}^{+\infty}\Omega_k^*\otimes_{\Sigma_k}M({D_n}^k)\stackrel{\simeq}{\longrightarrow}{\hat D_n}{}^\bullet,
%\eqno(\numb)
\label{eq:inj_isom1}\\
\bigoplus_{k=0}^{+\infty}\Omega_k^*\otimes_{\Sigma_k}M({P_n}^k)\stackrel{\simeq}{\longrightarrow}{\hat P_n}{}^\bullet.
%\eqno(\numb)
\label{eq:inj_isom2}
\end{gather}
To recall $\Omega_k^*$ in degree $\ell$ is a vector space whose basis is the set of surjective maps $k\twoheadrightarrow\ell$.\footnote{The right action of $\Omega$ on $\Omega_k^*$ is as follows. Given a basis element $\alpha\colon k\twoheadrightarrow \ell$ of $\Omega_k^*$ and a morphism $f\colon\ell'\twoheadrightarrow \ell$ in $\Omega$, the result of the action of $f$ on $\alpha$ is the sum (with all positive signs) of all surjections $\alpha'\colon k\twoheadrightarrow \ell'$, such that $\alpha=f\circ\alpha'$.}   Given a surjective map $\alpha\colon k\twoheadrightarrow \ell$, viewed as an element of $\Omega_k^*$,  and a graph $\Upsilon\in M({D_n}^k)$, which means $\Upsilon$ has $k$ external vertices all of valence~1, one can construct a graph in ${D_n}^\ell$ as follows: take $\Upsilon$ and take $\ell$ vertices labeled by $1\ldots \ell$, and then join each external vertex $i$ of $\Upsilon$ with the labeled vertex $\alpha(i)$:

\vspace{.3cm}

\begin{center}
\psfrag{O}[0][0][1][0]{$\bigotimes$}
\includegraphics[width=11cm]{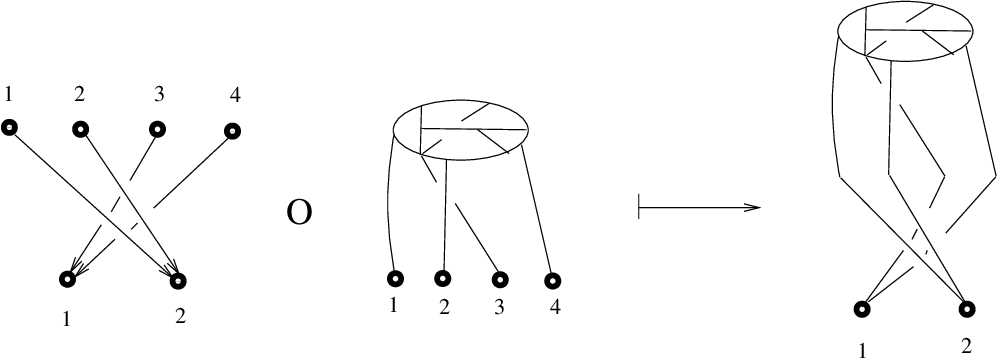},
\end{center}

\vspace{.3cm}

It is easy to see that the  maps~\eqref{eq:inj_isom1}-\eqref{eq:inj_isom2} defined as above are isomorphisms of $\Omega$-modules.
\end{proof}

We warn the reader that the isomorphisms~\eqref{eq:inj_isom1},~\eqref{eq:inj_isom2} send the differential of the source to the sum of expansions of internal vertices, which is only a part of the differential in ${\hat D_n}{}^\bullet$, ${\hat P_n}{}^\bullet$. The other part is the sum of expansion of external vertices. In other words, they are not morphisms of right $\Omega$-modules of chain complexes, but only of right modules of graded vector spaces. But the right hand sides of~\eqref{eq:inj_isom1} and~\eqref{eq:inj_isom2} do admit filtrations whose associated graded modules are isomorphic to the left hand sides. The $r$-th term of such filtration is spanned by the graphs  whose sum of valences of external vertices is $\leq r$.

\begin{remark}\label{r:loop_complexity1}
If $n$ is odd, the graphs with loops (edges connecting a vertex to itself) are canceled out by the orientation relations. For even $n$ if we quotient out ${P_n}^\bullet$ by the graphs with loops, the isomorphism~\eqref{eq:inj_isom2} fails to be true only in complexity~1, since there is only one graph
\begin{picture}(20,15)
\put(2,5){\circle*{3}}
\put(18,5){\circle*{3}}
\qbezier(2,5)(10,15)(18,5)
\put(0,-1){$_1$}
\put(16,-1){$_2$}
\end{picture}
  in $M({P_n}^\bullet)$ that can produce a loop by gluing external vertices. Because of that in complexity $\geq 2$ the graph-complex $\calEp$ can be reduced to a quasi-isomorphic complex consisting of graphs without loops.
\end{remark}

Now we finish the proof of Theorem~\ref{t:HE_Ext}.

\begin{proof}[Proof of Theorem~\ref{t:HE_Ext}]
One has
\begin{multline*}
\underset{\Omega}{\hRmod}\left(\widetilde{\HH}_*(S^{m\bullet},\Q),\hat\HH_*(\balls_n(\bullet),\Q)\right)\simeq
%\ext_{{\mathrm{mod}}-\Omega}\left(\tilde H(S^{m\bullet}),{\hat D_n}{}^\bullet\right)\simeq\\
\underset{\Omega}{\Rmod}(\widetilde{\HH}_*(S^{m\bullet},\Q),{\hat D_n}{}^\bullet))\simeq \\
\bigoplus_{k=0}^\infty\mathrm{hom}_{\Sigma_k}(\widetilde{\HH}_*(S^{mk}),M({D_n}^k)).
\end{multline*}
To recall the derived $\mathrm{hom}$ is taken in the model category $\mathrm{Ch}_{\geq 0}(\mathrm{mod}{-}\Omega)$.
The first isomorphism is due to the fact that the left-hand side can be written as a product of  $\mathrm{Ext}$ groups (in the abelian category $\mathrm{mod}{-}\Omega$ of right $\Omega$ modules in $\Q$ vector spaces)  since it is a space of derived maps between objects with trivial differential  and ${\hat D_n}{}^\bullet$ in any complexity $t$ is an injective resolution of the right $\Omega$-module $\hat\HH_{t(n-1)}(\balls_n(\bullet),\Q)$. Indeed it follows from Proposition~\ref{p:operad_inclusion} that the inclusion~\eqref{eq:operad_inclusion} is a quasi-isomorphic inclusion of right $\Omega$-modules. On the other hand  Proposition~\ref{p:injective} tells us that  ${\hat D_n}{}^\bullet$ is injective in any homological degree.
The second isomorphism is due to Lemmas~\ref{l:inj_mod} and~\ref{l:gr_omega_isom}.  Finally one can  notice that the  graded vector space $\bigoplus_k\mathrm{hom}_{\Sigma_k}(\widetilde{\HH}_*(S^{mk}),M({D_n}^k))=\bigoplus_s\mathrm{hom}_{\Sigma_s}
(\widetilde{\HH}_{ms}(S^{ms}),M({D_n}^s))$ is exactly $\calEH$ defined at the beginning of Subsection~\ref{ss:u3_graphs}. We only need to check that the differentials are   the same. Any element
$\phi\in \mathrm{hom}_{\Sigma_s}(\widetilde{\HH}_{ms}(S^{ms}),M({D_n}^s))$ in the above direct sum should be understood as an $\Omega$-module map $\phi\colon \oplus_k\widetilde{\HH}_{mk}(S^{m\bullet})\to {\hat D_n}{}^\bullet$ that sends all the summands to zero except the s-th one $\widetilde{\HH}_{ms}(S^{m\bullet})$. The latter $\Omega$-module is one dimensional and is concentrated in the $s$-th component. By Lemmas~\ref{l:inj_mod} and~\ref{l:gr_omega_isom} the map $\phi$ must send the generator of the one dimensional space $\widetilde{\HH}_{ms}(S^{ms})$ to some element  $\psi\in M({D_n}^s)\subset {\hat D_n}{}^s$. But  the part of the differential in ${\hat D_n}{}^\bullet$ that expands the external vertices must act trivially on $\psi$  (since  all the external vertices in any graph from $M({D_n}^s)$ are univalent). The other part of the differential corresponds to the expansion of internal vertices, which produces exactly the differential on $\calEH$.

The proof of~\eqref{eq:HEp_Ext} goes in the same way:
 \begin{multline*}
\underset{\Omega}{\hRmod}\left(\widetilde{\HH}_*(S^{m\bullet},\Q),\hat \pi_*(\balls_n(\bullet))\otimes\Q\right)\simeq
%\ext_{{\mathrm{mod}}-\Omega}\left(\tilde H(S^{m\bullet}),{\hat D_n}{}^\bullet\right)\simeq\\
\underset{\Omega}{\Rmod}(\widetilde{\HH}_*(S^{m\bullet},\Q),{\hat P_n}{}^\bullet)\simeq\\
\bigoplus_{k=0}^\infty\mathrm{hom}_{\Sigma_k}(\widetilde{\HH}_*(S^{mk}),M({P_n}^k)).
\end{multline*}
 But in this case even though ${\hat P_n}{}^\bullet$ is still a complex of injective $\Omega$-modules it is no more an injective resolution of $\pi_*(\balls_n(\bullet))\otimes\Q$, but is only quasi-isomorphic to it, see Proposition~\ref{p:primitive_zigzag}. However for any given complexity this quasi-isomorphism is a zigzag that goes through bounded above complexes, see  Remark~\ref{r:primitive_zigzag}.
 Thus the standard \lq\lq balancing $\ext$" argument can still be applied, see  the proof of~\cite[Theorem~2.7.6]{Weibel} and also~\cite[Exercise~10.7.1]{Weibel}.
  %prove the first isomorphism one has to evoke the model category structure~\cite{DwyerSpal} on the category of negatively graded chain complexes of right $\Omega$-modules.   The weak equivalences for this structure are as usual quasi-isomorphisms of complexes. The fibrations  are epic maps in each degree with injective kernels. The cofibrations are maps monic in every strictly negative degree. To get the negative grading we subtract from the total grading the complexity times $(n-1)$, since with a given complexity the maximal grading happens to the graphs without internal vertices. With this model category structure a result similar to~\cite[Proposition~7.3]{DwyerSpal} finishes  the proof of the first isomorphism.
  The second isomorphism and the fact that the obtained complex coincides with $\calEp$ are proven in the same way.%\footnote{GREG: DO YOU WANT TO SAY SOMETHING ABOUT THE UNBOUNDED COMPLEXES --- PROJECTIVE AND INJECTIVE MODEL STRUCTURES ON THEM?}

  A scrupulous  reader might prefer to see a product instead of a direct sum in~\eqref{eq:inj_isom1},~\eqref{eq:inj_isom2} and also in the two formulas above. But ${\hat D_n}{}^\bullet$ and ${\hat P_n}{}^\bullet$ are finite-dimensional in any homological degree (by Proposition~\ref{p:injective}) and so are complexes $\calEH$, $\calEp$, $n\geq 2m+2$ (by a similar argument). Thus in all these expressions the direct sum can be considered as a product.
\end{proof}

\section{Rational homotopy of $\Ebarmn$ in small dimensions}\label{ss:small_complex}
In this section we describe the rational homotopy of $\Ebarmn$ in complexities $t\leq 3$. The table in Subsection~\ref{ss:table_comput} summarizes these computations. Recall that we defined graph-complexes $\calEp$ computing the rational homotopy of $\Ebarmn$, see Theorem~\ref{t:HE=RH}. It is clear from the definition that up to a regrading the graph-complexes $\calEp$ depend on the parities of~$m$ and~$n$ only. The case $m=1$ was considered in~\cite[Section~9]{Turchin} which by a regrading describes the situation when $m$ is odd.  Notice that for even~$n$ the graphs with multiple edges cancel out by the orientation relations. For odd~$n$ the graphs with loops disappear by the same reason. Due to Remark~\ref{r:loop_complexity1} in complexities $\geq 2$ even when $n$ is even one can consider the reduced version of $\calEp$ spanned only by the graphs without loops.

We only give a brief summary of our computations that gives an idea of how these graph-complexes look like in small degrees.  In most of the cases the sign matters only when we check weather the corresponding graphs survive their symmetries. As for the differential in small complexities often the sign is not important --- for any choice of sign the resulting homology is the same. For this reason we do not specify how exactly the graphs are oriented, see Subsection~\ref{ss:u3_graphs} for the definition of the orientation of a graph and the differential in the graph-complex.

\subsection{Complexity~1}\label{sss:c1}
 There are only two graphs in this complexity:
 $$
 \begin{picture}(30,20)
\put(0,3){\line(2,1){30}}
\end{picture}\qquad\qquad
\begin{picture}(30,25)
\put(0,0){\line(2,1){15}}
\put(21,11.45){\circle{15}}
\end{picture}
$$
These graphs are in  Hodge degrees~2 and~1 respectively.  They survive the orientation relations and define non-trivial generators in rational homotopy according to Table~\ref{table:compl1}.
\begin{table}[!h]
\begin{center}
\begin{tabular}{|c|c|c|}
\hline
Cycle& appear when &degree\\
\hline
\begin{picture}(30,25)
\put(0,2){\line(2,1){30}}
\end{picture}& $n-m$ even& $n-2m-1$\\
\hline
\begin{picture}(30,25)
\put(0,0){\line(2,1){15}}
\put(21,11.45){\circle{15}}
\end{picture}&
$n$ even&$n-m-2$\\
\hline
\end{tabular}

\vspace{.2cm}

\end{center}

\caption{Rational homotopy generators in complexity 1}\label{table:compl1}
\end{table}

\subsection{Complexity 2}\label{sss:c2}

Due to Remark~\ref{r:loop_complexity1} one should consider only graphs without loops. Among those graphs there are only three that might survive the orientation relations, with only one graph in each of the Hodge degrees 3, 2, and~1. These graphs define non-trivial generators in rational homotopy according to Table~\ref{table:compl2}.

\begin{table}[!h]
\begin{center}
\begin{tabular}{|c|c|c|}
\hline
Cycle&appear when&degree\\
\hline
\includegraphics[width=1.5cm]{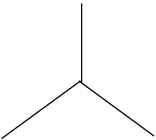}& $n-m$ odd& $2n-3m-3$\\
\hline
\includegraphics[width=2cm]{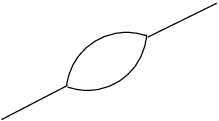}&
$m$ odd, $n$ odd&$2n-2m-4$\\
\hline
\includegraphics[width=1.5cm]{complexity2_odd_2}&$n$ odd&$2n-m-4$\\
\hline
\end{tabular}

\vspace{.2cm}

\end{center}

\caption{Rational homotopy generators in complexity 2}\label{table:compl2}
\end{table}

Notice that only the last graph is not uni-trivalent. The uni-trivalent graphs for all parities of $m$ and $n$ cancel out by the orientation relations in Hodge degree~1.

\subsection{Image of the connecting homomorphism}\label{ss:conn_hom}
One has a homotopy fibration
$$
\Ebarmn\to\Embmn\to\Omega^m\mathrm{Inj}(\R^m,\R^n),
$$
that produces a long exact sequence of homotopy groups:
\begin{multline}
\ldots\to \pi_{*+1}\Omega^m\mathrm{Inj}(\R^m,\R^n)\stackrel{{\partial_*}}{\to}\pi_*\Ebarmn\to\\
\to\pi_*\Embmn\to\pi_*\Omega^m\mathrm{Inj}(\R^m,\R^n)\to\ldots
\label{eq:long_ex_seq}
\end{multline}
Perhaps surprisingly for $n\geq 2m+2$ the image of ${\partial_*}$ in rational homotopy depends on the parities of $m$ and $n$ only. The following is an equivalent reformulation of Theorem~\ref{t:im_D*} or of Corollary~\ref{c:ranks}.

\begin{theorem}\label{t:connecting_hom}
For $n\geq 2m+2$, the image of the connecting homomorphism
$$
{\partial_*}\colon \Q\otimes \pi_{*+1}\Omega^m\mathrm{Inj}(\R^m,\R^n)\to \Q\otimes\pi_*\Ebarmn
\eqno(\numb)\label{eq:connect_hom}
$$
in rational homotopy is described by the homology of $\calEp$ spanned by the following graphs
$$
\begin{picture}(30,25)
\put(0,3){\line(2,1){30}}
\end{picture}\qquad
\begin{picture}(30,25)
\put(0,0){\line(2,1){15}}
\put(21,11.45){\circle{15}}
\end{picture}
\qquad
\includegraphics[width=1.2cm]{complexity2_odd_2}
\eqno(\numb)\label{eq:graphs_image}
$$
These classes are non-zero according to the following table:
%\begin{table}[!h]
\begin{center}
\begin{tabular}{|c|c|c|c|}
\hline
Cycle&appear when&degree&complexity\\
\hline
\begin{picture}(30,25)
\put(0,2){\line(2,1){30}}
\end{picture}& $n-m$ even& $n-2m-1$&$1$\\
\hline
\begin{picture}(30,25)
\put(0,0){\line(2,1){15}}
\put(21,11.45){\circle{15}}
\end{picture}&
$n$ even&$n-m-2$&$1$\\
\hline
\includegraphics[width=1.2cm]{complexity2_odd_2}&$n$ odd&$2n-m-4$&$2$\\
\hline
\end{tabular}
%\vspace{.2cm}
\end{center}
%\caption{Rational homotopy generators that are in the image of the connecting homomorphism}\label{table:imag_connect}
%\end{table}
\end{theorem}

We don't prove this result now. This theorem is equivalent to Theorem~\ref{t:im_D*}.
We will also see  in Section~\ref{s:emb} that the first and the second cycles come from the Euler classes of $\Q\otimes \pi_{*}\mathrm{Inj}(\R^m,\R^n)$, and the last one comes from the top Pontryagin class of $\Q\otimes \pi_{*}\mathrm{Inj}(\R^m,\R^n)$.

\subsection{Complexity 3}\label{ss:c3}

%\begin{center}
%{\it $m$ odd}
%\end{center}

The rank of the homology of $\calEp$ in complexity~3 is always two. The  homology generators are represented by the graphs shown in  Table~\ref{table:compl3}.
\begin{table}[!h]
\begin{center}
\begin{tabular}{|c|c|c|}
\hline
Cycle&appear when&degree\\
\hline
%Cycle&degree&parity\\
\includegraphics[width=1.7cm]{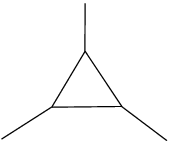}&$m$ even&$3n-3m-6$\\
\hline
\includegraphics[width=2.5cm]{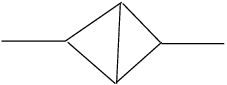}&$m$ odd& $3n-2m-7$\\
\hline
\includegraphics[width=1.7cm]{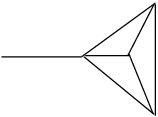}&always& $3n-m-7$\\
\hline
\end{tabular}

\vspace{.2cm}

\end{center}

\caption{Rational homotopy generators in complexity 3}\label{table:compl3}
\end{table}

Notice that the cycles depend only on the parity of $m$.
The graph-complexes actually do depend on the parity of $n$. As it follows from Remark~\ref{r:loop_complexity1}, when $n$ is even one can consider a quasi-isomorphic complex consisting of graphs without loops. It turns out that for even $n$ the only such graphs that are not canceled out by the orientation relations are those that appear in the above table. When $n$ is odd the complex is more complicated because of the presence of graphs with multiple edges. For the case $m$ odd, $n$ odd, see~\cite[Section~11]{Turchin}. The case $m$ even, $n$ odd is considered below.

%\vspace{.3cm}
%
%\begin{center}
%{\it $m$ even}
%\end{center}
%
% The rank of the homology of $\calEp$ is also 2 in this case. Table~\ref{table:compl3_m_even} describes the homology generators which are represented by one graph in each case.
%\begin{table}[!h]
%\begin{center}
%\begin{tabular}{|c|c|c|}
%\hline
%Cycle&degree&parity\\
%\hline
%\includegraphics[width=1.7cm]{complexity3_m_even_h2}& $3n-2m-7$& $n-1$\\
%\hline
%\includegraphics[width=1.7cm]{complexity3_even_hodge1.eps}&$3n-m-7$&$n$\\
%\hline
%\end{tabular}
%
%\vspace{.2cm}
%
%\end{center}
%
%\caption{Rational homotopy generators in complexity 3 for even $m$}\label{table:compl3_m_even}
%\end{table}
%
%\vspace{2cm}

%We describe below the graph-complexes.

\vspace{.3cm}

\noindent {\it $m$ even, $n$ odd.} For the Hodge degree 4 one has two non-trivial graphs:

$$
\psfrag{d}[0][0][1][0]{$\partial$}
\includegraphics[width=5cm]{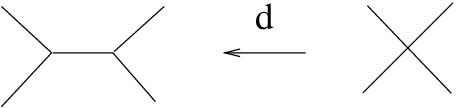}
\eqno(\numb)\label{eq:4hodge}
$$
The homology is still trivial, since $\partial(\,\includegraphics[width=0.5cm]{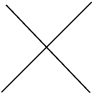}\,)=3\,\,\includegraphics[width=.8cm]{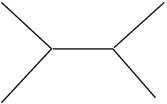}$.

\vspace{.2cm}

In the Hodge degree 3 one has the complex
$$
\psfrag{d}[0][0][1][0]{$\partial$}
\includegraphics[width=7cm]{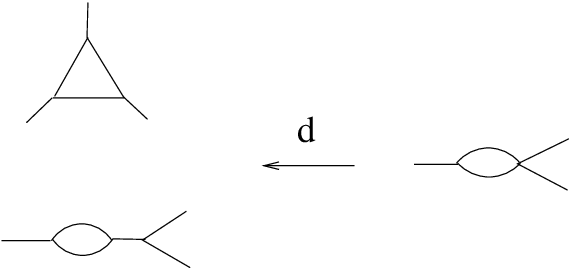}
%\eqno(\numb)\label{eq:4hodge}
$$
with the differential
$$
\psfrag{d}[0][0][1][0]{$\partial$}
\includegraphics[width=7cm]{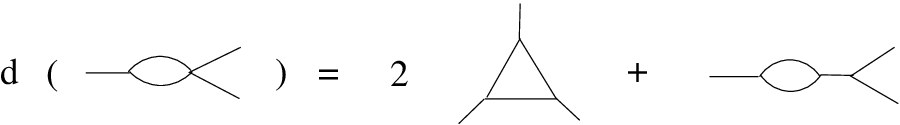}\,\, .
$$

This part of $\calEp$ produces only one-dimensional homology group spanned by the first graph from  Table~\ref{table:compl3}.

\vspace{.2cm}

In Hodge degree 2 one has the complex

$$
\psfrag{d}[0][0][1][0]{$\partial$}
\includegraphics[width=7cm]{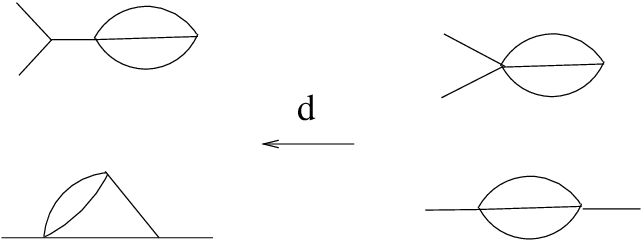}
%\eqno(\numb)\label{eq:4hodge}
$$

\noindent Which is acyclic.

\vspace{.2cm}

In the Hodge degree 1, one has the complex

\vspace{.3cm}
\begin{center}
\psfrag{d}[0][0][1][0]{$\partial$}
\includegraphics[width=10cm]{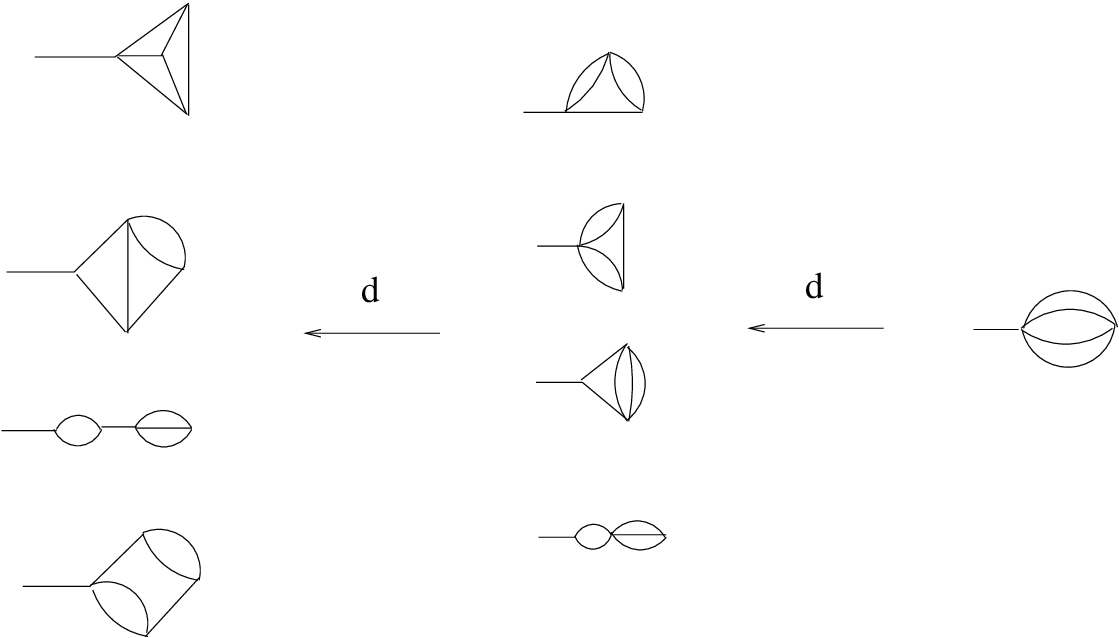}
\end{center}
\vspace{.3cm}

\noindent Which produces the last cycle from  Table~\ref{table:compl3}.

%\vspace{.5cm}
%
%\noindent {\it $m$ even, $n$ even.} This case is much simpler because graphs with multiple edges cancel out by the orientation relations in this situation.
%Due to Remark~\ref{r:loop_complexity1} one can consider only graphs without loops. It turns out that there are only two graphs without loops surviving the orientation relations. Both appear in the Table~\ref{table:compl3_m_even}.

\subsection{Subcomplex of trees or $s=t+1$}\label{ss:trees}
Denote by $\pi_*^{(s,t)}(\Ebarmn)$ the part of the rational homotopy of $\Ebarmn$ which lies in Hodge degree $s$ and complexity $t$. The summand $\calEp(s,t)$ in Hodge degree $s$ and complexity $t$ is non-trivial only if $s\leq t+1$. The case $s=t+1$ corresponds to the subcomplex of trees or graphs without cycles.

\begin{proposition}\label{p:trivial_in_trees}
For $t\geq 3$, one has $\pi_*^{(t+1,t)}(\Ebarmn)=0$ (with the usual assumption $n\geq 2m+2$).
\end{proposition}

\begin{proof}
Let $M_t({P_n}^{t+1})$ denote the subcomplex of $M({P_n}^{t+1})$, see Section~\ref{ss:reduced_graphs}, spanned by the trees. It is well known that the homology of this graph-complex is concentrated in the lowest degree and is described as the space of binary trees modulo $IHX$ relations~\cite{Whitehouse}. Up to a sign representation $\mathrm{sign}_{t+1}$ of the symmetric group $\Sigma_{t+1}$ this homology is isomorphic to the $t$-th component $\Lie(t)$ of the cyclic Lie operad. The subcomplex of $\calEp$ is isomorphic to
$$
M_t({P_n}^{t+1})\otimes_{\Sigma_{t+1}}(\sign_{t+1})^{\otimes m}.
$$
But it is well known, see for example~\cite{Klyachko,RobinsonWhite}, that both the invariants and anti-invariants of $\Lie(t)$ are trivial for $t\geq 3$. The result thus follows.
\end{proof}

As an example illustrating this result, the complex~\eqref{eq:4hodge} has trivial homology. Notice however that for the complexity $t=1$ and~2 one has non-trivial classes in Hodge degree $t+1$, see Tables~\ref{table:compl1}-\ref{table:compl2}.

\subsection{Subcomplex of graphs with one cycle or $s=t$}\label{ss:graphs1cycle}
It turns out that the homology of $\calEp$ can also be easily understood when the complexity $t$ equals the Hodge degree $s$, or in other words the homology of the subcomplex  $\bigoplus_{t\geq 1}\calEp(t,t)$ spanned by graphs with exactly one cycle. For any $k\geq 1$, define a $k$-wheel as a graph with $k$ univalent external vertices and $k$ trivalent internal vertices, which is obtained from a $k$-gon by adding an edge from each vertex. As an example a 5-wheel is pictured below:

$$
\includegraphics[width=3cm]{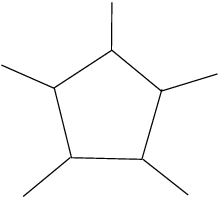}
$$

\begin{proposition}\label{p:wheels}
The homology of $\calEp(t,t)$  has rank 0 or 1 and is generated by the $t$-wheel, which is non-zero
\begin{itemize}
\item for $m$ odd, $n$ odd, if $t=2k$, $k\in\BN$;

\item for $m$ odd, $n$ even, if $t=4k-3$, $k\in\BN$;

\item for $m$ even, $n$ odd, if $t=4k-1$, $k\in\BN$;

\item for $m$ even, $n$ even, if $t=2k-1$, $k\in\BN$.
\end{itemize}
\end{proposition}

Notice that 1, 2, and 3-wheels appeared already in our computations. Since we are not going to use this result, we explain very briefly the way it can be proved.

\begin{proof}[Idea of the proof] The idea is that one can reduce $\calEp$ to a smaller complex of graphs without cut-vertices~\cite{ConVog}. The argument is similar to the one given in the above reference. For the case $s=t$ the only such graph is the $t$-wheel. One only needs to check when this $t$-wheel survives the dihedral symmetries.
\end{proof}

\begin{remark}\label{r:consist_comp}
The results of Subsections~\ref{sss:c1}-\ref{ss:graphs1cycle} are consistent with the computations of the Euler characteristics presented in Appendix in Tables~\ref{table:pi_oo},~\ref{table:pi_oe},~\ref{table:pi_eo},~\ref{table:pi_ee}. All these computations are also consistent with the previous computations of the rational homotopy of $\Ebar_c(\R,\R^n)$, see~\cite[Appendix~B, Tables~5-6]{T-OS}.
\end{remark}

Recently the second author together with Baltazar Chavez-Diaz, James Conant,  Jean Costello, and Patrick Weed computed the homology of the subcomplex of $\calEp$ of graphs with two cycles ($s=t-1$). The results of these computations show that the ranks of the homology of this part of the graph-complex grow linearly with~$t$. We believe the following is true.

\begin{conjecture}\label{c:asymp}
For $\ell$ fixed, the rank of $\pi_*^{(t-\ell,t)}(\Ebarmn)$ has asymptotic  $t^\ell$ when $t\to +\infty$.
\end{conjecture}

As another confirmation for this conjecture, the ranks of the  spaces of 3-loop uni-trivalent graphs   modulo $IHX$ relations (see Subsection~\ref{ss:u3_graphs}) grow quadratically with the number of univalent vertices in the graphs~\cite{Dasbach}.\footnote{These spaces are non-trivial only when the number of univalent vertices in 3-loop graphs is even~\cite{MoskOhtsuki}.}

\subsection{Recollecting results of computations}\label{ss:table_comput}

The table below resumes the results of computations from~\ref{sss:c1}-\ref{ss:graphs1cycle} and describes the dimensions of the first few linearly independent generators of $\Q\otimes\pi_*(\Ebarmn)$, $n>2m+1$. To obtain a similar description of the rational homotopy of $\mathrm{Emb}_c(\R^m,\R^n)$, $n>2m+1$, in small dimensions one can use Corollary~\ref{c:ranks}.

\vspace{.4cm}

\begin{center}
\begin{tabular}{|c|l|c|}
\hline
Case& Dimensions of rational homotopy&Other generators\\
&generators& have dimension\\
\hline
$m$ odd, $n$ odd& $n-2m-1$, $2n-2m-4$, $2n-m-4$,& $\geq 4n-3m-9$\\
                & $3n-2m-7$, $3n-m-7$, $4n-4m-8$& \\
                \hline
$m$ odd, $n$ even& $n-m-2$, $2n-3m-3$, $3n-2m-7$,& $\geq 4n-3m-9,$\\
                 & $3n-m-7$&                 \\
                 \hline
$m$ even, $n$ odd& $2n-3m-3$, $2n-m-4$, $3n-3m-6$,& $\geq 4n-3m-9$\\
                 & $3n-m-7$ & \\
                 \hline
$m$ even, $n$ even& $n-2m-1$, $n-m-2$, $3n-3m-6$,& $\geq 4n-3m-9$\\
                  & $3n-m-7$     &    \\
                  \hline
\end{tabular}

%\vspace{.2cm}
\end{center}

\section{Rational homotopy of $\Embmn$}\label{s:emb}
It is natural to ask ourselves: \lq\lq Is it possible to reconstruct from $\Q\otimes \pi_*(\Ebarmn)$ the rational homotopy of the initial embedding space $\Embmn$?" To answer this question it is enough to know the rational homotopy of $\Omega^m\mathrm{Inj}(\R^m,\R^n)$ (which is the homotopy of $\mathrm{Inj}(\R^m,\R^n)$ shifted by $m$), and the image of the morphism
$$
\Q\otimes\pi_*\Embmn\stackrel{D_*}{\longrightarrow}\Q\otimes\pi_*\Omega^m\mathrm{Inj}(\R^m,\R^n)
\eqno(\numb)\label{eq:D*}
$$
induced by the Smale-Hirsch map $D$. The space $\mathrm{Inj}(\R^m,\R^n)$ is  the Stiefel variety of isometric linear injections of $\R^m$ into $\R^n$. The rational homotopy of $\mathrm{Inj}(\R^m,\R^n)$ is finite-dimensional, see Theorem~\ref{t:hom_steif}. As a consequence up to a finite-dimensional correction the rational homotopy of $\Embmn$ and of $\Ebarmn$ is the same. We mention also that the rational homology of $\Embmn$, $n\geq 2m+2$, is a polynomial bialgebra generated by the rational homotopy. This follows from the fact that $\Embmn$, $n\geq 2m+2$,  is a double loop space. For $m\geq 2$ it is straightforward, since the spaces $\Embmn$, $n\geq 2m+2$, are connected and have an obvious action of the operad of $m$-cubes. For $m=1$ this result is due to Salvatore~\cite{Salvatore}.

\begin{theorem}\label{t:im_D*}
Let $n\geq 2m+2$. A non-zero element of $\Q\otimes\pi_*\Omega^m\mathrm{Inj}(\R^m,\R^n)$ is in the image of~\eqref{eq:D*} if and only it is both of degree $>2n-3m-4$ and  in the image of $(\Omega^mi)_*$, where $i$ is the natural inclusion
$$
\mathrm{Inj}(\R^{m-1},\R^{n-1})\stackrel{i}{\hookrightarrow}\mathrm{Inj}(\R^m,\R^n).
\eqno(\numb)\label{eq:st_inclus}
$$
\end{theorem}

\begin{proof}
R.~Budney proved that $\Embmn$ is $(2n-3m-4)$-connected~\cite[Proposition~3.9]{Budney2}. In the same paper he showed that the map $D$ up to homotopy factors through $\Omega^m\mathrm{Inj}(\R^{m-1},\R^{n-1})$~\cite[Theorem~2.5]{Budney2}:
$$
\xymatrix{
\Embmn\ar[rr]^D\ar[rd]_{hol}&&\Omega^m\mathrm{Inj}(\R^m,\R^n)\\
&\Omega^m\mathrm{Inj}(\R^{m-1},\R^{n-1})\ar[ru]_{\Omega^mi}&
}
$$
Thus to prove the above theorem, one only needs to show that every class $\omega\in\Q\otimes\pi_*\Omega^m\mathrm{Inj}(\R^m,\R^n)$, such that $\deg\,\omega>2n-3m-4$ and $\omega\in {\mathrm Im}\,
(\Omega^mi)_*$, does appear in ${\mathrm{Im}}\, D_*$. Notice that if $\omega\notin {\mathrm{Im}}\, D_*$, then ${\partial_*}\omega\neq 0$, where ${\partial_*}$ is the connecting homomorphism~\eqref{eq:connect_hom}. One also has
$$
\deg\, {\partial_*}\omega=\deg\,\omega-1.
$$

The end of the proof will go as follows. Assuming ${\partial_*}\omega\neq 0$ we will show that $\deg{\partial_*}\omega$ is too big to make ${\partial_*}\omega$ appear in complexity~1, too small to make it appear in complexity~$\geq 3$, and its parity does not match to appear in complexity~2. The space $\mathrm{Inj}(\R^m,\R^n)$ is a fibered product $S^{n-1}\ltimes S^{n-2}\ltimes\ldots\ltimes S^{n-m}$, and moreover the inclusion~\eqref{eq:st_inclus} is
$$
\xymatrix{
\mathrm{Inj}(\R^{m-1},\R^{n-1})\ar@{^{(}->}[r]\ar@{=}[d]&\mathrm{Inj}(\R^{m},\R^{n})\ar@{=}[d]\\
S^{n-2}\ltimes\ldots\ltimes S^{n-m}\ar@{^{(}->}[r]&S^{n-1}\ltimes S^{n-2}\ltimes\ldots\ltimes S^{n-m}}
$$

The rational homotopy of $\mathrm{Inj}(\R^{m},\R^{n})$ is described by the following well known result whose proof we will sketch.

\begin{theorem}\label{t:hom_steif}
Assuming $n>2m\geq 2$, one has
\begin{itemize}
\item for $m$ odd, $n$ odd:
$$
\Q\otimes\pi_*\mathrm{Inj}(\R^m,\R^n)=
\begin{cases}
\Q,& *=n-m \text{ or $2n-3-4k$, $0\leq k\leq \frac{m-1}2$;}\\
0,& \text{otherwise}.
\end{cases}
$$

\item for $m$ odd, $n$ even:
$$
\Q\otimes\pi_*\mathrm{Inj}(\R^m,\R^n)=
\begin{cases}
\Q,& *=n-1 \text{ or $2n-5-4k$, $0\leq k\leq \frac{m-3}2$;}\\
0,& \text{otherwise}.
\end{cases}
$$

\item for $m$ even, $n$ odd:
$$
\Q\otimes\pi_*\mathrm{Inj}(\R^m,\R^n)=
\begin{cases}
\Q,& *=\text{$2n-3-4k$, $0\leq k\leq \frac{m-2}2$;}\\
0,& \text{otherwise}.
\end{cases}
$$

\item for $m$ even, $n$ even:
$$
\Q\otimes\pi_*\mathrm{Inj}(\R^m,\R^n)=
\begin{cases}
\Q,& *=n-1, \text{ $n-m$, or $2n-5-4k$, $0\leq k\leq \frac{m-2}2$;}\\
0,& \text{otherwise}.
\end{cases}
$$
\end{itemize}
\end{theorem}

The classes of degree $n-1$ and $n-m$ will be called Euler classes, the other classes will be called Pontryagin classes. Notice that Pontryagin classes have degrees $4\ell-1$, $\ell\in\BN$. A similar statement is true in the range $n>m\geq 1$, but if $n\leq 2m$ the Euler and Pontryagin classes can lie in the same degrees, which makes more difficult to formulate the result.

\begin{proof}[Sketch of the proof]
One has a fibration
$$
SO(n-m)\stackrel{\iota}{\to}SO(n)\stackrel{p}{\to}\mathrm{Inj}(\R^{m},\R^{n}).
$$
The rational homotopy of $\mathrm{Inj}(\R^{m},\R^{n})$ can be split into a direct sum of the image of $p_*$ (containing all the Pontryagin classes of $\mathrm{Inj}(\R^{m},\R^{n})$ and the Euler class of degree $n-1$, appearing if $n$ is even) and a space transversal to $\mathrm{Im}\,p_*$ which gets mapped by a connecting homomorphism ${\partial_*}$ isomorphically on $\mathrm{Im}\,{\partial_*}$ (this space has dimension zero or one and is generated by the Euler class of degree $n-m$, appearing if $n-m$ is even). Theorem~\ref{t:hom_steif} is easily proved by a careful study of the map $\iota_*$ in rational homotopy. As we already mentioned $\mathrm{Inj}(\R^{m},\R^{n})$ is a fibered product of spheres:
$$
\mathrm{Inj}(\R^{m},\R^{n})=S^{n-1}\ltimes S^{n-2}\ltimes \ldots\ltimes S^{n-m}
\eqno(\numb)\label{eq:fib_prod}
$$
The Pontryagin class of degree $4\ell-1$ appears either from the factor $S^{2\ell}\ltimes S^{2\ell-1}$ of~\eqref{eq:fib_prod} or from the Hopf class of the last sphere if $n-m=2\ell$.  The Euler class of degree $n-1$ comes from the first sphere of~\eqref{eq:fib_prod} and it appears only if $n$ is even. The Euler class of degree $n-m$ corresponds to the last sphere of~\eqref{eq:fib_prod} and it appears only if $n-m$ is even. The case $m=1$ fits into Theorem~\ref{t:hom_steif}: when $n$ is even the only sphere $S^{n-1}$ is treated as the first sphere in~\eqref{eq:fib_prod}; when $n$ is odd it is treated as the last one.
\end{proof}

We now return to the proof of Theorem~\ref{t:im_D*}. We first list the classes in $\Q\otimes \pi_*\Omega^m\mathrm{Inj}(\R^{m},\R^{n})$ that can not lie in $\mathrm{Im}\, D_*$ by Budney's result. Basically these are the Euler classes (since their dimensions are too small) and the top Pontryagin class when $n$ is odd (since it is not in the image of $i_*$).  When $n$ is even the only class which is not in the image of $i_*\colon\Q\otimes\pi_*\mathrm{Inj}(\R^{m-1},\R^{n-1})\to\Q\otimes \pi_*\mathrm{Inj}(\R^{m},\R^{n})$ is the Euler class of degree $n-1$. Via the image of ${\partial_*}$ it appears as a class of degree $n-m-2$ in $\Q\otimes\pi_*\Ebarmn$. The shift of degree by $m$ comes from $m$ loops. This class corresponds to the graph
$$
\begin{picture}(30,25)
\put(0,0){\line(2,1){15}}
\put(21,11.45){\circle{15}}
\end{picture}
$$
\noindent from~\eqref{eq:graphs_image}.

When $n$ is odd and $m\geq 2$ there is only one class not lying in ${\mathrm{Im}}\, i_*$ which is the top Pontryagin class whose degree is $2n-3$. Via the connecting homomorphism it appears as a class of degree $2n-m-4$ in $\Q\otimes\pi_*\Ebarmn$. The corresponding graph from~\eqref{eq:graphs_image} is
$$
\includegraphics[width=1.2cm]{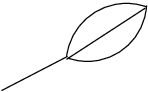}.
$$

For $n$ odd and $m=1$, there are two classes not in the image of $i_*$ (which is a zero map since the source is zero): the Pontryagin class of degree $2n-1$ and the Euler class of degree $n-1$. The corresponding graphs are
$$
\begin{picture}(30,25)
\put(0,2){\line(2,1){30}}
\end{picture},
\qquad
\includegraphics[width=1.2cm]{complexity2_odd_2.eps}.
$$

Assuming $n\geq 2m+2$, the Euler class of degree $n-2m$ in $\Q\otimes \pi_*\Omega^m\mathrm{Inj}(\R^{m},\R^{n})$ can not lie in ${\mathrm{Im}}\, D_*$ since by Budney's result~\cite[Proposition~3.9]{Budney2} the space $\Embmn$ is $(2n-3m-4)$-connected and
$$
n-2m>2n-3m-4
$$
is possible only if $n=4$ and $m=1$, but the case $m=1$ excludes the Euler class from the image of $D_*$ by the previous argument. Via the connecting homomorphism this class produces a generator of degree $n-2m-1$ in $\Q\otimes\pi_*\Ebarmn$. The corresponding graph from~\eqref{eq:graphs_image} is
$$
\begin{picture}(30,25)
\put(0,1){\line(2,1){30}}
\end{picture}.
$$

The degree of the bottom Pontyagin class of $\mathrm{Inj}(\R^m,\R^n)$ is $2n-2m-1$ or $2n-2m+1$ depending on the parity of $n-m$. Shifted by $m$ it appears in degree $\geq 2n-3m-1$ of $\Q\otimes\pi_*\Omega^m\mathrm{Inj}(\R^m,\R^n)$. But $2n-3m-1>2n-3m-4$, so neither of the other Pontryagin classes are excluded from ${\mathrm{Im}}\, D_*$ by Budney's results.

Notice that the above analysis shows that Theorem~\ref{t:im_D*} is equivalent to Theorem~\ref{t:connecting_hom}.

To finish the proof of Theorem~\ref{t:im_D*} one has to show that all the Pontryagin classes except the top one (when $n$ is odd) lie in ${\mathrm{Im}}\, D_*$. Notice that the degree of such class $\omega\in\Q\otimes\pi_*\Omega^m\mathrm{Inj}(\R^m,\R^n)$ is $4\ell-1-m$ for some $\ell>0$. If $\omega$ does not lie in $\mathrm{Im}\, D_*$, then  ${\partial_*}\omega\neq 0$ and the degree of ${\partial_*}\omega$ is $4\ell -m-2$. One has ${\partial_*}\omega$ can not lie in complexity~1, since complexity one is already taken by the image of the Euler classes. It can not lie in complexity~2, since the parity of its degree is $m$, and there are no classes other than \includegraphics[width=0.7cm]{complexity2_odd_2.eps}
of this parity in complexity~2, see Table~\ref{table:compl2}, which is also already taken by the image of the top Pontryagin class. Finally let us check that ${\partial_*}\omega$ can not lie in complexity~$\geq 3$.   Assuming that ${\partial_*}\omega$ corresponds to a cycle in graph-homology of $\calEp$ in complexity $t$ and Hodge degree $s$, let $I$ denote the number of internal vertices of one of the graphs in the linear combination representing ${\partial_*}\omega$, and $E$ be the number of edges in this graph. One has
$$
E\geq\frac{3I+s}2.
$$
One also has that the complexity $t$ is
$$
t=E-I.
$$
From the above
$$
E\leq 3t-s.
$$
The total degree is
$$
\deg\,{\partial_*}\omega=(n-1)E-n\cdot I-ms=nt-ms-E\geq (n-3)t-(m-1)s.
$$
Since for $t\geq 3$ one has $s\leq t$ (Proposition~\ref{p:trivial_in_trees}), the total degree can be estimated as
$$
\deg\,{\partial_*}\omega\geq (n-3)t -(m-1)s\geq (n-m-2)t\geq 3(n-m-2).
$$
On the other hand the degree of the highest Pontryagin class, which is  in the image of $(\Omega^mi)_*$, is $2n-5-m$ or $2n-7-m$ depending on the parity of $n$. One can easily check that $n>2m+1$ implies
$$
3(n-m-2)>2n-5-m,
$$
which finishes the proof of Theorem~\ref{t:im_D*}.
\end{proof}

Theorem~\ref{t:connecting_hom} is an equivalent reformulation of Theorem~\ref{t:im_D*}. Another equivalent reformulation is given below.

\begin{corollary}\label{c:ranks}
Assuming $n\geq 2m+2$, the ranks of the rational homotopy of $\Embmn$ are related to that of $\Ebarmn$ as follows:
\begin{itemize}
\item for $m$ odd, $n$ odd: %$\rank \Q\otimes\pi_*\Embmn=$
\begin{multline*}
\rank \Q\otimes\pi_*\Embmn=\\
\begin{cases}
\rank\Q\otimes\pi_*\Ebarmn+1,& *=2n-m-7-4k,\,\,\, 0\leq k\leq \frac{m-3}2;\\
\rank\Q\otimes\pi_*\Ebarmn-1,& *=2n-m-4,\,\,\, n-2m-1;\\
\rank\Q\otimes\pi_*\Ebarmn,& \text{otherwise}.
\end{cases}
\end{multline*}
%$$
%\begin{cases}
%\rank\Q\otimes\pi_*\Ebarmn+1,& *=2n-m-7-4k,\,\,\, 0\leq k\leq \frac{m-3}2;\\
%\rank\Q\otimes\pi_*\Ebarmn-1,& *=2n-m-4,\,\,\, n-2m-1;\\
%\rank\Q\otimes\pi_*\Ebarmn,& \text{otherwise}.
%\end{cases}
%$$

\item for $m$ odd, $n$ even:
\begin{multline*}
\rank \Q\otimes\pi_*\Embmn=\\
\begin{cases}
\rank\Q\otimes\pi_*\Ebarmn+1,& *=2n-m-5-4k,\,\,\, 0\leq k\leq \frac{m-3}2;\\
\rank\Q\otimes\pi_*\Ebarmn-1,& *=n-m-2;\\
\rank\Q\otimes\pi_*\Ebarmn,& \text{otherwise}.
\end{cases}
\end{multline*}

\item for $m$ even, $n$ odd:
\begin{multline*}
\rank \Q\otimes\pi_*\Embmn=\\
\begin{cases}
\rank\Q\otimes\pi_*\Ebarmn+1,& *=2n-m-7-4k,\,\,\, 0\leq k\leq \frac{m-4}2;\\
\rank\Q\otimes\pi_*\Ebarmn-1,& *=2n-m-4;\\
\rank\Q\otimes\pi_*\Ebarmn,& \text{otherwise}.
\end{cases}
\end{multline*}

\item for $m$ even, $n$ even:
\begin{multline*}
\rank \Q\otimes\pi_*\Embmn=\\
\begin{cases}
\rank\Q\otimes\pi_*\Ebarmn+1,& *=2n-m-5-4k,\,\,\, 0\leq k\leq \frac{m-2}2;\\
\rank\Q\otimes\pi_*\Ebarmn-1,& *=n-m-2,\,\,\, n-2m-1;\\
\rank\Q\otimes\pi_*\Ebarmn,& \text{otherwise}.
\end{cases}
\end{multline*}
\end{itemize}

\end{corollary}

\section{Koszul complexes}\label{s:koszul}
\subsection{Cofibrant model for $\widetilde{\HH}_*(S^{m\bullet},\Q)$}\label{ss:proj_cofibr_model}
In this section we introduce an explicit cofibrant model (in the projective model structure) of the right $\Omega$-module $\widetilde{\HH}_*(S^{m\bullet},\Q)$. In Subsection~\ref{ss:compl_der_maps} using this cofibrant replacement we produce another model for the complex of derived maps $\underset{\Omega}{\hRmod}(\widetilde{\HH}_*(S^{m\bullet},\Q),N)$, where $N$ is any right $\Omega$-module. Applying this to the case $N=\hatH$ or $N=\hatP$ one obtains complexes $\calK_\HH^{m,n}$ and $\calK_\pi^{m,n}$ computing respectively $\HH_*(\Ebarmn,\Q)$ and $\Q\otimes\pi_*(\Ebarmn)$.

One has $\widetilde{\HH}_*(S^{m\bullet},\Q)=\bigoplus_{s\geq 0}\widetilde{\HH}_{ms}(S^{m\bullet},\Q)$. Denote by $Q_s^m$ the $\Omega$-module $\widetilde{\HH}_{ms}(S^{m\bullet},\Q)$. It is concentrated in a single component
$$
Q_s^m(k)=
\begin{cases}
\Sigma^{ms}\Q,& k=s;\\
0,& k\neq s.
\end{cases}
$$
In the above expression $\Sigma^{ms}$ means that the corresponding 1-dimensional vector space lies in the grading $ms$. The symmetric group action on this component coincides with the sign representation in case $m$ is odd and is trivial in case $m$ is even.

Below we describe explicitly a cofibrant replacement $CQ_s^m$ of $Q_s^m$. This construction comes from the theory of operads. As we explained in Subsection~\ref{ss:operads} the structure of a right $\Omega$-module is the same thing as the structure of a right module over the operad $\Comm_+$ of commutative non-unital algebras. In~\cite[Section~5.3]{Fresse} Fresse defines a cofibrant replacement of a right module $M$ over any Koszul operad $\calP$. He denotes such replacement $CM$ by $K(M,P,P)$. As a symmetric sequence it is $CM=M\circ\calP^\text{!`}[1]\circ\calP$, where $\calP^\text{!`}$ is the cooperad Koszul dual to $\calP$, and \lq\lq $[1]$" denotes the operadic suspension, see Subsection~\ref{ss:operads}. The differential in $M\circ\calP^\text{!`}[1]\circ\calP$ comes from two ingredients: when a cooperation from $\calP^\text{!`}[1]$ acts from the right on $M$ or when it acts from the left on $\calP$. For the precise definition and construction, see~\cite{Fresse}.  Below we explain this construction  in the case of $CQ_s^m=K(Q_s^m,\Comm_+,\Comm_+)$. Notice that the action of $\Comm_+$ on $Q_s^m$ is trivial. Because of that only the second part of the differential is non-trivial for $CQ_s^m$. The operad $\Comm_+$ is well known to be Koszul whose Koszul dual is the operad $\Lie$ of Lie algebras. We will use the description of the cooperad $\coLie=\Comm_+{}^\text{!`}$ of Lie algebras that interprets  the components of the cooperad as  spaces of trees modulo Arnol'ld relations. This description arises from the duality between the homology and cohomology of configuration spaces, see for example~\cite{Sinha:GD} or~\cite[Section~5]{T-OS}, see also Subsection~\ref{sss:operads_examples}.

For a finite set $K$ with $k$ elements, the component $CQ_s^m(K)$ of $CQ_s^m=Q_s^m\circ\coLie[1]\circ\Comm_+$ is a chain complex spanned by {\it oriented forests} with $s$ connected components. The vertices of the forests are disjoint subsets of $K$ so that all the vertices of any such forest define a partition of $K$. An {\it orientation set} of a forest is the union of the set of its connected components (considered as elements of degree $m$) and the set of edges (considered as elements of degree~1). An {\it orientation} of a forest is an ordering of its orientation set. The relations in $CQ_s^m$ are the orientation relations and the Arnol'd relations. More precisely one has $T_1=\pm T_2$ if the forest $T_1$ differs from $T_2$ only in reordering of the orientation set. The sign is the corresponding Koszul sign of permutation. The Arnol'd relations have the form:

\begin{center}
\psfrag{S1}[0][0][1][0]{$S_1$}
\psfrag{S2}[0][0][1][0]{$S_2$}
\psfrag{S3}[0][0][1][0]{$S_3$}
\includegraphics[width=10cm]{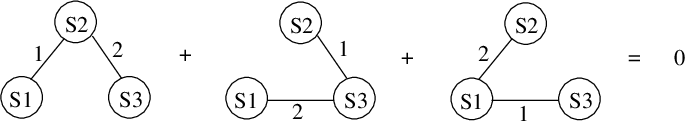}
\end{center}
where $S_i$, $i=1$, $2$, $3$, are disjoint subsets of $K$. All the other edges (and vertices) in each of the three forests are the same. The differential in $CQ_s^m$ is the sum (with appropriate signs) of contractions of edges. The new vertex is labeled by the union of two sets corresponding to the vertices that bounded the collapsed edge. The sign is obtained by pulling the edge that is contracted to the first place in the orientation set and then forgetting it since it is no more present in the orientation set of the new forest:

\begin{center}
\psfrag{d}[0][0][1][0]{$\partial$}
\includegraphics[width=14cm]{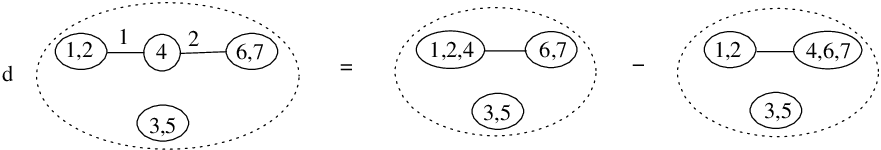}
\end{center}

For a surjective map $g\colon K_1\to K_2$ between two finite sets, one has the induced map
$$
g^*\colon CQ_s^m(K_2)\to CQ_s^m(K_1),
$$
that sends any forest $T\in CQ_s^m(K_2)$ to the forest $g^*(T)$ obtained from $T$ by replacing each vertex $S_i\subset K_2$ with $g^{-1}(S_i)$.

The map $CQ_s^m\to Q_s^m$ sends everything to zero except the forest without edges and whose all vertices are singletons. The latter forest gets mapped to a generator of $Q_s^m(s)$.

From the definition it is clear that $CQ_s^m$ as a right $\Comm_+$ module is freely generated by a certain symmetric sequence $KQ_s^m=Q_s^m\circ\coLie[1]$ that we call the {\it Koszul dual} of $Q_s^m$. Notice that $KQ_s^m$ is naturally a (cofree) right comodule over the cooperad $\coLie[1]$. Explicitly $KQ_s^m(K)$ is a subspace of $CQ_s^m(K)$ spanned by forests whose all vertices are singletons. Notice also that $KQ_s^m(K)$ is concentrated in the single homological degree $ms+(k-s)=(m-1)s+k$ (such forests have $s$ connected components and $k-s$ edges).

It is easy to show that for $m\geq 2$ one has a natural isomorphism of symmetric sequences:
$$
KQ_s^m(\bullet)\simeq \widetilde{\HH}_{(m-1)s+\bullet}(S^{m\bullet}/\Delta^\bullet S^m,\Q),
\eqno(\numb)\label{eq:koszul_Q_config_homology}
$$
where $\Delta^kS^m$ denotes the \lq\lq fat diagonal" in $S^{mk}$ ---  the union of subspaces $x_i=x_j$, $1\leq i\neq j\leq k$, in the smash product $S^m\wedge\ldots \wedge S^m=S^{mk}$.

To see this isomorphism one can notice first that
$$
\widetilde{\HH}_*(S^{mk}/\Delta^kS^m,\Q)=\overline{\HH}_*(\mathrm{C}(k,\R^m),\Q),
\eqno(\numb)\label{eq:one_pt_comp_isom_conf_sp}
$$
where $\overline{\HH}_*(-)$ denotes the locally compact singular homology. Then one should use the Poincar\'e duality together with the description of the cohomology groups of $\mathrm{C}(k,\R^m)$ in terms of spaces of trees modulo Arnol'd relations~\cite{Arn}. We leave it as an exercise to the reader. In~\cite{AT} we gave a more geometric explanation for the isomorphism~\eqref{eq:koszul_Q_config_homology}.

To finish this section we mention that the Koszul dual of $\widetilde{\HH}_*(S^{m\bullet},\Q)$ is the sum of Koszul duals to $Q_s^m$, $s\geq 0$, and is exactly the symmetric sequence $\overline{\HH}_*(\mathrm{C}(\bullet,\R^m),\Q)$.

\subsection{Koszul complex of derived maps}\label{ss:compl_der_maps}
Let $N$ be a right $\Omega$-module. For simplicity we will be assuming that $N$ has a trivial differential, as is the case with $\hatH$ and $\hatP$. We will now desrcibe explicitly the complex of derived maps
$$
\underset{\Omega}{\hRmod}(\widetilde{\HH}_*(S^{m\bullet},\Q),N)=\prod_{s\geq 0}\underset{\Omega}{\Rmod}(CQ_s^m,N).
\eqno(\numb)\label{eq:rmod_maps}
$$
Since $CQ_s^m$ is freely generated by the symmetric sequence $KQ_s^m$, one has that the space of the above complex is the product
$$
\prod_{s\geq 0}\prod_{k\geq s} \mathrm{hom}_{\Sigma_k}(KQ_s^m(k),N(k)).
\eqno(\numb)\label{eq:prod_rmod_maps}
$$
In the case $N=\hatH$ or $N=\hatP$, $n\geq 2m+2$, the product above can be replaced by a direct sum. Indeed, the homological degree of $KQ_s^m(k)$ is $(m-1)s+k$. The spaces $\hat\HH_*(B_n(k),\Q)$ and $\Q\otimes\hat\pi_*(B_n(k))$ are concentrated in the degrees $\geq\frac k2(n-1)$ and $\geq (k-1)(n-2)+1$ respectively. Therefore
$\mathrm{hom}_{\Sigma_k}(KQ_s^m(k),\hat\HH_*(B_n(k),\Q))$ and $\mathrm{hom}_{\Sigma_k}(KQ_s^m(k),\Q\otimes\hat\pi_*(B_n(k)))$ are concentrated in the degrees $\geq\frac k2(n-3)-(m-1)s$ and $\geq k(n-3)-(m-1)s-n+3$ respectively. Notice that when $s$ is fixed this gradings go to infinity with $k$. As a consequence the second product in~\eqref{eq:prod_rmod_maps} can be replaced by a direct sum in these two cases. On the other hand since $k\geq s$ , we obtain
$$
\frac k2(n-3)-(m-1)s\geq \frac{n-2m-1}2s,
$$
$$
k(n-3)-(m-1)s-n+3\geq (n-m-2)s-n+3.
$$
Since $n\geq 2m+2$ we obtain that this minimal grading goes to infinity with $s$, and therefore the first product in~\eqref{eq:prod_rmod_maps} can also be replaced by a direct sum. A similar argument shows that $\mathrm{hom}_{\Sigma_k}(KQ_s^m(k),N(k))$ is also a direct sum of spaces taken by the homological degree of $N(k)$. Thus the complexes of derived maps we are interested in can be written as direct sums of complexes
\begin{multline}
\calK_\HH^{m,n}=
\underset{\Omega}{\hRmod}\left(\widetilde{\HH}_*(S^{m\bullet},\Q),\hatH\right)=\\
\bigoplus_{s,t}\left(\bigoplus_{k\geq s}\mathrm{hom}_{\Sigma_k}\left(KQ_s^m(k),\hat\HH_{t(n-1)}(B_n(k),\Q)\right),\partial\right);
\label{eq:compl_der_maps_H}
\end{multline}

\begin{multline}
\calK_\pi^{m,n}=
\underset{\Omega}{\hRmod}\left(\widetilde{\HH}_*(S^{m\bullet},\Q),\hatP\right)=\\
\bigoplus_{s,t}\left(\bigoplus_{k\geq s}\mathrm{hom}_{\Sigma_k}\left(KQ_s^m(k),\Q\otimes\hat\pi_{1+t(n-2)}(B_n(k))\right),\partial\right).
\label{eq:compl_der_maps_pi}
\end{multline}

In the first case one  has a restriction $k\leq 2t$. In the second case one has restriction $k\leq t+1$. %We call the additional grading $s$ {\it Hodge degree} and we call the grading $t$ {\it complexity}.

Now let us describe the differential in the complex $\underset{\Omega}{\Rmod}(CQ_s^m,N)$. Let $f$ be a pure element in this complex lying in $\mathrm{hom}_{\Sigma_k}(KQ_s^m(k),N(k))$.
The element $\partial f$ is also pure and lies in $\mathrm{hom}_{\Sigma_{k+1}}(KQ_s^m(k+1),N(k+1))$ (here we are using the fact that $N$ has trivial differential).  Let $T$ be a forest in $KQ_s^m(k+1)$. One has
$$
(\partial f)(T)=\partial(f(T))-(-1)^{|f|}f(\partial T).
$$
Since we assume that $N$ has trivial differential the first summand can be ignored, and one has
$$
(\partial f)(T)=(-1)^{|f|-1}f(\partial T).
$$
On the other hand,
$$
\partial T=\sum_{e\in E(T)}\pm\gamma_e^*(T/e),
$$
where $E(T)$ is the set of edges of $T$; $T/e$ is the forest obtained from $T$ by contracting edge $e$. We view $T/e$ as an element of $KQ_s^m((k+1)/e)$ with $(k+1)/e$ being the set obtained from $\{1,2,\ldots,k+1\}$ by identifying the endpoints of $e$. The map
$$
\gamma_e^*\colon CQ_s^m((k+1)/e)\to CQ_s^m(k+1)
$$
above is induced by the surjective map
$$
\gamma_e\colon k+1\to (k+1)/e.
$$
(Abusing notation $\gamma_e^*$ also  denotes below the induced map $\gamma_e^*\colon N((k+1)/e)\to N(k+1)$.) Since $f$ is a morphism of $\Omega$-modules, we finally get
$$
(\partial f)(T)=\sum_{e\in E(T)}\pm\gamma_e^*(f(T/e)).
\eqno(\numb)\label{eq:differential_koszul}
$$
The sign is $(-1)^{|f|-1}$ times the sign obtained by pulling the edge $e$ on the first place of the orientation set of $T$.

\subsection{Complex of bicolored  graphs}\label{ss:bicol_graphs}
In this subsection we will describe a complex $\HHHH^{m,n}$  dual to $\calK_\HH^{m,n}$. The construction is a straightforward application of the duality between the homology and cohomology of configuration spaces~\cite{Sinha:GD} to the description of $\calK_\HH^{m,n}$ that we gave in the previous subsection. This dual complex  $\HHHH^{m,n}$ was already described in our previous paper~\cite[Section~11]{AT}, but it was constructed using a slightly different  approach.
 The complex $\HHHH^{m,n}$ computes the rational cohomology $\HH^*(\Ebarmn,\Q)$. Applying~\eqref{eq:koszul_Q_config_homology}, \eqref{eq:one_pt_comp_isom_conf_sp}, \eqref{eq:compl_der_maps_H}, \eqref{eq:compl_der_maps_pi} one gets that this dual complex has the form
$$
\HHHH^{m,n}=\left(\bigoplus_k\overline{\HH}_{-*}(\mathrm{C}(k,\R^m),\Q)\otimes_{\Sigma_k}\hat\HH{}^*(\mathrm{C}(k,\R^n),\Q),d\right).
\eqno(\numb)\label{eq:bic_gr}
$$
To recall Subsections~\ref{sss:operads_examples}-\ref{sss:inf_bimod}, $\hat\HH{}^*(\mathrm{C}(k,\R^n),\Q)$ and $\HH{}^*(\mathrm{C}(k,\R^m),\Q)$ are described as certain spaces of forests modulo Arnol'd relations. Finally applying Poincar\'e duality
$$
\overline{\HH}_{*}(\mathrm{C}(k,\R^m),\Q)\simeq \HH^{mn-*}(\mathrm{C}(k,\R^m),\Q)\otimes(\sign_k)^{\otimes m}
\eqno(\numb)\label{eq:poincare2}
$$
we can describe $\HHHH^{m,n}$ as a complex of graphs that have two types of edges: dotted and full. The dotted edges correspond to the generators of $\HH^*(\mathrm{C}(\bullet,\R^m),\Q)$ and the full edges correspond to the generators of $\HH^*(\mathrm{C}(\bullet,\R^n),\Q)$. The vertices of the graphs are non-labeled (this corresponds to the fact that in~\eqref{eq:bic_gr} the tensor product is taken over $\Sigma_k$). There are two restrictions on the graphs. If we remove the dotted edges the resulting graph is a forest whose all connected components have at least two vertices. If we remove the full edges the resulting graph is a forest with any type of connected components. The number of full edges in a graph is its {\it complexity}, the number of connected components obtained by removing full edges is its {\it Hodge degree}. The space of graphs is taken modulo Arnol'd relations with respect to both types of edges. The differential is the sum of contractions of dotted edges. There are two equivalent ways two define the {\it orientation set} of a graph. In the first way, see~\eqref{eq:bic_gr},~\eqref{eq:poincare2}, it consists of
\begin{itemize}
\item full edges (of degree $(n-1)$)
\item dotted edges (of degree $(m-1)$)
\item vertices (of degree $-m$)
\end{itemize}
It is easy to see that equivalently we can define an orientation set as a union of
\begin{itemize}
\item full edges (of degree $(n-1)$)
\item dotted edges (of degree $-1$)
\item connected components with respect to dotted edges (of degree $-m$)
\end{itemize}
The edges are oriented. Changing orientation of an edge produces the sign $(-1)^n$ in the case of a full edge and $(-1)^m$ in the case of a dotted edge (assuming we choose the first way to define an orientation set, otherwise there is no sign). The latter way is more natural to the operadic approach, see Subsections~\ref{ss:proj_cofibr_model},~\ref{ss:compl_der_maps} in which we describe $KQ^m_s=Q^m_s\circ \coLie[1]$.

The complex $\HHHH^{m,n}$ naturally carries a structure of a polynomial bialgebra (with respect to the operation of connected sum). The space of generators is given by its subcomplex $\HHHH^{m,n}_\pi$ of connected graphs. Obviously $\HHHH^{m,n}_\pi$ computes the dual of the rational homotopy  of $\Ebarmn$. The table below describes the homology generators of $\HHHH^{m,n}_\pi$ in complexities  $\leq 2$ and the corresponding to them generators of $\calE^{m,n}_\pi$.

\begin{center}
\begin{tabular}{|c|c|c|c|c|c|}
\hline
\raisebox{2.5mm}{$\HHHH^{m,n}_\pi$}&
\begin{picture}(30,25)
\put(0,2){\line(2,1){30}}
\end{picture}&
\includegraphics[width=1cm]{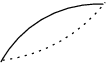}&
\includegraphics[width=1.2cm]{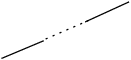}&
\includegraphics[width=2cm]{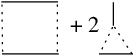}&
\includegraphics[width=1.2cm]{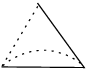}\\
\hline
\raisebox{2.5mm}{$\calE^{m,n}_\pi$}&
\begin{picture}(30,25)
\put(0,2){\line(2,1){30}}
\end{picture}&
\begin{picture}(30,25)
\put(0,0){\line(2,1){15}}
\put(21,11.45){\circle{15}}
\end{picture}&
\includegraphics[width=1.2cm]{complexity2_even}&
\includegraphics[width=1.8cm]{complexity2_odd_1}&
\includegraphics[width=1.2cm]{complexity2_odd_2}\\
\hline
\end{tabular}

\vspace{.2cm}

\end{center}

As we mentioned earlier all our graph-complexes look very similar to the graph-complexes that appear in the Bott-Taubes type integration construction for spaces of long embeddings~\cite{CattRossi,Sakai,SakW,Wat}. The latter construction produces a map from a certain graph-complex to the de Rham complex of differential forms on $\Embmn$. Unfortunately for $m\geq 2$ this construction works nicely  only on the level of graphs with no more than  one cycle. As example if we look at the classes from the previous table, only the third and the forth ones are proper to $\Embmn$, see Theorem~\ref{t:connecting_hom}. Both classes are recovered by Sakai in his construction~\cite{Sakai}. The third class from our table is the Haefliger class. It is the only class proper to $\Embmn$ and appearing on the level of graphs without loops. The corresponding cycle in Sakai's graph-complex is

\begin{center}
\psfrag{13}[0][0][1][0]{$\frac 13$}
\includegraphics[width=5cm]{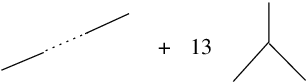}.
\end{center}
It is quite interesting that both graphs appear in our computations but in two different complexes. As another example, most of  the wheel-type cocycles considered in Proposition~\ref{p:wheels} are constructed in~\cite{SakW}.

\subsection{Complex of deformations}\label{ss:operadic_interp}
It turns out that the complex $\calK^{m,n}_\HH$ has a natural interpretation from the point of view of the deformation theory of operads. This theory was initiated by Kontsevich and Soibelman~\cite{Kontsevich} and was further developed by Merkulov and Vallette~\cite{MerkVal}. Given $m<n$, one has inclusion of the operads of little discs:
$$
\balls_m\stackrel{i}{\hookrightarrow}\balls_n.
$$
inducing morphism in the homology
$$
\HH_*(\balls_m,\Q)\stackrel{i_*}{\longrightarrow}\HH_*(\balls_n,\Q).
\eqno(\numb)\label{eq:mor_hom_balls}
$$
The operad $\HH_*(\balls_m,\Q)$ is the associative operad $\Assoc$ if $m=1$ and is the operad $\Poiss_{m-1}$ of graded Poisson algebras with bracket $[x_1,x_2]$ of degree $(m-1)$ and commutative product $x_1x_2$ of degree~0. The map~\eqref{eq:mor_hom_balls} sends the product to the product and the bracket to zero. We notice that $i_*$ factors through the commutative operad.

\begin{theorem}\label{th:def_operads}
For $n\geq 2m+2$, one has
$$
\chains^\Q\Ebarmn\simeq \Sigma^{m+1}\mathrm{Def}\left(\HH_*(\balls_m,\Q)\stackrel{i_*}{\rightarrow}\HH_*(\balls_n,\Q))\right).
$$
\end{theorem}

In the above $\mathrm{Def}(\bullet\to\bullet)$ states for a deformation complex of a morphism of operads as defined in~\cite{Kontsevich,MerkVal}. For simplicity the operads $\balls_m$, $\balls_n$ in the above theorem are taken without the degree zero component.\footnote{We believe that the result is still true even if one includes the degree zero components, but technically it is more complicated, since one has to find a cofibrant model of the unital Poisson and associative operads.}  The case $m=1$ of this theorem is well known. Indeed, the complex $\mathrm{Def}\left(\Assoc\stackrel{i_*}{\longrightarrow}\HH_*(\balls_m,\Q)\right)$ is the usual Hochschild complex of the operad $\HH_*(\balls_n,\Q)=\Poiss_{n-1}$, see~\cite{Kontsevich,Tur-HSLK,LTV}. For $m\geq 2$ this relation between the homology of higher dimensional long knots and the deformation homology of the morphism~\eqref{eq:mor_hom_balls} was conjectured earlier by Kontsevich.

We should also mention that the deformation complex of the map~\eqref{eq:mor_hom_balls} in case $m=n$, i.e. when $i_*$ is the identity map, was studied by Willwacher in~\cite{Willwacher}. In the latter work  several graph-complexes similar to those appearing in our paper are produced. Moreover in~\cite[Section~4]{Willwacher}, the author considers a filtration in one of his graph-complexes whose associated graded quotient up to a regrading is quasi-isomorphic to our deformation complex (assuming $m$ and $n$ are of the same parity), see \cite[Remark~4.7]{Willwacher}. One of the interesting consequences of Willwacher's work for us is that in case $n$ is even, the Hodge degree one part of the rational homotopy of $\Ebarmn$ always contains the non-completed  Grothendieck-Teichm\"uller Lie algebra, see \cite[Proposition~7.2]{Willwacher}. Since the Hodge degree one part of the complex $\calEp$ depends only on the parity of $n$ the latter result holds for either parity of $m$, but of course always assuming the stable range $n\geq 2m+2$. Besides that a careful reader might also find several overlaps between our paper and the one of Willwacher~\cite{Willwacher} in the way of working with graph-complexes.

Let $\Poiss_{m-1}^\infty$ be a cofibrant model of $\Poiss_{m-1}$, that is a quasi-free operad quasi-isomorphic to $\Poiss_{m-1}$ via a projection
$$
\Poiss_{m-1}^\infty\stackrel{p}{\longrightarrow}\Poiss_{m-1}.
$$
The operad $\Poiss_{m-1}$ is known to be Koszul~\cite{GetzJon}. Its Koszul dual $\Poiss_{m-1}^{!}$ is the operad of graded Poisson algebras with bracket of degree zero and commutative product of degree $(m-1)$. The cofibrant model $\Poiss_{m-1}^\infty$ can be chosen as the quasi-free operad ${\mathcal F}(\Sigma^{-1}\Poiss_{m-1}^{\text{!`}}[1])$ generated by
the symmetric sequence $\Sigma^{-1}\Poiss_{m-1}^\text{!`}[1]$, where $\Poiss_{m-1}^\text{!`}[1]$ is the cooperad whose dual is the operad $\Poiss_{m-1}^![1]$ of graded Poisson algebras with bracket of degree~1 and commutative product of degree $m$,\footnote{The operad $\Poiss_{m-1}^![1]$ is obtained from $\Poiss_{m-1}^!$ by an operadic suspension, see Subsection~\ref{sss:operads_monoids}.} and $\Sigma^{-1}$ is the desuspension of each component in the symmetric sequence. The differential in ${\mathcal F}(\Sigma^{-1}\Poiss_{m-1}^\text{!`}[1])$ comes from the cooperadic structure of $\Poiss_{m-1}^\text{!`}[1]$. We refer to~\cite{MerkVal} for a better account on the deformation theory of Koszul objects. The space of the complex $\mathrm{Def}\left(\Poiss_{m-1}\stackrel{i_*}{\rightarrow}\Poiss_{n-1}\right)$ is the space of infinitesimal deformations (derivations) of the morphism of operads:
$$
{\mathcal F}(\Sigma^{-1}\Poiss_{m-1}^\text{!`}[1])\stackrel{p\circ i_*}{\longrightarrow}\Poiss_{n-1},
$$
that is
$$
\mathrm{Def}\left(\Poiss_{m-1}\stackrel{i_*}{\rightarrow}\Poiss_{n-1}\right)=\bigoplus_{k\geq 2}\mathrm{hom}_{\Sigma_k}\left(\Sigma^{-1}\Poiss_{m-1}^\text{!`}[1](k),\Poiss_{n-1}(k)\right).
\eqno(\numb)\label{eq:deform_expl}
$$
One has a natural isomorphism of graded $\Sigma_k$-modules
$$
\widetilde{\HH}_*(S^{mk}/\Delta^kS^m,\Q)\simeq \Sigma^m\Poiss_{m-1}^\text{!`}[1](k),
$$
 that one can see for example from the forest description of $\widetilde{\HH}_*(S^{mk}/\Delta^kS^m,\Q)=\overline{\HH}_*(\mathrm{C}(k,\R^m),\Q)$ given above and a similar forest description of the cooperad of Poisson algebras given in~\cite{Sinha:GD} (here $\overline\HH$ denotes singular locally finite homology). The differential in \eqref{eq:deform_expl} is the pre-Lie commutator with $\mu=x_1x_2$, see~\cite{MerkVal}:\footnote{Here we use the fact that $i_*$ sends $[x_1,x_2]$ to zero.}
$$
df=\mu\circ f-(-1)^{|f|-1}f\circ \mu.
$$
The second summand is exactly the right-hand side of~\eqref{eq:differential_koszul}. The first summand can be written as follows
$$
(\mu\circ f)(T)=\sum_v(-1)^\delta x_v\cdot f(T\setminus v),
\eqno(\numb)\label{eq:m_circ_f}
$$
where the sum is taken over the univalent vertices of the forest $T$. The sign $(-1)^\delta$ is  the Koszul sign of permutation taking the only edge adjacent to $v$ on the first place of the orientation set of $T$, and $T\setminus v$ denotes the forest obtained from $T$ by removing the vertex $v$ and the edge adjacent to it. The total complex~\eqref{eq:deform_expl} can be written as
$$
\bigoplus_{k\geq 2}\Sigma^{m+1}\mathrm{hom}_{\Sigma_k}\bigl(\overline{\HH}_*(\mathrm{C}(k,\R^m),\Q),\HH_*(\mathrm{C}(k,\R^n),\Q)\bigr).
$$
This complex contains an acyclic complex spanned by degeneracies. The quotient-complex is exactly  the (shifted by $m+1$ gradings) complex $\calK^{m,n}_\HH$:
$$
\bigoplus_{k\geq 2}\Sigma^{m+1}\mathrm{hom}_{\Sigma_k}\left(\overline{\HH}_*(\mathrm{C}(k,\R^m),\Q),\hat\HH_*(\mathrm{C}(k,\R^n),\Q)\right),
\eqno(\numb)\label{eq:unital_def}
$$
computing the reduced rational homology of $\Ebarmn$. This  is straightforward and similar to the fact that the Hochschild complex of the Poisson operad is a direct sum of its normalized subcomplex and an acyclic one spanned by degeneracies.
The complex~\eqref{eq:unital_def} can also be interpreted as the complex of unital deformations of the morphism $\Poiss_{m-1}\stackrel{i_*}{\longrightarrow}\Poiss_{n-1}$. By {\it unital} we mean deformations
preserving the zero-ary operation $1\!\! 1\in\Poiss_{m-1}(0)$ (the deformation of higher operations are annihilated when applied to $1\!\! 1$) if the latter one is added to the operad $\Poiss_{m-1}$.

\section{Euler characteristics of the double splitting}\label{s:Euler}
\sloppy
In this section we investigate the generating function of the Euler characteristics of the double splitting in the rational homology of $\Ebarmn$. Let $\chi_{st}$ denote the Euler characteristic of the summand $\calK^{m,n}_\HH(s,t)$ of complexity~$t$ and Hodge degree $s$, see~\eqref{eq:compl_der_maps_H}. This summand is finite dimensional, so the homology should also be finite dimensional and the Euler characteristic is well defined. Let
$$
F_{mn}(x,u)=\sum_{st}\chi_{st}x^su^t
\eqno(\numb)\label{eq:Fmn}
$$
denote the generating function of the Euler characteristics of the double splitting in $\HH_*(\Ebarmn;\Q)$. It is clear that this function depends only on the parities of $m$ and $n$.

\begin{theorem}\label{t:gen_funct}
One has
$$
F_{mn}(x,u)=\prod_{\ell\geq 1}\frac{\Gamma((-1)^{n-1}E_\ell(\frac 1u)-(-1)^{m-1}E_\ell(x))}{((-1)^{n-1}\ell u^\ell)^{(-1)^{m-1}E_\ell(x)}\Gamma((-1)^{n-1}E_\ell(\frac 1u))},
\eqno(\numb)\label{eq132}
$$
where each factor in the product is understood as the asymptotic expansion of the underlying function when $u$ is  complex  and $(-1)^{n-1}u^\ell\to +0$ and  $x$ is considered as a fixed parameter. In the above $\Gamma(y)$ is the gamma function which is $(y-1)!$ on positive integers, $E_\ell(y)=\frac 1\ell\sum_{d|\ell}\mu(d)y^{\ell/d}$ (where $\mu(-)$ is the standard M\"obius function).
\end{theorem}

Tables~\ref{table:H_oo},~\ref{table:H_oe},~\ref{table:H_eo},~\ref{table:H_ee} describe $\chi_{st}$ for complexities $t\leq 23$. From these tables the above generating functions start as follows:
\begin{align*}
F_{oo}(x,u)&=1+x^2u+(x^4+x^2-x)u^2+(x^6+x^4-x^3+x^2-x)u^3+\ldots\\
F_{oe}(x,u)&=1-xu+x^3u^2+(-x^4-x^2+x)u^3+\ldots\\
F_{eo}(x,u)&=1+(-x^3+x)u^2+(-x^3+x)u^3+\ldots\\
F_{ee}(x,u)&=1+(-x^2+x)u+(-x^3+x^2)u^2+(-x^4+2x^3-x)u^3+\ldots
\end{align*}
The subscript $o$ refers to the case when the corresponding variable $m$ or $n$ is odd, and the subscript $e$ refers to the situation when $m$ or $n$ is even.

The case $m=1$ of this formula was  proved in~\cite{Turchin}. We refer the reader to this paper which explains this formula in more details. In the rest of the section we sketch the main steps of the computations.

From Subsection~\ref{ss:compl_der_maps} the complex computing $\HH_*(\Ebarmn,\Q)$ in complexity $t$ and Hodge degree $s$ is
$$
\left(\bigoplus_{k=s}^{2t}\mathrm{hom}_{\Sigma_{k}}\left(KQ_s^m(k),\hat\HH_{t(n-1)}(\mathrm{C}(k,\R^n);\Q)\right),\, \partial\right).
$$

To recall \eqref{eq:koszul_Q_config_homology}, the $\Sigma_{k}$-module $KQ_s^m(k)$ is $\widetilde{\HH}_{k+s(m-1)}(S^{mk}/\Delta^kS^m;\Q)=\overline{\HH}_{k+s(m-1)}(\mathrm{C}(k,\R^m);\Q)$ (assuming $m\geq 2$), where $\overline{\HH}$ denotes the locally compact singular homology. The entire complex computing $\HH_*(\Ebarmn;\Q)$ is
$$
\left(\bigoplus_{k\geq 0}\mathrm{hom}_{\Sigma_k}\left(\overline{\HH}_*(\mathrm{C}(k,\R^m);\Q),\hat\HH_*(\mathrm{C}(k,\R^n),\Q)\right),\, \partial\right).
$$
The above formula is also true for $m=1$, but in that case the splitting over the Hodge degree can not be defined in the same way. For $m=1$ this splitting is similar to the Hodge-type splitting in the Hochschild (co)homology of a commutative algebra~\cite{Turchin}. As we mentioned earlier the case $m=1$ of Theorem~\ref{t:gen_funct} was done in~\cite[Theorem~13.1]{Turchin}. For simplicity of exposition we will be assuming below that $m\geq 2$.

We start by introducing some standard notation. For each permutation $\sigma\in \Sigma_k$ define $Z(\sigma)$, the {\it cycle indicator} of $\sigma$, by
$$
Z(\sigma)=\prod_\ell a_\ell^{j_\ell(\sigma)},
$$
where $j_\ell(\sigma)$  is the number of $\ell$-cycles of $\sigma$ and where $a_1$, $a_2$, $a_3$, $\ldots$ is an infinite family of commuting variables.

Let $\rho^V\colon S_k\to GL(V)$ be a representation of $\Sigma_k$. Define $Z_V(a_1,a_2,\ldots)$ the {\it cycle index} of $V$, by
$$
Z_V(a_1,a_2,\ldots)=\frac 1{n!}\sum_{\sigma\in \Sigma_k}\tr \rho^V(\sigma)\cdot Z(\sigma).
$$
Similarly for a symmetric sequence $W=\{W(k),k\geq 0\}$ --- sequence of $\Sigma_k$-modules, one defines its {\it cycle index sum} $Z_W$ by
$$
Z_W(a_1,a_2,\ldots)=\sum_{k=0}^{+\infty}Z_{W(k)}(a_1,a_2,\ldots).
$$

The following lemma is well known~\cite[Lemma~15.4]{Turchin}.

\begin{lemma}\label{l:gener_hom}
Let $V$ and $W$ be two $\Sigma_k$-modules, then
$$
\dim {\mathrm{hom}}_{\Sigma_k}(V,W)=\bigl(Z_V(a_\ell\leftarrow\partial/{\partial a_\ell},\,\ell\in\BN)\,\,Z_W(a_\ell\leftarrow \ell a_\ell,\,\ell\in\BN)\bigr)\Bigr|_{\substack{a_\ell=0,\\ \ell\in\BN}}.
$$
\end{lemma}

In the above formula $Z_W(a_\ell\leftarrow \ell a_\ell,\,\ell\in\BN)$ is a polynomial obtained from $Z_W(a_1,a_2,\ldots)$ by replacing each variable $a_\ell$ by $\ell a_\ell$. The expression $Z_V(a_\ell\leftarrow\partial/{\partial a_\ell},\,\ell\in\BN)$ is a differential operator obtained from the polynomial $Z_V(a_1,a_2,\ldots)$ by replacing each $a_\ell$ by $\partial/\partial a_\ell$. The differential operator is applied to the polynomial and at the end one takes all the variables $a_\ell,$ $\ell\in\BN$, to be zero.

In case $V=\bigoplus_iV_i$, $W=\bigoplus_iW_i$ are graded $S_k$-modules, and $\dim{\mathrm{hom}}_{\Sigma_k}(V,W)$ is the graded dimension:
$$
\dim{\mathrm{hom}}_{\Sigma_k}(V,W)=\sum_{i,j\in\BZ}\dim {\mathrm{hom}}_{\Sigma_k}(V_i,W_j)z^{j-i},
$$
and $Z_V$, $Z_W$ are {\it graded} cycle indices:
\begin{align*}
Z_V(z;a_1,a_2,\ldots)=&\sum_{i\in \BZ}Z_{V_i}(a_1,a_2,\ldots)z^i,\\
Z_W(z;a_1,a_2,\ldots)=&\sum_{i\in \BZ}Z_{W_i}(a_1,a_2,\ldots)z^i.
\end{align*}
Then
$$
\dim\,{\mathrm{hom}}_{\Sigma_k}(V,W)=Z_V(1/z;a_\ell\leftarrow\partial/{\partial a_\ell},\,\ell\in\BN)\,\,Z_W(z;a_\ell\leftarrow \ell a_\ell,\,\ell\in\BN)\Bigl|_{\substack{a_\ell=0,\\ \ell\in\BN}}.
\eqno(\numb)\label{eq:gr_dim}
$$

 We will be considering bigraded symmetric sequences. Similarly in our computations we will add one more variable $x$ or $u$ responsible for the second grading.

Next step is to find the graded cycle index sum of the symmetric sequences $\overline{\HH}_*(\mathrm{C}(\bullet,\R^m);\Q)$, and $\hat\HH_*(\balls_n(\bullet),\Q)$.
The symmetric group action on the homology of configuration spaces $\mathrm{C}(k,\R^n)$ is well understood~\cite{CT:RCCS,Lehrer,LO:ASGCH}.

\begin{proposition}\label{p:Sn_conf_spaces}
The graded cycle index sum for the symmetric sequence
$$
\HH_*(\mathrm{C}(\bullet,\R^n),\Q)=\left\{\HH_*(\mathrm{C}(k,\R^n),\Q)|\, k\geq 0\right\}
$$
is given by the following formula:
$$
Z_{\HH_*(\mathrm{C}(\bullet,\R^n),\Q)}(z;a_1,a_2,\ldots)=\prod_{\ell=1}^{+\infty}\left(1+(-1)^{n}(-z)^{(n-1)\ell}a_\ell\right)^{(-1)^{n}E_\ell\left(\frac 1{(-z)^{n-1}}\right)},
\eqno(\numb)\label{eq_Sn_conf_spaces}
$$
where $E_\ell(y)=\frac 1\ell\sum_{d|\ell}\mu(d)y^{\frac \ell d}.$
\end{proposition}

\begin{proof} It is an easy consequence of~\cite[Theorem~B]{Lehrer}.
\end{proof}

Let us add another variable $u$ that will be responsible for the complexity, which is the homology degree divided by $(n-1)$. Thus the $u$-degree is the $z$-degree divided by $(n-1)$:
$$
Z_{\HH_*(\mathrm{C}(\bullet,\R^n),\Q)}(z,u;a_1,a_2,\ldots)=\prod_{\ell=1}^{+\infty}\left(1+(-1)^{n}((-z)^{(n-1)}u)^\ell a_\ell\right)^{(-1)^{n}E_\ell\left(\frac 1{(-z)^{n-1}u}\right)}.
\eqno(\numb)\label{eq_Sn_conf_spaces_complex}
$$

The following result is easily obtained from~\eqref{eq_Sn_conf_spaces_complex}, see~\cite[Section~15.2]{Turchin}.

\begin{proposition}[\cite{Turchin}]\label{p:Sn_conf_spaces_norm}
One has
\begin{equation}
Z_{\hat\HH_*(\mathrm{C}(\bullet,\R^n),\Q)}(z,u;a_1,a_2,\ldots)=\prod_{\ell=1}^{+\infty}e^{-\frac {a_\ell}\ell}\left(1+(-1)^{n}((-z)^{(n-1)}u)^\ell a_\ell\right)^{(-1)^{n}E_\ell\left(\frac 1{(-z)^{n-1}u}\right)},
\label{eq:Sn_conf_spaces_norm}
\end{equation}
where the variable $u$ is responsible for the complexity and $z$ for the homology degree.
\end{proposition}

\begin{proposition}\label{p:Sn_conf_spaces_comp}
One has
\begin{equation}
 Z_{\overline{\HH}_*(\mathrm{C}(\bullet,\R^m),\Q)}(z,x;\,a_\ell,\,\ell\in\BN)=\prod_{\ell=1}^{+\infty}(1+(-z)^{\ell}a_\ell)^{(-1)^{m}
 E_\ell((-z)^{m-1}x)},
 \label{eq:Sn_conf_spaces_comp}
\end{equation}
where the variable $x$ is responsible for the Hodge degree and $z$ is responsible for the homology degree.
\end{proposition}
\begin{proof}
From the Poincar\'e duality one has an isomorphism of $\Sigma_\ell$-modules:
$$
\overline{\HH}_*(\mathrm{C}(\ell,\R^m),\Q)\simeq\HH^{\ell m-*}(\mathrm{C}(\ell,\R^m),\Q)\otimes (\sign_\ell)^{\otimes m}.
\eqno(\numb)\label{eq:poincare}
$$
The sign representation appears due to the fact that the Poincar\'e duality uses the orientation of the variety. From the above one has
$$
\overline{\HH}_{i(m-1)+\ell}(\mathrm{C}(\ell,\R^m),\Q)\simeq\HH^{(\ell m-i)(m-1)}(\mathrm{C}(\ell,\R^m),\Q)\otimes (\sign_\ell)^{\otimes m}.
$$
So, the part of $\HH_*(\mathrm{C}(\ell,\R^m),\Q)$ lying in Hodge degree $s$ corresponds to the part of $\HH^*(\mathrm{C}(\ell,\R^m),\Q)$ lying in complexity $(\ell-i)$.

It follows from this, that
\begin{multline}
Z_{\overline{\HH}_*(\mathrm{C}(\bullet,\R^m),\Q)}(z,x;\,a_\ell,\,\ell\in\BN)=\\
Z_{\HH_*(\mathrm{C}(\bullet,\R^m),\Q)}(z\leftarrow 1/z;u\leftarrow 1/x;\, a_\ell\leftarrow (-1)^{(\ell-1)m}x^\ell z^{m\ell}a_\ell,\,\ell\in\BN).
\label{eq:change_of_var}
\end{multline}
Indeed, the sign $(-1)^{(\ell-1)m}$ arises because of the factor $(\sign_\ell)^{\otimes m}$ in~\eqref{eq:poincare}. The isomorphism~\eqref{eq:poincare} sends the complexity
$t$ to the Hodge degree $s=\ell-t$, which explains why $u$ is replaced by $1/x$, and also the presence of the factor $x^\ell$ in $a_\ell\leftarrow (-1)^{(\ell-1)m}x^\ell z^{m\ell}a_\ell$. The total homological degree $*$ of the left-hand side corresponds to $(m\ell-*)$ of the right-hand side in~\eqref{eq:poincare}: this explains why $z$ is replaced by $1/z$, and also the presence of the factor $z^{m\ell}$ in $a_\ell\leftarrow (-1)^{(\ell-1)m}x^\ell z^{m\ell}a_\ell$.

The equation~\eqref{eq:Sn_conf_spaces_comp} follows immediately from~\eqref{eq:change_of_var} and~\eqref{eq_Sn_conf_spaces}.
\end{proof}

\begin{proof}[Proof of Theorem~\ref{t:gen_funct}]
Define $\Psi_{mn}(x,u,z)$ as the following generating function:
\begin{multline*}
\Psi_{mn}(x,u,z)=\\
\sum_{s,t,k}\dim\left(\mathrm{hom}_{\Sigma_k}\left(\overline{\HH}_{s(m-1)+k}(\mathrm{C}(k,\BR^m);\Q),\hat\HH_{t(n-1)}(\mathrm{C}(k,\BR^n);\Q)\right)\right)
x^su^tz^{t(n-1)-s(m-1)-k}.
\end{multline*}
The variable $x$ is responsible for the Hodge degree; the variable $u$ is responsible for the complexity; and the variable $z$ for the total homological degree. %Notice that from the $x$, $u$, and $z$ degrees we can capture the number of points $k$, therefore we do not need an additional variable responsible for $k$.
It follows from~\eqref{eq:gr_dim}, and Propositions~\ref{p:Sn_conf_spaces_norm}-\ref{p:Sn_conf_spaces_comp} that
\begin{multline*}
\Psi_{mn}(x,u,z)=\\
\scriptstyle
\left.\left(\prod_{\ell=1}^{+\infty}\left(1+(-1/z)^{\ell}\partial/\partial a_\ell\right)^{(-1)^{m}E_\ell\left(\frac x{(-z)^{m-1}}\right)}
\prod_{\ell=1}^{+\infty}e^{-a_\ell}\left(1+(-1)^{n}\ell((-z)^{(n-1)}u)^\ell a_\ell\right)^{(-1)^{n}E_\ell\left(\frac 1{(-z)^{n-1}u}\right)}\right)\right|_{\substack{a_\ell=0\\ \ell\in\BN}}=\\
\scriptstyle
\prod_{\ell=1}^{+\infty}\left.\left(\left(1+(-1/z)^{\ell}\partial/\partial a_\ell \right)^{(-1)^{m}E_\ell\left(\frac x{(-z)^{m-1}}\right)}
e^{-a_\ell}\left(1+(-1)^{n}\ell((-z)^{(n-1)}u)^\ell a_\ell\right)^{(-1)^{n}E_\ell\left(\frac 1{(-z)^{n-1}u}\right)}\right)\right|_{a_\ell=0}.
\end{multline*}

Since $F_{mn}(x,u)=\Phi_{mn}(x,u,-1)$ we get
$$
F_{mn}(x,u)=
\prod_{\ell=1}^{+\infty}\left.\left(\left(1+\partial/\partial a \right)^{(-1)^{m}E_\ell(x)}
e^{-a}\left(1+(-1)^{n}\ell u^\ell a\right)^{(-1)^{n}E_\ell\left(\frac 1{u}\right)}\right)\right|_{a=0}.
$$
Notice that in the above formula we replaced $a_\ell$ by $a$. We could do so because each factor uses only one variable $a_\ell$ which is anyway taken to be zero.

The rest follows from the formula
\begin{multline*}
\left.\left(\left(1+\partial/\partial a \right)^{(-1)^{m}E_\ell(x)}
e^{-a}\left(1+(-1)^{n}\ell u^\ell a\right)^{(-1)^{n}E_\ell\left(\frac 1{u}\right)}\right)\right|_{a=0}=\\
=\frac{\Gamma((-1)^{n-1}E_\ell(\frac 1u)-(-1)^{m-1}E_\ell(x))}{\bigl( (-1)^{n-1} \ell u^\ell\bigr)^{(-1)^{m-1}E_\ell(x)}\Gamma((-1)^{n-1}E_\ell(\frac 1u))},
\end{multline*}
which is proved in~\cite[Proposition~15.7]{Turchin}.

\end{proof}

\subsection*{Acknowledgement} The authors are grateful to P.~Lambrechts and the Universit\'e Catholique de Louvain for hospitality. It was actually at the UCL where the idea for this paper first appeared.

%\newpage

\renewcommand{\thetable}{\arabic{table}}
\setcounter{table}{0}

\appendix

\section{Tables of Euler characteristics}

We present below results of computer calculations which were produced using Maple. These results appeared partially in~\cite{Turchin}. We add here the case of even~$m$. For completeness of presentation we keep the case of odd $m$ as well. To recall $\chi_{st}$ denotes the Euler characteristic of $\HH_*(\Ebarmn,\Q)$ in complexity $t$ and Hodge degree $s$. Let $\chi^\pi_{st}$ denote the Euler characteristic of an analogous component of $\Q\otimes\pi_*\Ebarmn$.

\begin{lemma}\label{l:euler_homot}
$$
F_{mn}(x,u)=\sum_{st}\chi_{st}x^su^t=\prod_{st}\frac 1{(1-x^su^t)^{\chi_{st}^\pi}}.
\eqno(\numb)\label{eq:euler_homot}
$$
\end{lemma}

\begin{proof}
See \cite[Lemma~16.1]{Turchin}.
\end{proof}

This formula was used to fill Tables~\ref{table:pi_oo},~\ref{table:pi_oe},~\ref{table:pi_eo},~\ref{table:pi_ee}. The last column \lq\lq total" stays for the sum of absolute values of the Euler characteristics of the terms in a given complexity. This gives a lower bound estimation for the rank of the rational homotopy in a given complexity.

\begin{table}[h!]

{\tiny \begin{center}
\begin{tabular}{|c|c|c|c|c|c|c|c|c|c|c|c|c|c|c|c|c|c|c|c|c|c|c|c|c|}
\hline
$t$&\multicolumn{23}{|c|}{Hodge degree $s$}&\\
\hline
 & 1 & 2 & 3 & 4 & 5 & 6 & 7 & 8 & 9 & 10 & 11 & 12 & 13 & 14 & 15 & 16 & 17 & 18 & 19 & 20 & 21 & 22 & 23 & total \\
\hline
1 & \begin{turn}{80}{0}\end{turn} & \begin{turn}{80}{1}\end{turn} & \begin{turn}{80}{0}\end{turn} & \begin{turn}{80}{0}\end{turn} & \begin{turn}{80}{0}\end{turn} & \begin{turn}{80}{0}\end{turn} & \begin{turn}{80}{0}\end{turn} & \begin{turn}{80}{0}\end{turn} & \begin{turn}{80}{0}\end{turn} & \begin{turn}{80}{0}\end{turn} & \begin{turn}{80}{0}\end{turn} & \begin{turn}{80}{0}\end{turn} & \begin{turn}{80}{0}\end{turn} & \begin{turn}{80}{0}\end{turn} & \begin{turn}{80}{0}\end{turn} & \begin{turn}{80}{0}\end{turn} & \begin{turn}{80}{0}\end{turn} & \begin{turn}{80}{0}\end{turn} & \begin{turn}{80}{0}\end{turn} & \begin{turn}{80}{0}\end{turn} & \begin{turn}{80}{0}\end{turn} & \begin{turn}{80}{0}\end{turn} & \begin{turn}{80}{0}\end{turn} & 1 \\
\hline
2 & \begin{turn}{80}{-1}\end{turn} & \begin{turn}{80}{1}\end{turn} & \begin{turn}{80}{0}\end{turn} & \begin{turn}{80}{0}\end{turn} & \begin{turn}{80}{0}\end{turn} & \begin{turn}{80}{0}\end{turn} & \begin{turn}{80}{0}\end{turn} & \begin{turn}{80}{0}\end{turn} & \begin{turn}{80}{0}\end{turn} & \begin{turn}{80}{0}\end{turn} & \begin{turn}{80}{0}\end{turn} & \begin{turn}{80}{0}\end{turn} & \begin{turn}{80}{0}\end{turn} & \begin{turn}{80}{0}\end{turn} & \begin{turn}{80}{0}\end{turn} & \begin{turn}{80}{0}\end{turn} & \begin{turn}{80}{0}\end{turn} & \begin{turn}{80}{0}\end{turn} & \begin{turn}{80}{0}\end{turn} & \begin{turn}{80}{0}\end{turn} & \begin{turn}{80}{0}\end{turn} & \begin{turn}{80}{0}\end{turn} & \begin{turn}{80}{0}\end{turn} & 2 \\
\hline
\end{tabular}
\end{center}}
\vspace{.4cm}
\caption{\small Table for Euler characteristics $\chi_{st}^\pi$ by complexity $t$ and Hodge degree $s$ of $\pi_*\Ebarmn\otimes\BQ$ for both $m$ and $n$ odd.}
\label{table:pi_oo}
\end{table}

\begin{table}[h!]

\begin{center}
\tiny
\begin{tabular}{|c|c|c|c|c|c|c|c|c|c|c|c|c|c|c|c|c|c|c|c|c|c|c|c|}\hline $t$&\multicolumn{23}{|c|}{Hodge degree $s$}\\
\hline
 & 1 & 2 & 3 & 4 & 5 & 6 & 7 & 8 & 9 & 10 & 11 & 12 & 13 & 14 & 15 & 16 & 17 & 18 & 19 & 20 & 21 & 22 & 23 \\
\hline
1 & \begin{turn}{80}{0}\end{turn} & \begin{turn}{80}{1}\end{turn} & \begin{turn}{80}{0}\end{turn} & \begin{turn}{80}{0}\end{turn} & \begin{turn}{80}{0}\end{turn} & \begin{turn}{80}{0}\end{turn} & \begin{turn}{80}{0}\end{turn} & \begin{turn}{80}{0}\end{turn} & \begin{turn}{80}{0}\end{turn} & \begin{turn}{80}{0}\end{turn} & \begin{turn}{80}{0}\end{turn} & \begin{turn}{80}{0}\end{turn} & \begin{turn}{80}{0}\end{turn} & \begin{turn}{80}{0}\end{turn} & \begin{turn}{80}{0}\end{turn} & \begin{turn}{80}{0}\end{turn} & \begin{turn}{80}{0}\end{turn} & \begin{turn}{80}{0}\end{turn} & \begin{turn}{80}{0}\end{turn} & \begin{turn}{80}{0}\end{turn} & \begin{turn}{80}{0}\end{turn} & \begin{turn}{80}{0}\end{turn} & \begin{turn}{80}{0}\end{turn}  \\
\hline
2 & \begin{turn}{80}{-1}\end{turn} & \begin{turn}{80}{1}\end{turn} & \begin{turn}{80}{0}\end{turn} & \begin{turn}{80}{1}\end{turn} & \begin{turn}{80}{0}\end{turn} & \begin{turn}{80}{0}\end{turn} & \begin{turn}{80}{0}\end{turn} & \begin{turn}{80}{0}\end{turn} & \begin{turn}{80}{0}\end{turn} & \begin{turn}{80}{0}\end{turn} & \begin{turn}{80}{0}\end{turn} & \begin{turn}{80}{0}\end{turn} & \begin{turn}{80}{0}\end{turn} & \begin{turn}{80}{0}\end{turn} & \begin{turn}{80}{0}\end{turn} & \begin{turn}{80}{0}\end{turn} & \begin{turn}{80}{0}\end{turn} & \begin{turn}{80}{0}\end{turn} & \begin{turn}{80}{0}\end{turn} & \begin{turn}{80}{0}\end{turn} & \begin{turn}{80}{0}\end{turn} & \begin{turn}{80}{0}\end{turn} & \begin{turn}{80}{0}\end{turn}  \\
\hline
\end{tabular}
\end{center}

\vspace{.4cm}
\caption{\small Table for Euler characteristics $\chi_{st}$ by complexity $t$ and Hodge degree $s$ of $\HH_*(\Ebarmn;\BQ)$ for both $m$ and $n$ odd.}
\label{table:H_oo}
\end{table}

\begin{table}[h!]

{\tiny \begin{center}
\begin{tabular}{|c|c|c|c|c|c|c|c|c|c|c|c|c|c|c|c|c|c|c|c|c|c|c|c|c|}\hline $t$&\multicolumn{23}{|c|}{Hodge degree $s$}&\\
\hline
 & 1 & 2 & 3 & 4 & 5 & 6 & 7 & 8 & 9 & 10 & 11 & 12 & 13 & 14 & 15 & 16 & 17 & 18 & 19 & 20 & 21 & 22 & 23 & total \\
\hline
1 & \begin{turn}{80}{-1}\end{turn} & \begin{turn}{80}{0}\end{turn} & \begin{turn}{80}{0}\end{turn} & \begin{turn}{80}{0}\end{turn} & \begin{turn}{80}{0}\end{turn} & \begin{turn}{80}{0}\end{turn} & \begin{turn}{80}{0}\end{turn} & \begin{turn}{80}{0}\end{turn} & \begin{turn}{80}{0}\end{turn} & \begin{turn}{80}{0}\end{turn} & \begin{turn}{80}{0}\end{turn} & \begin{turn}{80}{0}\end{turn} & \begin{turn}{80}{0}\end{turn} & \begin{turn}{80}{0}\end{turn} & \begin{turn}{80}{0}\end{turn} & \begin{turn}{80}{0}\end{turn} & \begin{turn}{80}{0}\end{turn} & \begin{turn}{80}{0}\end{turn} & \begin{turn}{80}{0}\end{turn} & \begin{turn}{80}{0}\end{turn} & \begin{turn}{80}{0}\end{turn} & \begin{turn}{80}{0}\end{turn} & \begin{turn}{80}{0}\end{turn} & 1 \\
\hline
2 & \begin{turn}{80}{0}\end{turn} & \begin{turn}{80}{0}\end{turn} & \begin{turn}{80}{1}\end{turn} & \begin{turn}{80}{0}\end{turn} & \begin{turn}{80}{0}\end{turn} & \begin{turn}{80}{0}\end{turn} & \begin{turn}{80}{0}\end{turn} & \begin{turn}{80}{0}\end{turn} & \begin{turn}{80}{0}\end{turn} & \begin{turn}{80}{0}\end{turn} & \begin{turn}{80}{0}\end{turn} & \begin{turn}{80}{0}\end{turn} & \begin{turn}{80}{0}\end{turn} & \begin{turn}{80}{0}\end{turn} & \begin{turn}{80}{0}\end{turn} & \begin{turn}{80}{0}\end{turn} & \begin{turn}{80}{0}\end{turn} & \begin{turn}{80}{0}\end{turn} & \begin{turn}{80}{0}\end{turn} & \begin{turn}{80}{0}\end{turn} & \begin{turn}{80}{0}\end{turn} & \begin{turn}{80}{0}\end{turn} & \begin{turn}{80}{0}\end{turn} & 1 \\
\hline
\end{tabular}
\end{center}}
\vspace{.4cm}
\caption{\small Table of Euler characteristics $\chi_{st}^\pi$ by complexity $t$ and Hodge degree $s$ of $\pi_*\Ebarmn\otimes\BQ$ for $m$ odd and $n$ even.}
\label{table:pi_oe}
\end{table}

\begin{table}[h!]

\begin{center}
\tiny
\begin{tabular}{|c|c|c|c|c|c|c|c|c|c|c|c|c|c|c|c|c|c|c|c|c|c|c|c|}\hline $t$&\multicolumn{23}{|c|}{Hodge degree $s$}\\
\hline
 & 1 & 2 & 3 & 4 & 5 & 6 & 7 & 8 & 9 & 10 & 11 & 12 & 13 & 14 & 15 & 16 & 17 & 18 & 19 & 20 & 21 & 22 & 23 \\
\hline
1 & \begin{turn}{80}{-1}\end{turn} & \begin{turn}{80}{0}\end{turn} & \begin{turn}{80}{0}\end{turn} & \begin{turn}{80}{0}\end{turn} & \begin{turn}{80}{0}\end{turn} & \begin{turn}{80}{0}\end{turn} & \begin{turn}{80}{0}\end{turn} & \begin{turn}{80}{0}\end{turn} & \begin{turn}{80}{0}\end{turn} & \begin{turn}{80}{0}\end{turn} & \begin{turn}{80}{0}\end{turn} & \begin{turn}{80}{0}\end{turn} & \begin{turn}{80}{0}\end{turn} & \begin{turn}{80}{0}\end{turn} & \begin{turn}{80}{0}\end{turn} & \begin{turn}{80}{0}\end{turn} & \begin{turn}{80}{0}\end{turn} & \begin{turn}{80}{0}\end{turn} & \begin{turn}{80}{0}\end{turn} & \begin{turn}{80}{0}\end{turn} & \begin{turn}{80}{0}\end{turn} & \begin{turn}{80}{0}\end{turn} & \begin{turn}{80}{0}\end{turn}  \\
\hline
2 & \begin{turn}{80}{0}\end{turn} & \begin{turn}{80}{0}\end{turn} & \begin{turn}{80}{1}\end{turn} & \begin{turn}{80}{0}\end{turn} & \begin{turn}{80}{0}\end{turn} & \begin{turn}{80}{0}\end{turn} & \begin{turn}{80}{0}\end{turn} & \begin{turn}{80}{0}\end{turn} & \begin{turn}{80}{0}\end{turn} & \begin{turn}{80}{0}\end{turn} & \begin{turn}{80}{0}\end{turn} & \begin{turn}{80}{0}\end{turn} & \begin{turn}{80}{0}\end{turn} & \begin{turn}{80}{0}\end{turn} & \begin{turn}{80}{0}\end{turn} & \begin{turn}{80}{0}\end{turn} & \begin{turn}{80}{0}\end{turn} & \begin{turn}{80}{0}\end{turn} & \begin{turn}{80}{0}\end{turn} & \begin{turn}{80}{0}\end{turn} & \begin{turn}{80}{0}\end{turn} & \begin{turn}{80}{0}\end{turn} & \begin{turn}{80}{0}\end{turn}  \\
\hline
\end{tabular}
\end{center}

\vspace{.4cm}
\caption{\small Table for Euler characteristics $\chi_{st}$ by complexity $t$ and Hodge degree $s$ of $\HH_*(\Ebarmn;\BQ)$ for  $m$ odd  and $n$ even.}
\label{table:H_oe}
\end{table}

\begin{table}[h!]

\tiny \begin{center}
\begin{tabular}{|c|c|c|c|c|c|c|c|c|c|c|c|c|c|c|c|c|c|c|c|c|c|c|c|c|}
\hline $t$&\multicolumn{23}{|c|}{Hodge degree $s$}&\\
\hline
 & 1 & 2 & 3 & 4 & 5 & 6 & 7 & 8 & 9 & 10 & 11 & 12 & 13 & 14 & 15 & 16 & 17 & 18 & 19 & 20 & 21 & 22 & 23 & total \\
\hline
1 & \begin{turn}{80}{0}\end{turn} & \begin{turn}{80}{0}\end{turn} & \begin{turn}{80}{0}\end{turn} & \begin{turn}{80}{0}\end{turn} & \begin{turn}{80}{0}\end{turn} & \begin{turn}{80}{0}\end{turn} & \begin{turn}{80}{0}\end{turn} & \begin{turn}{80}{0}\end{turn} & \begin{turn}{80}{0}\end{turn} & \begin{turn}{80}{0}\end{turn} & \begin{turn}{80}{0}\end{turn} & \begin{turn}{80}{0}\end{turn} & \begin{turn}{80}{0}\end{turn} & \begin{turn}{80}{0}\end{turn} & \begin{turn}{80}{0}\end{turn} & \begin{turn}{80}{0}\end{turn} & \begin{turn}{80}{0}\end{turn} & \begin{turn}{80}{0}\end{turn} & \begin{turn}{80}{0}\end{turn} & \begin{turn}{80}{0}\end{turn} & \begin{turn}{80}{0}\end{turn} & \begin{turn}{80}{0}\end{turn} & \begin{turn}{80}{0}\end{turn} & 0 \\
\hline
2 & \begin{turn}{80}{1}\end{turn} & \begin{turn}{80}{0}\end{turn} & \begin{turn}{80}{-1}\end{turn} & \begin{turn}{80}{0}\end{turn} & \begin{turn}{80}{0}\end{turn} & \begin{turn}{80}{0}\end{turn} & \begin{turn}{80}{0}\end{turn} & \begin{turn}{80}{0}\end{turn} & \begin{turn}{80}{0}\end{turn} & \begin{turn}{80}{0}\end{turn} & \begin{turn}{80}{0}\end{turn} & \begin{turn}{80}{0}\end{turn} & \begin{turn}{80}{0}\end{turn} & \begin{turn}{80}{0}\end{turn} & \begin{turn}{80}{0}\end{turn} & \begin{turn}{80}{0}\end{turn} & \begin{turn}{80}{0}\end{turn} & \begin{turn}{80}{0}\end{turn} & \begin{turn}{80}{0}\end{turn} & \begin{turn}{80}{0}\end{turn} & \begin{turn}{80}{0}\end{turn} & \begin{turn}{80}{0}\end{turn} & \begin{turn}{80}{0}\end{turn} & 2 \\
\hline
\end{tabular}

\end{center}

\vspace{.4cm}
\caption{\small Table of Euler characteristics $\chi_{st}^\pi$ by complexity $t$ and Hodge degree $s$ of $\pi_*\Ebarmn\otimes\BQ$ for $m$ even and $n$ odd.}
\label{table:pi_eo}
\end{table}

\begin{table}[h!]

\begin{center}
\tiny
\begin{tabular}{|c|c|c|c|c|c|c|c|c|c|c|c|c|c|c|c|c|c|c|c|c|c|c|c|}
\hline $t$&\multicolumn{23}{|c|}{Hodge degree $s$}\\
\hline
 & 1 & 2 & 3 & 4 & 5 & 6 & 7 & 8 & 9 & 10 & 11 & 12 & 13 & 14 & 15 & 16 & 17 & 18 & 19 & 20 & 21 & 22 & 23 \\
\hline
1 & \begin{turn}{80}{0}\end{turn} & \begin{turn}{80}{0}\end{turn} & \begin{turn}{80}{0}\end{turn} & \begin{turn}{80}{0}\end{turn} & \begin{turn}{80}{0}\end{turn} & \begin{turn}{80}{0}\end{turn} & \begin{turn}{80}{0}\end{turn} & \begin{turn}{80}{0}\end{turn} & \begin{turn}{80}{0}\end{turn} & \begin{turn}{80}{0}\end{turn} & \begin{turn}{80}{0}\end{turn} & \begin{turn}{80}{0}\end{turn} & \begin{turn}{80}{0}\end{turn} & \begin{turn}{80}{0}\end{turn} & \begin{turn}{80}{0}\end{turn} & \begin{turn}{80}{0}\end{turn} & \begin{turn}{80}{0}\end{turn} & \begin{turn}{80}{0}\end{turn} & \begin{turn}{80}{0}\end{turn} & \begin{turn}{80}{0}\end{turn} & \begin{turn}{80}{0}\end{turn} & \begin{turn}{80}{0}\end{turn} & \begin{turn}{80}{0}\end{turn}  \\
\hline
2 & \begin{turn}{80}{1}\end{turn} & \begin{turn}{80}{0}\end{turn} & \begin{turn}{80}{-1}\end{turn} & \begin{turn}{80}{0}\end{turn} & \begin{turn}{80}{0}\end{turn} & \begin{turn}{80}{0}\end{turn} & \begin{turn}{80}{0}\end{turn} & \begin{turn}{80}{0}\end{turn} & \begin{turn}{80}{0}\end{turn} & \begin{turn}{80}{0}\end{turn} & \begin{turn}{80}{0}\end{turn} & \begin{turn}{80}{0}\end{turn} & \begin{turn}{80}{0}\end{turn} & \begin{turn}{80}{0}\end{turn} & \begin{turn}{80}{0}\end{turn} & \begin{turn}{80}{0}\end{turn} & \begin{turn}{80}{0}\end{turn} & \begin{turn}{80}{0}\end{turn} & \begin{turn}{80}{0}\end{turn} & \begin{turn}{80}{0}\end{turn} & \begin{turn}{80}{0}\end{turn} & \begin{turn}{80}{0}\end{turn} & \begin{turn}{80}{0}\end{turn}  \\
\hline
\end{tabular}
\end{center}

\vspace{.4cm}
\caption{\small Table for Euler characteristics $\chi_{st}$ by complexity $t$ and Hodge degree $s$ of $\HH_*(\Ebarmn;\BQ)$ for  $m$ even and $n$ odd.}
\label{table:H_eo}
\end{table}

\begin{table}[h!]

{\tiny \begin{center}
\begin{tabular}{|c|c|c|c|c|c|c|c|c|c|c|c|c|c|c|c|c|c|c|c|c|c|c|c|c|}
\hline $t$&\multicolumn{23}{|c|}{Hodge degree $s$}&\\
\hline
 & 1 & 2 & 3 & 4 & 5 & 6 & 7 & 8 & 9 & 10 & 11 & 12 & 13 & 14 & 15 & 16 & 17 & 18 & 19 & 20 & 21 & 22 & 23 & total \\
\hline
1 & \begin{turn}{80}{1}\end{turn} & \begin{turn}{80}{-1}\end{turn} & \begin{turn}{80}{0}\end{turn} & \begin{turn}{80}{0}\end{turn} & \begin{turn}{80}{0}\end{turn} & \begin{turn}{80}{0}\end{turn} & \begin{turn}{80}{0}\end{turn} & \begin{turn}{80}{0}\end{turn} & \begin{turn}{80}{0}\end{turn} & \begin{turn}{80}{0}\end{turn} & \begin{turn}{80}{0}\end{turn} & \begin{turn}{80}{0}\end{turn} & \begin{turn}{80}{0}\end{turn} & \begin{turn}{80}{0}\end{turn} & \begin{turn}{80}{0}\end{turn} & \begin{turn}{80}{0}\end{turn} & \begin{turn}{80}{0}\end{turn} & \begin{turn}{80}{0}\end{turn} & \begin{turn}{80}{0}\end{turn} & \begin{turn}{80}{0}\end{turn} & \begin{turn}{80}{0}\end{turn} & \begin{turn}{80}{0}\end{turn} & \begin{turn}{80}{0}\end{turn} & 2 \\
\hline
2 & \begin{turn}{80}{0}\end{turn} & \begin{turn}{80}{0}\end{turn} & \begin{turn}{80}{0}\end{turn} & \begin{turn}{80}{0}\end{turn} & \begin{turn}{80}{0}\end{turn} & \begin{turn}{80}{0}\end{turn} & \begin{turn}{80}{0}\end{turn} & \begin{turn}{80}{0}\end{turn} & \begin{turn}{80}{0}\end{turn} & \begin{turn}{80}{0}\end{turn} & \begin{turn}{80}{0}\end{turn} & \begin{turn}{80}{0}\end{turn} & \begin{turn}{80}{0}\end{turn} & \begin{turn}{80}{0}\end{turn} & \begin{turn}{80}{0}\end{turn} & \begin{turn}{80}{0}\end{turn} & \begin{turn}{80}{0}\end{turn} & \begin{turn}{80}{0}\end{turn} & \begin{turn}{80}{0}\end{turn} & \begin{turn}{80}{0}\end{turn} & \begin{turn}{80}{0}\end{turn} & \begin{turn}{80}{0}\end{turn} & \begin{turn}{80}{0}\end{turn} & 0 \\
\hline

\end{tabular}
\end{center}}
\vspace{.4cm}
\caption{\small Table of Euler characteristics $\chi_{st}^\pi$ by complexity $t$ and Hodge degree $s$ of $\pi_*\Ebarmn\otimes\BQ$ for both $m$  and $n$ even.}
\label{table:pi_ee}
\end{table}

\begin{table}[h!]

\begin{center}
\tiny
\begin{tabular}{|c|c|c|c|c|c|c|c|c|c|c|c|c|c|c|c|c|c|c|c|c|c|c|c|}
\hline $t$&\multicolumn{23}{|c|}{Hodge degree $s$}\\
\hline
 & 1 & 2 & 3 & 4 & 5 & 6 & 7 & 8 & 9 & 10 & 11 & 12 & 13 & 14 & 15 & 16 & 17 & 18 & 19 & 20 & 21 & 22 & 23 \\
\hline
1 & \begin{turn}{80}{1}\end{turn} & \begin{turn}{80}{-1}\end{turn} & \begin{turn}{80}{0}\end{turn} & \begin{turn}{80}{0}\end{turn} & \begin{turn}{80}{0}\end{turn} & \begin{turn}{80}{0}\end{turn} & \begin{turn}{80}{0}\end{turn} & \begin{turn}{80}{0}\end{turn} & \begin{turn}{80}{0}\end{turn} & \begin{turn}{80}{0}\end{turn} & \begin{turn}{80}{0}\end{turn} & \begin{turn}{80}{0}\end{turn} & \begin{turn}{80}{0}\end{turn} & \begin{turn}{80}{0}\end{turn} & \begin{turn}{80}{0}\end{turn} & \begin{turn}{80}{0}\end{turn} & \begin{turn}{80}{0}\end{turn} & \begin{turn}{80}{0}\end{turn} & \begin{turn}{80}{0}\end{turn} & \begin{turn}{80}{0}\end{turn} & \begin{turn}{80}{0}\end{turn} & \begin{turn}{80}{0}\end{turn} & \begin{turn}{80}{0}\end{turn}  \\
\hline
2 & \begin{turn}{80}{0}\end{turn} & \begin{turn}{80}{1}\end{turn} & \begin{turn}{80}{-1}\end{turn} & \begin{turn}{80}{0}\end{turn} & \begin{turn}{80}{0}\end{turn} & \begin{turn}{80}{0}\end{turn} & \begin{turn}{80}{0}\end{turn} & \begin{turn}{80}{0}\end{turn} & \begin{turn}{80}{0}\end{turn} & \begin{turn}{80}{0}\end{turn} & \begin{turn}{80}{0}\end{turn} & \begin{turn}{80}{0}\end{turn} & \begin{turn}{80}{0}\end{turn} & \begin{turn}{80}{0}\end{turn} & \begin{turn}{80}{0}\end{turn} & \begin{turn}{80}{0}\end{turn} & \begin{turn}{80}{0}\end{turn} & \begin{turn}{80}{0}\end{turn} & \begin{turn}{80}{0}\end{turn} & \begin{turn}{80}{0}\end{turn} & \begin{turn}{80}{0}\end{turn} & \begin{turn}{80}{0}\end{turn} & \begin{turn}{80}{0}\end{turn}  \\
\hline
\end{tabular}
\end{center}

\vspace{.4cm}
\caption{\small Table for Euler characteristics $\chi_{st}$ by complexity $t$ and Hodge degree $s$ of $\HH_*(\Ebarmn;\BQ)$ for  both $m$ and $n$ even.}
\label{table:H_ee}
\end{table}


\begin{thebibliography}{99}
%\bibitem{AhearnKuhn} S. Ahearn and N. Kuhn. Product and other fine structure in polynomial resolutions of mapping spaces. {\em Algebr. Geom. Topol.}  2 (2002), 591--647.
\bibitem{Arn} V. Arnol'd. The cohomology ring of the group of colored braids. (Russian) \emph{Mat. Zametki}  5,  1969, pp 227--231.
%\bibitem{Schwartz} G. Arone. A note on the homology of $\Sigma_n$, the Schwartz genus, and solving polynomial equations. {\em An alpine anthology of homotopy theory}, 1--10, Contemp. Math., 399, Amer. Math. Soc., Providence, RI, 2006.
%\bibitem{Arone} G. Arone, Derivatives of embedding functors. I. The stable case. {\em J. Topol.} 2 (2009), no. 3, 461--516.
\bibitem{ALTV} G. Arone, P. Lambrechts, V. Turchin, and I. Voli\'c. Coformality and rational homotopy groups of spaces of long knots. {\em Math. Res. Lett.} 15 (2008), no. 1, 1--14.
\bibitem{ALV} G. Arone, P. Lambrechts, and I. Voli\'c. Calculus of functors, operad formality, and rational homology of embedding spaces. {\em Acta Math.} 199 (2007), no. 2, 153--198.
\bibitem{AT} G. Arone, V. Turchin. On the rational homology of high dimensional analogues of spaces of long knots. arXiv:1105.1576.
\bibitem{BarNatan} Dr. Bar-Natan. On the Vassiliev knot invariants.  \emph{Topology}
 34  (1995),  no. 2, 423--472.
\bibitem{Budney} R. Budney. Little cubes and long knots. {\em Topology} 46 (2007), no. 1, 1--27.

\bibitem{Budney2} R. Budney, A family of embedding spaces. \emph{Groups, homotopy and configuration spaces}, 41--83, \emph{Geom. Topol. Monogr.}, 13, Geom. Topol. Publ., Coventry, 2008.

\bibitem{Catt} A. Cattaneo, P. Cotta-Ramusino, and R. Longoni.
Configuration spaces and Vassiliev classes in any dimension.
\emph{Algeb. Geom. Topol.,} 2:~949--1000, 2002.

\bibitem{CattRossi} A. Cattaneo, C. Rossi.\
Wilson surfaces and higher dimensional knot invariants. \emph{Comm. Math. Phys.} 256 (2005), no.~3, 513-537.

%\bibitem{Ching} M. Ching. Bar constructions for topological operads and the Goodwillie derivatives of the identity. {\em Geom. Topol.} 9 (2005), 833--933.
\bibitem{Coh} F. Cohen. The homology of $C_{n+1}$ spaces. In {\it Lecture Notes in Mathematics}, Vol. 533, 1976.
\bibitem{CT:RCCS} F. R. Cohen, L. R. Taylor. On the representation theory associated to the cohomology of
configuration spaces, In: Algebraic Topology (Oaxtepec, 1991), {\em Contemp. Math.} 146, Amer.
Math. Soc., Providence, RI, 1993, pp. 91--109.

\bibitem{ConVog} J. Conant, F. Gerlits, K. Vogtmann.\
Cut vertices in commutative graphs.\
\emph{Q. J. Math.} 56 (2005), no.~3, pp 321--336.

\bibitem{Dasbach} O. Dasbach, On the combinatorial structure of primitive Vassiliev invariants II,
\emph{J. Combin. Theory Ser. A} 81 (2) (1998) 127--139.


%\bibitem{DwyerSpal} W. G. Dwyer, J. Spalinski. Homotopy theories and model categories. {\em Handbook of algebraic topology,} 73--126, North-Holland, Amsterdam.

\bibitem{Fresse} B. Fresse.\
Koszul duality of operads and homology of partition posets.\
In: Homotopy theory: relations with algebraic geometry, group cohomology, and algebraic $K$-theory, \
\emph{Contemp. Math.}, 346, Amer. Math. Soc. Providence, RI, 2004, pp 115-215.

%\bibitem{Fresse1} B. Fresse.\
%Modules over operads and functors.\
%\emph{Lecture Notes in Mathematics} 1967. \
%Springer-Verlag, Berlin, 2009.

\bibitem{GS} M. Gerstenhaber, S. D. Schack.
A Hodge-type decomposition for commutative algebra cohomology.\
{\em J. Pure Appl. Algebra} 48 (1987), no. 3, 229--247.

\bibitem{GetzJon} E. Getzler, J. D. S. Jones. Operads, homotopy algebra and iterated integrals for double loop spaces. arXiv:hep-th/9403055.

%\bibitem{GinotTradlerZ} G. Ginot, T. Tradler,  M. Zeinalian. Derived Higher Hochschild Homology, Topological Chiral Homology and Factorization algebras. arXiv:1011.6483.

\bibitem{Hirsch} M. Hirsch. Immersions of manifolds. \emph{Trans. Amer. Math. Soc.} 93, 1959, 242--276.


\bibitem{Klyachko} A. A. Klyachko, Lie elements in the tensor algebra,
\emph{Siberian Math. J.} 15 (1974), 914--920.

\bibitem{Kontsevich} M. Kontsevich, Ya. Soibelman, Deformations of algebras over operads and the Deligne conjecture. In
Conf\'erence Mosh\'e Flato 1999, Vol. I (Dijon), volume 21 of Math. Phys. Stud., pages 255--307. Kluwer Acad.
Publ., Dordrecht, 2000.

\bibitem{LT} P. Lambrechts, V. Turchin. Homotopy graph-complex for configuration and knot spaces.
\emph{Trans. Amer. Math. Soc.} 361 (2009), no.~1, 207--222.

\bibitem{LTV} P. Lambrechts, V. Turchin, I. Voli\'c. The rational homology of spaces of long knots in codimension $>2$. \
\emph{Geom. Topol.} 14 (2010), no.~4, 2151--2187.

\bibitem{LV} P. Lambrechts, I. Voli\'c. Formality of the little $N$-disks operad.
\ To appear in \emph{Memoirs of the AMS}.
\ Preprint arXiv:0808.0457.

\bibitem{Lehrer} G. I. Lehrer. Equivariant Cohomology of Configurations in $\BR^d$.\
Special issue dedicated to Klaus Roggenkamp on the occasion of his 60th birthday.\
{\em Algebr. Represent. Theory} 3 (2000), no. 4, 377--384.

\bibitem{LO:ASGCH} G. I. Lehrer, L. Solomon. On the action of the symmetric group on the cohomology of
the complement of its reflecting hyperplanes, {\em J. Algebra} 104 (1986), 410--424.

\bibitem{Loday} J.-L. Loday.
 \ Op\'erations sur l'homologie cyclique des alg\`ebres commutatives.
 \ \emph{Invent. Math.}  96  (1989),  no. 1, 205--230.

\bibitem{LodayVal} J.-L. Loday, B. Vallette. Algebraic Operads.\quad
Grundlehren der Mathematischen Wissenschaften [Fundamental Principles of Mathematical Sciences], 346. Springer, Heidelberg, 2012. xxiv+634 pp.

\bibitem{MerkVal} V. Merkulov, B. Vallette. Deformation theory of representation of prop(erad)s I.\
\emph{J. Reine Angew. Math.}, Issue 634 (2009), pp 51-106.

\bibitem{MoskOhtsuki} D. Moskovich, T. Ohtsuki, Vanishing of 3-loop Jacobi diagrams of odd degree.
\ \emph{J. Combin. Theory Ser. A} 114 (2007), no. 5, 919--930.

\bibitem{PirashviliDold} T. Pirashvili, Hodge decomposition for higher order Hochschild homology. \
\emph{Ann. Sci. Ecole Norm. Sup.} (4) 33 (2000), no. 2, 151--179.

\bibitem{RobinsonWhite} C. A. Robinson, S. Whitehouse, The tree representation of $\Sigma_{n+1}$.
\ \emph{Journal of Pure and Applied Algebra},
1996, 111~(1-3), 245-253.

\bibitem{Salvatore} P. Salvatore. Knots, operads, and double loop spaces.
{\em Int. Math. Res. Not.} 2006, Art. ID 13628, 22 pp.

\bibitem{Sakai} K. Sakai. Configuration space integrals for embedding spaces and the Haefliger invariant.\
\emph{J. Knot Theory Ramifications} 19 (2010), no.~12, 1597-1644.

\bibitem{SakW} K. Sakai, T. Watanabe. \
1-loop graphs and configuration space integral for embedding spaces.
 \ \emph{Math. Proc. Cambridge Philos. Soc.} 152 (2012), no.~3, 497--533.

\bibitem{SevWil} P. Severa, T. Willwacher. Equivalence of formalities of the little discs operad. \
\emph{Duke Math. J.},
\ 160(1):175–-206, 2011.


\bibitem{Sinha-OKS} D. Sinha.
\ Operads and knot spaces.
\ {\em J. Am. Math. Soc.}
\ 19(2):~461--486, 2006.

\bibitem{Sinha:GD} D. Sinha. A pairing between graphs and trees. arXiv:math/0502547.
\bibitem{Sinha-LDO} D. Sinha. The homology of the little discs operad. Preprint arXiv:math/0610236.

\bibitem{Tur-HSLK} V. Turchin (Tourtchine). On the homology of the spaces of long knots. {\em Advances in topological quantum field theory,} 23--52, NATO Sci. Ser. II Math. Phys. Chem., 179, Kluwer Acad. Publ., Dordrecht, 2004.
\bibitem{T-OS} V. Turchin (Tourtchine). On the other side of the bialgebra of chord diagrams.
\ {\it Journal of Knot Theory and its Ramifications}.
\ Vol. 16 (5), May 2007, pp. 575-629.

\bibitem{Turchin} V. Turchin,  Hodge-type decomposition in the homology of long knots. {\em J. Topol.} 3 (2010), no. 3, 487--534.
\bibitem{Vassiliev} V. Vassiliev. Complements of Discriminants of Smooth Maps: Topology and Applications.
Revised ed. Providence, R.I.: AMS, 1994 (\emph{Translation of Mathem. Monographs}, 98).
\bibitem{Wat} T. Watanabe.
Configuration space integral for long $n$-knots and the Alexander polynomial.
\emph{Algebr. Geom. Topol.} 7 (2007), 47--92.

\bibitem{Weibel} Ch. A. Weibel. An introduction to homological algebra.
\ \emph{Cambridge Studies in Advanced Mathematics}, 38. Cambridge University Press, Cambridge, 1994. xiv+450 pp.



%\bibitem{WeissEmb} M. Weiss, Embeddings from the point of view of immersion theory. I. {\em Geom. Topol.} 3 (1999), 67--101

\bibitem{Weiss:HomologyEmb} M. Weiss, Homology of spaces of smooth embeddings.
\ \emph{Q. J. Math.} 55 (2004), no. 4, 499-504.

\bibitem{Whitehouse} S. Whitehouse. Gamma Homology of Commutative Algebras and Some Related Representations
of the Symmetric Group, thesis. Warwick University, 1994.

\bibitem{Willwacher} T. Willwacher. M. Kontsevich's graph complex and the Grothendieck-Teichmueller Lie algebra. Preprint arXiv:1009.1654.

\end{thebibliography}
\end{document}